\documentclass[reqno,english]{amsart}
\usepackage{algorithm,enumerate}
\usepackage{algpseudocode}
\usepackage{comment}

\usepackage{amsmath}
\usepackage[hidelinks]{hyperref}
\usepackage[noabbrev,capitalize,nameinlink]{cleveref}
\usepackage{enumerate,amsmath,amsthm,latexsym,amssymb,color}
\usepackage{graphicx}
\usepackage{enumitem} %new
\usepackage[%texencoding=ascii,
    sorting=nyt,
    maxbibnames=99
]{biblatex}

\def\cS{{\mathcal S}}
\def\cC{{\mathcal C}}

\def\cP{\mathcal{P}}

\newcommand{\set}[1]{\left\{#1\right\}}

\def\cE{\mathcal{E}}

\def\cV{\mathcal{V}}

\def\cC{\mathcal{C}}

\def\ccA{\overline{A}}
\def\ccC{\overline{C}}

\renewcommand{\Pr}{\operatorname{\bf Pr}}
\newcommand\bfrac[2]{\left(\frac{#1}{#2}\right)}

\newtheorem{theorem}{Theorem}[section]

\newtheorem{lemma}[theorem]{Lemma}
\newtheorem{corollary}[theorem]{Corollary}

\newtheorem{claim}[theorem]{Claim}

\newtheorem{definition}[theorem]{Definition}

\newtheorem{question}[theorem]{Question}

\begin{document}

\title{Robust Hamiltonicity in families of Dirac graphs}
\author{Michael Anastos$^\ast$}
\thanks{$\ast$Institute of Science and Technology Austria (ISTA), Klosterneuburg 3400, Austria. Email: {\tt michael.anastos@ist.ac.at}. Supported by the European Union’s Horizon 2020 research and innovation programme under the Marie Sk\l{}odowska-Curie grant agreement No.\ 101034413 \includegraphics[width=4.5mm, height=3mm]{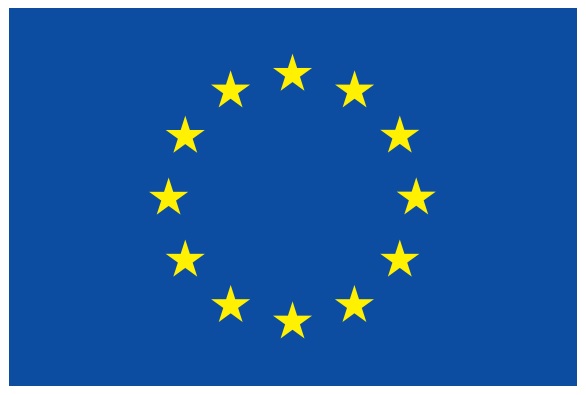}.}
\author{Debsoumya Chakraborti$^\dagger$}
\thanks{$\dagger$Mathematics Institute, University of Warwick, Coventry, CV4 7AL, UK. E-mail: {\tt debsoumya.chakraborti@warwick.ac.uk}. Supported by the Institute for Basic Science (IBS-R029-C1), and the European Research Council (ERC) under the European Union Horizon 2020 research and innovation programme (grant agreement No. 947978).}
\date{}

\begin{abstract}
A graph is called Dirac if its minimum degree is at least half of the number of vertices in it. Joos and Kim showed that every collection $\mathbb{G}=\{G_1,\ldots,G_n\}$  of Dirac graphs on the same vertex set $V$ of size $n$ contains a Hamilton cycle transversal, i.e., a Hamilton cycle $H$ on $V$ with a bijection $\phi:E(H)\rightarrow [n]$ such that $e\in G_{\phi(e)}$ for every $e\in E(H)$. 

In this paper, we determine, up to a multiplicative constant, the threshold for the existence of a Hamilton cycle transversal in a collection of random subgraphs of Dirac graphs in various settings. Our proofs rely on constructing a spread measure on the set of Hamilton cycle transversals of a family of Dirac graphs.

As a corollary, we obtain that every collection of $n$ Dirac graphs on $n$ vertices contains at least $(cn)^{2n}$ different Hamilton cycle transversals $(H,\phi)$ for some absolute constant $c>0$. This is optimal up to the constant $c$. Finally, we show that if $n$ is sufficiently large, then every such collection spans $n/2$ pairwise edge-disjoint Hamilton cycle transversals, and this is best possible. These statements generalize classical counting results of Hamilton cycles in a single Dirac graph.
\end{abstract}

\maketitle

\section{Introduction}

Determining whether a graph contains a Hamilton cycle (i.e., a cycle that uses every vertex of the graph) has been one of the classical and well-known NP-complete problems; see~\cite{karp2010reducibility}. Thus, finding general conditions or properties of graphs ensuring a Hamilton cycle has been one of the fundamental areas of research. Dirac~\cite{dirac1952some} proved that every $n$-vertex graph with minimum degree at least $n/2$ (such graphs will be referred to as \textit{Dirac graphs} throughout this paper) contains a Hamilton cycle. This minimum degree condition is best possible. Since the inception of the result of Dirac, there has been much interest in extending and finding robust versions of this result. 

For example, a natural way to show robustness is to show that any Dirac graph has many Hamilton cycles, not just one. Indeed, S\'{a}rk\"{o}zy, Selkow, and Szemer\'{e}di~\cite{sarkozy2003number} showed that the number of Hamilton cycles in an $n$-vertex Dirac graph is at least $(cn)^n$ for some constant $c>0$. Later, Cuckler and Kahn~\cite{cuckler2009hamiltonian} proved the same with $c = (1-o(1))/2e$, which is easily seen to be best possible by considering the expected number of Hamilton cycles in random graphs with density slightly above half.

Another way to show robustness of Dirac's theorem is to show that a Dirac graph remains Hamiltonian under random edge removals whp\footnote{We say that a sequence of events $\{\mathcal{E}_n\}_{n\ge 1}$ occurs \textit{with high probability} or in short \textit{whp}, if $\lim_{n\to \infty } \Pr(\mathcal{E}_n)=1$.}. Equivalently, we can say that a sufficiently large random subgraph (also referred to as random sparsification) of a Dirac graph contains a Hamilton cycle whp. To discuss this direction, we first briefly mention classical results on Hamiltonicity in random graphs.

\subsection{Hamiltonicity in random graphs}
For $0\le p=p(n) \le 1$, the binomial random graph, denoted by $G(n,p)$, is a random graph on $n$ vertices where each possible edge is present independently with probability~$p$. Building upon the work of Posa~\cite{posa1976hamiltonian}, Korshunov~\cite{korshunov1976solution}  proved that the threshold for the random graph to contain a Hamilton cycle is $\log n/n$. His result was further sharpened independently by Bollob\'as~\cite{bollobas1984evolution}, Koml{\'o}s and Szemer{\'e}di~\cite{komlos1983limit}. They  proved that, near the Hamiltonicity threshold, the limit (as $n$ tends to infinity) of the probability of $G(n,p)$ being Hamiltonian equals the probability of its minimum degree being at least~$2$. 

For a graph $G$ and $0\le p=p(|V(G)|)\le 1$, we denote by $G_p$ the graph generated by taking each edge of $G$ independently with probability $p$. Thus, $G(n,p)=(K_n)_p$ where $K_n$ denotes the complete graph on $n$ vertices. Krivelevich, Lee, and Sudakov~\cite{krivelevich2014robust} prove that if $G$ is an $n$-vertex Dirac graph, then the property that $G$ is Hamiltonian is retained by $G_p$ whp provided that $p\ge c\log n/n$ for some constant $c>0$. Johanson~\cite{Johanson2020} improved this result to $p\ge (2+o(1))\log n/n$ provided that the minimum degree $\delta (G)\ge (1/2+\epsilon)n$ for any constant $\epsilon>0$.
\begin{theorem}[Robust Dirac's theorem~\cite{krivelevich2014robust}] \label{thm:robust dirac}
There exists an absolute constant $c$ such that the following holds. Suppose $G$ is a Dirac graph on $n$ vertices, and $p\ge c\log n/n$. Then, whp the graph $G_p$ contains a Hamilton cycle.
\end{theorem}

\Cref{thm:robust dirac} can be interpreted as follows. Starting with $K_n$, choose an arbitrary spanning subgraph $G$ of $K_n$ with the property that the degree of every vertex is at most halved. Following this, generate $G_p$. We have that if $p\ge c\log n/n$, then $G_p$ is Hamiltonian whp. A natural question is whether one can swap the order of these two operations. This question is often phrased in terms of \textit{resilience} of random graphs. The general study of the resilience of different graph properties in the binomial random graph was initiated by Sudakov and Vu~\cite{sudakov2008local}. We say a graph $G$ is $\alpha$-resilient with respect to the property $\mathcal{P}$ if, for any subgraph $H\subseteq G$ such that $d_{H}(v)\le \alpha d_G(v)$ for every $v\in V(G)$ we have that the graph $G-H$ has the property $\mathcal{P}$. Lee and Sudakov~\cite{lee2012dirac} proved that $G(n,p)$ is whp $(1/2-\epsilon)$-resilient for containing Hamilton cycles provided that $np\ge C\log n$ for some constant $C=C(\epsilon)$ that depends on $\epsilon$. Montgomery~\cite{montgomery2019hamiltonicity} improved this result to $np\ge (1+o(1))\log n$ and also proved an analogous theorem for directed binomial random graphs~\cite{montgomerydirectedresilient}. For further results related to the resilience of graph properties of random graphs, see~\cite{alon2010increasing,balogh2011local,ben2011resilience,dellamonica2008resilience,krivelevich2010resilient,lee2012dirac,vskoric2018local}. For related problems on robust Hamiltonicity, we suggest the readers consult the surveys~\cite{frieze2019hamilton,sudakov2017robustness}. 

There have been some recent developments on robust Hamiltonicity problems in hypergraphs, see~\cite{kelly2024optimal,montgomery2024counting}.

\subsection{Hamilton cycle transversals}

A generalization of a graph embedding problem is to find a transversal of a graph in a given collection of graphs (also referred to as family of graphs) on the same vertex set. For a given family $\mathbb{F}=\set{F_1,\ldots,F_m}$ of sets on a common ground set $\Omega$, a set $X\subseteq \Omega$ is called an $\mathbb{F}$-transversal if $X\cap F_i\neq \emptyset$ for every $i\in [m]$. There have been significant studies on finding transversals of various objects; see, e.g., \cite{aharoni2017rainbow,kalai2005topological,pokrovskiy2020}.
For a given collection $\mathbb{G}=\{G_1,\ldots,G_{m}\}$ of graphs on the vertex set $V$, we define a \textit{$\mathbb{G}$-transversal} to be a pair $(E,\phi)$ where $E$ is a set of edges spanned by $V$ and $\phi:E\to [m]$ is a bijection such that $e\in E(G_{\phi (e)})$ for every $e\in E$. We often consider the edges of $G_i$ to have color $i$ and think of $\phi$ as a \textit{coloring} of the graph $(V,E)$. We say that a coloring $\phi$ of a set of edges $E$ is rainbow if no two edges in $E$ have the same color. Thus, informally speaking, $\mathbb{G}$-transversals are rainbow edge sets of size $|\mathbb{G}|$.  

Aharoni, DeVos, Gonz\'alez Hermosillo de la Maza, Montejano, and \v{S}\'{a}mal \cite{aharoni2020rainbow} studied the existence of a transversal of a triangle in a collection of $3$ graphs, each having more than $\alpha n^2$ edges (i.e. an analog of Mantel's theorem). They proved that there exists a constant $c>1/4$ such that if $\alpha\ge c$, then such a transversal exists. In addition, they provided a construction showing that this constant is tight. The above result demonstrates that classical results do not always carry over to the transversal setting in a trivial way. However, Aharoni \cite{aharoni2020rainbow} conjectured that this is the case for Hamilton cycles. Addressing Aharoni's conjecture, Cheng, Wang, and Zhao~\cite{cheng2021rainbow} established an asymptotic version of the classical Dirac's theorem in the transversal setting. Joos and Kim~\cite{Joos2020} proved that Aharoni's conjecture holds. For further results on the existence of transversals, see~\cite{bowtell2025universality,chakraborti2023bandwidth,chakraborti2024hamilton,chakraborti2024transversal,cheng2023rainbow,cheng2026transversals,falgas2024rainbow,gupta2022general,im2025rainbow,montgomery2022transversal,sun2024transversal}. For a graph family $\mathbb{G}=\{G_1,\ldots,G_n\}$ on a common vertex set $V$, a Hamilton $\mathbb{G}$-transversal is a $\mathbb{G}$-transversal $(E,\phi)$ such that $E$ is the edge set of a Hamilton cycle on $V$.
\begin{theorem} [Transversal Dirac's theorem~\cite{Joos2020}] \label{thm:joos kim}
Let $n\ge 3$. Let $\mathbb{G}=\{G_1,\ldots,G_n\}$ be a collection of $n$ Dirac graphs on the same vertex set of size $n$. Then, there exists a Hamilton $\mathbb{G}$-transversal.  
\end{theorem}
Bradshaw, Halasz, and Stacho~\cite{bradshaw2022} extended the above theorem by showing that any such family actually has at least $(cn/e)^{cn}$ different such transversals for some constant $c\ge 1/68$. In this paper, we improve this result by determining the correct exponent in this count, see \Cref{thm:main_counting}.

On the other hand, Ferber, Han, and Mao~\cite{ferber2022dirac} proved a resilience result for a collection of $n$ graphs on $n$ vertices where each graph is an independent copy of $G(n,p)$ provided that $np=\omega(\log n)$. Namely, they proved the following (in~\cite{ferber2022dirac}, their corresponding theorem is stated to hold for $p=\omega(\log n/n)$ but they actually prove the following stronger result).
\begin{theorem}[\cite{ferber2022dirac}] \label{thm:ferber}
For every $\epsilon >0$, there exists an absolute constant $c=c(\epsilon)$ such that the following holds. Let $\mathbb{G}=\{G_1,\ldots,G_n\}$ be a collection of $n$ independent copies of $G(n,p)$ with $p\ge c\log n/n$ on the vertex set $V$ of size $n$. For every $i\in [n]$, let $H_i\subseteq G_i$ be such that $d_{H_i}(v)\ge (1/2 + \epsilon) d_{G_i}(v)$ for all $v\in V$. Then, whp there exists a Hamilton $\{H_1,\ldots,H_n\}$-transversal.
\end{theorem}
Note that the above theorem implies that if $\mathbb{G}=\{G_1,\ldots,G_n\}$ is a collection of graphs on $n$ vertices of minimum degree $(1/2+\epsilon)n$, and $\mathbb{F}=\{F_1,\ldots,F_n\}$ where $F_i=(G_i)_p$ for some $p\ge c\log n/n$, then whp there is a Hamilton $\mathbb{F}$-transversal. Indeed, this statement can be derived from \Cref{thm:ferber} by taking $H_i= G(n,p)\cap F_i$. The proof of Ferber, Han, and Mao~\cite{ferber2022dirac} is based on an ingenious reduction of the given problem to the $(1/2-\epsilon)$-resilience of Hamiltonicity of random directed graphs problem and then applying the corresponding result of Montgomery~\cite{montgomerydirectedresilient}.

\subsection{Main results} In this paper, we prove a number of results on the existence of a Hamilton cycle transversal in a collection of random subgraphs of Dirac graphs in various settings. We start with the following two extensions of Dirac's theorem in the transversal setting, both of which are strengthenings of \Cref{thm:robust dirac}. For a graph family $\mathbb{G}=\{G_1,\ldots,G_n\}$, we denote the common vertex set by $V(\mathbb{G})$. For a graph family $\mathbb{G}=\{G_1,\ldots,G_n\}$ and a graph $H$ on $V(\mathbb{G})$, we denote by $\mathbb{G}\cap H$ the graph family $\{G_1\cap H,\ldots,G_n\cap H\}$.

\begin{theorem}\label{thm:main_same} 
There exists an absolute constant $c$ such that the following holds. Suppose that $\mathbb{G}=\{G_1,\ldots,G_n\}$ is a collection of Dirac graphs on $n$ vertices, and let $p\ge c\log n/n$. Then, whp there is a Hamilton $\mathbb{G}\cap G(n,p)$-transversal.
\end{theorem}

\begin{theorem}\label{thm:dirac_same}
There exists an absolute constant $c$ such that the following holds. Suppose $G$ is a Dirac graph on $n$ vertices, and $p\ge c\log n/n^2$. Let $\mathbb{G}=\{G_1,\ldots,G_n\}$ be a family of graphs where each $G_i$ is independently distributed as $G_p$ for $i\in [n]$. Then, whp there is a Hamilton $\mathbb{G}$-transversal.
\end{theorem}

In \Cref{thm:main_same}, we intersect each graph of a given family with the same random graph, while in \Cref{thm:dirac_same}, we consider a family that consists of $n$ independent random subgraphs of a Dirac graph. In both theorems, the lower bound on $p$ is best possible up to the value of $c$. This is because in the first case, we need $G(n,p)$ to have minimum degree at least $2$, and in the second one, we need each graph to span at least $1$ edge. At this point, one might hope for a combined generalization of the above two theorems. 

\begin{question} \label{qs:ideal theorem}
Does there exist a constant $c$ such that the following holds? 
Let $\mathbb{G}=\{G_1,\ldots,G_n\}$ be a collection of Dirac graphs on $n$ vertices, and $p\ge c\log n/n^2$. For $i\in [n]$, let $F_i=(G_i)_p$ and let $\mathbb{F}=\{F_1,\ldots,F_n\}$. Then, whp there exists a Hamilton $\mathbb{F}$-transversal.
\end{question}
Unfortunately, this question turns out to be false in general. To see that, assume that $n$ is even, consider a set $A\subseteq V(\mathbb{G})$ of size $n/2$. Then for every $i\in [n-1]$, let $G_i$ be the complete bipartite graph with the partition $A\cup \overline{A}$ (where $\overline{A}$ denotes the set $V(\mathbb{G})\setminus A$). Also, let $G_n$ be the union of two disjoint copies of $K_{n/2}$ induced by $A$ and $\overline{A}$ connected by a perfect matching. For this family, the best one can hope for is $p=\omega(1/n)$. Indeed, every Hamilton $\mathbb{G}$-transversal $(H,\phi)$ spans an edge of $G_i$ from $A$ to $\overline{A}$ for $i\in [n-1]$ (we say that $(H,\phi)$ spans a specific edge $e$ of $G_i$ if $e\in H$ and $\phi(e)=i$). Since every Hamilton cycle on $V(\mathbb{G})$ spans an even number of edges from $A$ to $\overline{A}$, it must be the case that $(H,\phi)$ spans an edge of $G_n$ from $A$ to $\overline{A}$. Every Hamilton $\mathbb{F}$-transversal inherits this property. The condition $p=\omega(1/n)$ is needed to ensure that $F_n\subseteq G_n$ spans an edge $A$ to $\overline{A}$. Notice that there is a parity issue in this particular example. It turns out that this is essentially the only type of issue that can go wrong. 

More generally, for a given graph family $\mathbb{G}=\{G_1,\ldots,G_n\}$ on $n$ vertices, a set of indices $C\subseteq [n]$ of odd size, and a set of vertices $A\subseteq V(\mathbb{G})$, every Hamilton $\mathbb{G}$-transversal $(H,\phi)$ satisfies one of the following. Either $(H,\phi)$ spans an edge of $G_i$ not in\footnote{Here and henceforward, we slightly abuse the notation and think of an edge as both a $2$ element set and a $2$ element ordered set. Thus by $A\times B$, we may denote the set of edges with one endpoint in $A$ and the other in $B$.} $A\times \overline{A}$ for some $i\in C$, or there exists a set of odd size $C'\subseteq \overline{C}$ (hence of size at least 1) so that $(H,\phi)$ spans an edge of $G_i$ in $A\times \overline{A}$ for every $i\in C\cup C'$. Thus, there exists a Hamilton $\mathbb{F}$-transversal only if  $\bigcup_{i\in C} E(F_i)$ has an edge not in $A\times \overline{A}$  or $\bigcup_{i\in \overline{C}} F_i$ has an edge in $A\times \overline{A}$. 

This motivates the following definitions of $e_C(\mathbb{G},A)$ and $r(\mathbb{G})$. For a graph $G$ and $A,B\subseteq V(G)$, let $e_G(A)$ denote the number of edges in the subgraph induced by $A$, and let $e_G(A,B)$ be the number of ordered pairs $(a,b)\in A\times B$ such that $ab$ is an edge of $G$. We call a set $A\subseteq V(G)$ \textit{a half-set} if $(|V(G)|-1)/2\le |A|\le (|V(G)|+1)/2$. For a given graph family $\mathbb{G}=\{G_1,\ldots,G_n\}$ on $n$ vertices, a set of indices $C\subseteq [n]$, and a set of vertices $A\subseteq V(\mathbb{G})$, we define
\begin{align*}
    e_C(\mathbb{G},A)&:= \sum_{i\in C} \left( e_{G_i}(A)+e_{G_i}(\ccA) \right) + \sum_{i\in \ccC} e_{G_i}(A,\ccA). \\
    r(\mathbb{G})&:=\min_{\substack{A\subseteq V(\mathbb{G}),C\subseteq [n] \\ A \; \text{is half-set} \\ |C| \; \text{is odd}}} e_C(\mathbb{G},A).
\end{align*}
Also, let $A(\mathbb{G})$ and $C(\mathbb{G})$ be such that $A(\mathbb{G})$ is a half-set, $|C(\mathbb{G})|$ is odd, and $ r(\mathbb{G})= e_{C(\mathbb{G})}(\mathbb{G},A(\mathbb{G}))$. In case of multiple options for $A(\mathbb{G})$ and $C(\mathbb{G})$, choose one of them arbitrarily. 

In the definition of $r(\mathbb{G})$, we could have minimized over all subsets $A$ of $V(\mathbb{G})$ instead of just all half-subsets. However, as seen shortly, the value of $r(\mathbb{G})$ comes into play only when it is small, say of the order $O(n^2/\log n)$, and such a value can be attained only by sets $A$ of size $n/2$. Thus, we choose to minimize over half-sets $A$. 

We remark that $r(\mathbb{G})\ge n/2$. Furthermore, if $n\ge 3$ is odd, then $r(\mathbb{G})\ge n^2/4$. Indeed, for a fixed half-set $A\subseteq V(\mathbb{G})$ and $C\subseteq [n]$ the following hold (since $e_C(\mathbb{G},A)=e_C(\mathbb{G},\ccA)$ we may assume that $|A| < n/2 < |\ccA|$). For $i\in [n]$, every vertex in $G_i$ has a neighbor in $\ccA$. Therefore, $e_{G_i}(\ccA) \ge |\ccA|/2\ge n/4$ and  $e_{G_i}(A,\ccA)\ge |A|\ge n/4$. It follows that $e_C(\mathbb{G},A)\ge n^2/4$, and, hence, that $r(\mathbb{G})\ge n^2/4$. This observation, together with the following result, implies that the statement in \Cref{qs:ideal theorem} holds when $n$ is odd. 
\begin{theorem}\label{thm:main_distinct} 
There exists a constant $c$ such that the following holds. Let $\mathbb{G}=\{G_1,\ldots,G_n\}$ be a collection of Dirac graphs on $n$ vertices, and $p\ge c\log n/n^2$. Let $F_i=(G_i)_p$ for $i\in [n]$, and let $\mathbb{F}=\{F_1,\ldots,F_n\}$. Then,
$$\lim_{n\to \infty}\Pr(\text{there exists a Hamilton  $\mathbb{F}$-transversal}) =1-\lim_{n\to \infty }e^{-p\cdot r(\mathbb{G})}.$$
\end{theorem}
Thus, if $r(\mathbb{G})=\omega(n^2/\log n)$ and $p\ge c\log n/n^2$ then there exists a Hamilton cycle which is $\mathbb{F}$-transversal whp. Note that one needs $p = \Omega(\log n/n^2)$ to ensure that none of $F_i$ is empty whp. On the other hand, if $r(\mathbb{G})=O(n^2/\log n)$, then $ 1-\lim_{n\to \infty }e^{-p\cdot r(\mathbb{G})}$ equals the probability that at least one of the edges that contribute to $e_C(\mathbb{G},A)$ for the pair $A,C$ that minimizes $e_C(\mathbb{G},A)$ is spanned by $\mathbb{F}$. As demonstrated by the example given earlier, this is a necessary condition for the existence of a Hamilton  $\mathbb{F}$-transversal. As $r(\mathbb{G})\ge n/2$ in general, we have that there is a Hamilton $\mathbb{F}$-transversal whp whenever $p=\omega(1/n)$. 

\Cref{thm:main_distinct} is stronger than what one can directly deduce from \Cref{thm:ferber}. \Cref{thm:dirac_same} follows from \Cref{thm:main_distinct} by taking $G=G_1=\dots=G_n$. 

\subsection{Spread measures and thresholds} Our proofs rely on constructing a spread measure (see \Cref{def:spreadmu}) on the set of transversal Hamilton cycles and then applying a result of Frankston, Kahn, Narayanan, and Park~\cite{frankston2021thresholds}. This specific result was used to prove a fractional version of the Kahn-Kalai expectation threshold vs. threshold conjecture~\cite{kahn2007thresholds}. This conjecture was later resolved by Park and Pham~\cite{pp2022}.

\begin{definition}\label{def:spreadmu}
Let $Z$ be a finite ground set and $\mathcal{H}\subseteq \mathcal{P}(Z)$ a nonempty collection of subsets of $Z$. Let $\mu$ be a probability measure on $\mathcal{H}$. For $q>0$, we say that $\mu$ is $q$-spread if for every $S\subseteq Z$, we have
$$\mu(\{H\in \mathcal{H}:S\subseteq H\}) \le q^{|S|}.$$
\end{definition}
For a finite set $Z$ and $0\le p\le 1$, we denote by $Z(p)$ the random subset of $Z$ where each element of $Z$ is present with probability $p$ independently.

\begin{theorem}[Threshold result; see Theorem~1.6 in ~\cite{frankston2021thresholds}] \label{thm:thresholds}
There exists a constant $C$ such that the following holds. Consider a non-empty ground set $Z$ and let $\mathcal{H}\subseteq \mathcal{P}(Z)$. Suppose that there exists a $q$-spread probability measure on $\mathcal{H}$.  If $p\ge \min\{1,Cq\log |Z|\}$, then $Z(p)$ contains an element of $\mathcal{H}$ as a subset with probability tending to $1$ as $|Z|\to \infty$. 
\end{theorem}

We say that a graph family $\mathbb{G}=\{G_1,\ldots,G_n\}$ on $n$ vertices is \textit{exceptional} if $r(\mathbb{G})\le 0.1 n^2$. We remark that we could use any sufficiently small positive constant instead of $0.1$ in this definition.

\begin{theorem}\label{thm:main_spread}
There exists a constant $c$ such that the following holds. For every non-exceptional collection $\mathbb{G}=\{G_1,\ldots,G_n\}$ of Dirac graphs on $n$ vertices, there is a $(c/n^2)$-spread probability measure on the set of Hamilton $\mathbb{G}$-transversals.
\end{theorem}
To prove \Cref{thm:main_spread}, we adapt the iterative absorption technique introduced by K\"{u}hn and Osthus, and Knox, K\"{u}hn, and Osthus for packing edge-disjoint Hamilton cycles in graphs \cite{knox2015edge, kuhn2013hamilton}. The combination of the threshold result and the iterative absorption technique has already been used to prove a plethora of results, including determining (up to a multiplicative constant) the threshold for the existence of a Steiner triple system and the threshold for the existence of a Latin square~\cite{jain2022optimal,  kang2023thresholds, kang2022perfect, keevash2022optimal, pham2022toolkit, sah2023threshold}.

Note that if $\mathbb{G}=\{G_1,\ldots,G_n\}$ is exceptional, then $A(\mathbb{G})$ has size exactly $n/2$ (because $r(\mathbb{G})\ge n^2/4$ when $n$ is odd). In this case, we have the following theorem in place of \Cref{thm:main_spread}.

\begin{theorem}\label{thm:main_spread_exceptional}
There exists a constant $c$ such that the following holds. Let $\mathbb{G}=\{G_1,\ldots,G_n\}$ be an exceptional collection of Dirac graphs on $n$ vertices, and let $A=A(\mathbb{G})$ and $C=C(\mathbb{G})$. Let $(e,i)$ be such that $e\in E(G_i)$  and either $i\in C$ and $e\notin A\times \overline{A}$ or $i\in \overline{C}$ and $e\in A\times \overline{A}$. Also let $\mathbb{G}^{-i}=\{G_1,\ldots, G_{i-1},G_{i+1},\ldots,G_n\}$. Then, there exists a $(c/n^2)$-spread probability measure on the set of $\mathbb{G}^{-i}$-transversals $(E,\phi)$ where $E$ spans a Hamilton path on $V(\mathbb{G})$ that joins the endpoints of $e$.
\end{theorem}

We now derive \Cref{thm:main_distinct,thm:main_same} from \Cref{thm:main_spread,thm:main_spread_exceptional}.

\begin{proof}[{\bf Proof of \Cref{thm:main_distinct}}] We let $c$ be sufficiently large relative to the constants obtained from \Cref{thm:main_spread,thm:main_spread_exceptional}. If $\mathbb{G}$ is not exceptional, then \Cref{thm:main_distinct} follows from \Cref{thm:thresholds,thm:main_spread}. In the case that $\mathbb{G}$ is exceptional, then let $A,C$ be as in the statement of \Cref{thm:main_spread_exceptional} and first reveal the edges in the sets $E_1=\bigcup_{i\in C} \big(E(F_i)\setminus (A\times \overline{A})\big)$ and $E_2= \bigcup_{i\in \overline{C}} \big(E(F_i)\cap (A\times \overline{A})\big)$. In the event that $E_1\cup E_2$ is empty, as discussed earlier before the definition of $r(\mathbb{G})$, there is no Hamilton $\mathbb{F}$-transversal. Thus, assume that $E_1\cup E_2$ is non-empty. Then, there exists $i\in C$ and $e\in E(F_i)$  such that $e\in E_1$ or $i\in \overline{C}$ and $e\in E(F_i)$  such that $e\in E_2$. In this case, \Cref{thm:thresholds,thm:main_spread_exceptional} imply that whp there exists a $\mathbb{F}^{-i}$-transversal whose edge set spans a Hamilton path $P$ which joins the endpoints of $e\in F_i$, where $\mathbb{F}^{-i}:=\{F_1,\ldots,F_{i-1},F_{i+1},\ldots,F_n\}$. In such an event, $P+e$ with the corresponding coloring gives a Hamilton $\mathbb{F}$-transversal, as desired. It follows that 
$$\Pr(\text{there exists a Hamilton  $\mathbb{F}$-transversal})= 1-\Pr(E_1\cup E_2=\emptyset) - o(1) =1-e^{-p\cdot r(\mathbb{G})} - o(1).$$
\end{proof}

\begin{proof}[{\bf Proof of \Cref{thm:main_same}}] 
Similar to the previous proof, if $\mathbb{G}$ is not exceptional, then \Cref{thm:main_same} follows from \Cref{thm:thresholds,thm:main_spread}. In the case that $\mathbb{G}$ is exceptional, then let $A,C$ be as in the statement of \Cref{thm:main_spread_exceptional}. Then, as $\mathbb{G}$ is exceptional, $n$ is even and $|C|$ is odd. In particular, there exists an index $i\notin C$ such that  $|E(G_i)\cap (A\times \overline{A})|\ge n/2$. Finally, let $H_1, H_2$ be two independent random graphs, both distributed as $G(n,p/2)$. Then $H_1,H_2$ and $G(n,p)$ can be coupled such that $H_1\cup H_2\subseteq G(n,p)$. Now first, reveal the edges in $H_1$. Then, whp there exists an edge $e\in E(G_i)\cap (A\times \overline{A})$ that belongs to $H_1$. After this, reveal the edges of $H_2$. In doing so, \Cref{thm:thresholds,thm:main_spread_exceptional} imply that whp there exists a $\mathbb{G}^{-i}\cap H_2$-transversal whose edge set spans a Hamilton path $P$ which joins the endpoints of $e$, where $\mathbb{G}^{-i}:=\{G_1,\ldots,G_{i-1},G_{i+1},\ldots,G_n\}$. In such an event, $P+e$ with the corresponding coloring gives a Hamilton $\mathbb{G}\cap G(n,p)$-transversal, as desired.
\end{proof}

We next show how the spread results of this subsection (i.e., \Cref{thm:main_spread,thm:main_spread_exceptional}) can be used to prove other robust Hamiltonicity results. For example, \Cref{thm:main_spread,thm:main_spread_exceptional} imply the following counting result that simultaneously extends \Cref{thm:joos kim} and the counting result on Hamilton cycles in a Dirac graph.
\begin{corollary}\label{thm:main_counting}
There exists an absolute constant $C>0$ such that the following holds for every $n\ge 3$. If $\mathbb{G}=\{G_1,\ldots,G_n\}$ is a collection of Dirac graphs on a vertex set of size $n$, then there are $(Cn)^{2n}$ different Hamilton $\mathbb{G}$-transversals.  
\end{corollary}
This is best possible up to the constant $C$ as there may be at most $(n!)^2<n^{2n}$ many Hamilton $\mathbb{G}$-transversals.

\begin{proof}[{\bf Proof of \Cref{thm:main_counting}}]
\Cref{thm:main_spread,thm:main_spread_exceptional} give that there exists a probability measure $\mu$ on the set of Hamilton transversals that assigns to any fixed such cycle a value of at most $(c/n^2)^{n-1}$ for some constant $c$. Thus, there exist at least $(n^2/c)^{n-1}\ge (Cn)^{2n}$ many Hamilton $\mathbb{G}$-transversals for some constant $C>0$. 
\end{proof}

Another way to extend Dirac's theorem in a robust way is to show that a Dirac graph contains many edge-disjoint Hamilton cycles. Nash-Williams \cite{nash1970} showed that this number is linear. His result was finally improved to $\lfloor \frac{n-2}{8} \rfloor$ by Csaba, K{\"u}hn, Lo, Osthus, and Treglown~\cite{csaba2016} for sufficiently large $n$. This bound of $\lfloor \frac{n-2}{8} \rfloor$ is tight, and this follows from a construction of Nash-Williams of a Dirac graph with no $r$-regular subgraph for $r>\frac{n-2}{4}$. 
Extending the arguments given in \Cref{thm:main_spread,thm:main_spread_exceptional}, one can also prove the following robust version of \Cref{thm:joos kim}.

\begin{theorem}\label{thm:main_distinct_packing} 
There exists a constant $\zeta>0$ such that the following holds. Let $\mathbb{G}=\{G_1,\ldots,G_n\}$ be a collection of Dirac graphs on $n$ vertices. Then there exist at least $\min\{\zeta n^2,r(\mathbb{G})\}$ pairwise edge-disjoint Hamilton $\mathbb{G}$-transversals. 
\end{theorem}
In particular, if $n$ is sufficiently large ($n > 1/(2\zeta)$), then there exist $n/2$ pairwise edge-disjoint Hamilton $\mathbb{G}$-transversals. This bound is optimal by considering the extremal example preceding the definition of $r(\mathbb{G})$.

\subsection{Organization}
The remainder of this paper is organized as follows. In the next subsection, we introduce some notation used throughout this paper. In \Cref{sec:outline}, we outline a brief proof sketch of our main results. In \Cref{section:vortex}, we develop some preliminary tools (including vortex and absorber) needed to prove our main results. The subsequent two sections contain the main proofs (i.e., the proofs of \Cref{thm:main_spread,thm:main_spread_exceptional}), and they are separated to deal with the cases depending on whether the family is close to an extremal family or not (see~\Cref{def:extremal family} for the definition). We prove \Cref{thm:main_distinct_packing} in parallel to \Cref{thm:main_spread,thm:main_spread_exceptional}. Finally, \Cref{section:property preserve} contains the proof of a technical lemma from \Cref{sec:outline}.

\subsection{Notation}
We write $[n]$ to denote the set $\set{1,\ldots,n}$ and write $[m,n]$ to denote the set $\set{m,m+1,\ldots,n}$. For a set $X$, we denote by $\binom{X}{2}$ the collection of all subsets of $X$ of size~$2$. For two sets $X,Y$, their symmetric difference is denoted by $X\triangle Y := (X\cup Y)\setminus (X\cap Y)$. For three numbers $a,b,c$, we write $a=b\pm c$ to mean that $b-c\le a \le b+c$. We use the usual hierarchy notation $\alpha \ll \beta$ to mean that there is a non-decreasing function ${f:(0,1]\rightarrow (0,1]}$ such that $\alpha \le f(\beta)$. When there are more than two constants in a hierarchy, they are chosen from right to left. We omit the rounding signs since they are not crucial for our arguments.

We use standard graph theoretic notations. Consider a graph $G$. We denote the vertex set of $G$ by $V(G)$ and the edge set of $G$ by $E(G)$. For a vertex $v\in V(G)$, we denote the neighborhood of $v$ by $N_G(v)$ and the degree of $v$ by $d_G(v)$. We denote by $\delta(G)$ the minimum degree of $G$. For a vertex $v\in V(G)$ and a set $U\subseteq V(G)$, we denote by $d_G(v,U)$ the number of neighbors of $v$ in the set $U$. For a pair of vertices $u,v\in V(G)$, their co-degree $d_G(u,v)$ denotes the size of their common neighborhood $(N_G(u)\cap N_G(v))\setminus \{u,v\}$. For $U\subseteq V(G)$, we write $N_G(U) = \cup_{u\in U} N_G(u)$. We often drop the subscripts when the graph $G$ is clear from the context. For $A\subseteq V(G)$, we denote by $G[A]$ the subgraph induced by $A$. For a set $V$ of vertices (and a set $E$ of edges), we say $G$ \textit{covers} $V$ if $V\subseteq V(G)$ (and $G$ \textit{covers} $E$ if $E\subseteq E(G)$). For an edge $e\in \binom{V(G)}{2}$, we denote by $G+e$ the graph induced by the edge set $E(G)\cup \{e\}$. For two graphs $G_1$ and $G_2$ on the same vertex set, we denote by $G_1\cup G_2$ (and $G_1\cap G_2$) the graph induced by the edge set $E(G_1)\cup E(G_2)$ (and $E(G_1)\cap E(G_2)$). 

Consider a graph family $\mathbb{G}$ indexed by the elements of a set $C$ on a common vertex set $V$. For sets $S\subseteq V$, $T\subseteq C$, we denote by $\mathbb{G}[S]$ the graph family $\{G_i[S]:i\in C\}$, by $\mathbb{G}_T$ the graph family $\{G_i:i\in C\}$, and by $\mathbb{G}_T[S]$ the graph family $\{G_i[S]:i\in T\}$. We also say that $\mathbb{G}_T[S]$ is induced by the pair $(S,T)$. For $T\subseteq C$, an edge-colored graph $H$ on $V$ is called \textit{$T$-rainbow} if the edges of $H$ receive distinct colors from $T$. An edge-colored graph $H$ is called \textit{rainbow} if $H$ is $C$-rainbow. Recall that for a set $X$, we let $X(p)$ be a random subset of $X$ where each element belongs to $X(p)$ with probability $p$ independently. For $0\le p=p(|V|)\le 1$, and $0\le q=q(|C|)\le 1$, we denote by $\mathbb{G}_{C(q)}[V(p)]$ the random graph family that is induced by the (random) pair $(V(p),C(q))$.

%%%%%%%%%%%%%%%%%%%%%%%%%%%%%%%%%%%%%%%%%%%%%%%%%%%%%%%%
\section{Proof Outline and a robust Dirac's theorem} \label{sec:outline}
As mentioned in the introduction, our proof of \Cref{thm:main_spread} adapts the method of iterative absorption to construct an $O(1/n^2)$-spread probability measure on the set of Hamilton transversals. To demonstrate it, first, consider the case where one wants to build an $O(1/n)$-spread probability measure on the set of Hamilton cycles of an $n$-vertex graph $G$ with minimum degree $(1/2 +\eta)n$ for some $\eta>0$.

Let $0 < 1/n \ll 1/L \ll \gamma \ll \beta \ll \eta$. Randomly partition the vertex set $V(G)$ into a sequence of sets $V_1,\ldots,V_N$ with $|V_{i+1}|\approx \beta |V_i|$ for each $i\in [N-1]$, where $N$ is chosen so that $|V_N| \in [L,2\beta^{-1} L]$. Standard concentration inequalities yield that with probability at least $9/10$, every vertex $v\in V_{i-1}\cup V_i$ has at least $(1/2 + \eta/2)|V_i|$ neighbors in $V_i$ for each $i\in [N]$. Using this minimum degree condition, at step $i<N$, starting with a (potentially empty) matching $M_i$ on $V_i$ of size at most $\gamma |V_i|$, cover $V_i$ by a set of paths $\cP_i$ of size at most $\gamma \beta |V_i|\lesssim \gamma |V_{i+1}|$. Here, we allow paths of length $0$, corresponding to vertices. $\cP_i$ has the property that its paths cover all the edges in $M_i$. Then, sequentially match every endpoint of each path in $\cP_i$ to a random vertex in $V_{i+1}$ (we match the vertices that correspond to paths of length 0 twice) and let $M_i'$ be the resulting matching. Augment the paths in $\cP_i$ using the edges in $M_i'$. Finally, construct the matching $M_{i+1}$ on $V_{i+1}$ by placing for each path $P\in \cP_i$ an edge $e_P$ in $M_{i+1}$ that joins the endpoints of $P$.  Note that the edges in $M_{i+1}$ do not need to belong to $G$ and should be considered as ``dummy" edges. In the last step, given a matching $M_N$ on $V_N$ of size at most $\gamma |V_N|$, find a Hamilton cycle $H_N$ that covers $M_N$. Now observe that $H_N$ corresponds to a Hamilton cycle $H$ of $G$. 

For finding the Hamilton cycle $H_N$, we appeal to the following robust version of Dirac's theorem. Its proof is located in \Cref{subsec:robustDirac}. Note that in the robust Dirac's theorem, we have substituted the traditional $n/2$ minimum degree requirement with a slightly weaker minimum degree and an appropriate expansion requirement (which is satisfied by graphs with minimum degree $(1/2 + \eta)n$ with $\eta > 0$). 

\begin{definition}[Extremal graph] \label{def:extremal graph}
Let $\alpha>0$. We say that an $n$-vertex graph $G$ is $\alpha$-extremal if there exists a half-set $A\subseteq V(G)$ such that 
\begin{equation*}
\min\set{e(A), e(A,\overline{A})} \le \alpha n^2.
\end{equation*}
If the above inequality holds for some half-set $A$, we often say $G$ is $\alpha$-extremal with respect to $A$.
\end{definition}

\begin{theorem}[Robust Dirac's theorem] \label{thm:robustdirac}
Let $0<1/n \ll \epsilon,\zeta \ll \alpha \le 1$. Let $G$ be a non $\alpha$-extremal $n$-vertex graph of minimum degree at least $(1/2 -\epsilon)n$. Let $M$ be a matching of size at most $\zeta n$ on $V(G)$. Then, there exists a Hamilton cycle in $G\cup M$ that contains every edge in $M$.
\end{theorem}

To construct $\cP_i$, we consider the auxiliary bipartite graph $B_i$ on $X\cup Y$ where each of $X,Y$ is a copy of $V_i$ and for $x\in X$, $y\in Y$ the edge $xy$ belongs to $B_i$ if and only if it belongs to $G$. Then, for every edge in $M_i$ we remove one of its endpoints from $X$ and the other from $Y$. We then generate the random subgraph $(B_i)_p$ of $B_i$ by retaining each of its edges with probability $p=O(1/|V_i|)$. We continue and construct a large matching $M_i''$ from $(B_i)_p$. Note that every vertex is incident to at most $2$ edges in $M_i\cup M_i''$. Thus, $M_i\cup M_i''$ induces a set of cycles and paths in $G$ (where the number of cycles and paths is small whp). Removing an edge, not in $M_i$, from each cycle results in the desired set of paths $\cP_i$.

We let $\mathcal{D}$ be the probability measure that assigns to each Hamilton cycle $H$ of $G$ probability proportional to the probability that $H$ is output by the above procedure. It turns out that $\mathcal{D}$ is $O(1/n)$-spread. For example, a fixed edge $e$ belongs to a Hamilton cycle sampled according to $\mathcal{D}$ only if (i) $e$ is spanned by some $V_i$, $i<N$ and belongs to $(B_i)_p$ i.e. it survives the sparsification or, (ii) $e$ is spanned by some $V_i\times V_{i+1}$ and is chosen by its $V_i$-endpoint to be added to $M_i'$, or (iii) $e$ is spanned by $V_N$. Now, $e$ is spanned by some $V_i$, $i\le N$ (respectively belongs to $V_i\times V_{i+1}$) with probability at most $|V_i|^2/n^2$ (respectively $|V_i||V_{i+1}|/n^2\le |V_i|^2/n^2$). Then, conditioned on $e$ being spanned by $V_i$, we have that $e$ survives the sparsification with probability $O(1/|V_i|)$. This follows from the observation that, by symmetry, for $v\in V_i$ each of the at least $|V_i|/2$ edges that is incident to $v$ and is spanned by $V_i$ survives the sparsification with the same probability. On the other hand, conditioned on $e\in V_i\times V_{i+1}$, since the $V_{i}$-endpoint of $e$ has at least $|V_{i+1}|/2$ neighbors in $V_{i+1}$ we have that $e$ belongs to $M_i'$ with probability $O(1/|V_{i+1}|)=O(1/|V_{i}|)$. It follows that $e$ belongs to $H$ with probability 
\begin{equation}\label{eq:toy_spread}
\sum_{i=1}^{N-1} \frac{|V_i|^2}{n^2} \cdot O\bfrac{1}{|V_i|}+ \sum_{i=1}^{N-1} \frac{|V_i|^2}{n^2} \cdot O\bfrac{1}{|V_i|} +\frac{|V_N|^2}{n^2}=O\bfrac{1}{n}+\frac{|V_N|^2}{n^2}=O\bfrac{1}{n}.
\end{equation}

To prove a simpler version of \Cref{thm:main_spread} where each of the graphs in $\mathbb{G}$ has minimum degree at least $(1/2+\eta)n$, one has to modify the above argument in a number of ways to ensure that the Hamilton cycle $H$ is $\mathbb{G}$-transversal. The first is to randomly partition the set of colors into sets $C_1,\dots,C_N$. For $i<N$ the size of the color set $C_i$ will be about the size of the vertex set $V_i$ while $|C_N|\gg |V_N|$. Now, the first step, as seen shortly, will be slightly different from the next ones. At step $i \in \{2,\dots, N-1\}$, the edges spanned by $\cP_i$ that do not belong to $M_i$ will be $C_i$-rainbow (and there will be $|C_i|$ such edges). It will turn out that during these steps, we have a lot of room for selecting the corresponding colored edges. At the last step, we will ensure that $E(H_N)\setminus M_N$ is $C_N$-rainbow. To do so, since $C_N$ contains way more than $|V_N|$ colors, we can focus on finding the set $E(H_N)\setminus M_N$ among the set of edges that can be colored with at least $|V_N|$ colors and then assigning colors to these edges in a greedy way. 

The last part is to ensure that all the colors in $C_1 \cup C_N$ are used on $H$. To do so, we start the very first step of the procedure by building an absorber  $(M_{abs}, C_{abs})$. $M_{abs}$ will be a matching on $V_1$ and $C_{abs}$ will be a subset of $C_1$. At this point, it should be thought that the edges in $M_{abs}$ will appear on the final Hamilton cycle $H$. However, we still have not decided what color they will be assigned. Instead, we ensure that we have some flexibility on how we can color them at the very end. Namely, the absorber $(M_{abs}, C_{abs})$ will have the property that for every subset $C_{\text{leftover}}$ of $C_N$ of colors that are not used in the last step of the procedure, the set $M_{abs}$ can be colored with the colors $C_{abs}\cup C_{\text{leftover}}$. At the first step of the procedure after building the absorber $(M_{abs}, C_{abs})$ we set $M_1=M_{abs}$ and continue by finding a set of paths $\cP_1$ so that the set of edge of $\cP_1$ not spanned by $M_1$ is $C_1\setminus C_{abs}$-rainbow.  Now, at the end of our procedure, given the colors that appear on $E(H_N)\setminus M_N$, we use the absorbing property of $M_1=M_{abs}$, $C_{abs}$ and color $M_1$ with the colors in $C_{abs}\cup C_N$ that are not used at step $N$. This ensures that $H$ is rainbow. 

To prove \Cref{thm:main_spread}, we use a stability approach. It is well known that if $G$ is an $n$-vertex Dirac graph, then for every pair of half-sets $A,B\subseteq V(G)$ there are $\Omega(n^2)$ edges from $A$ to $B$ or $G$ is close (in edit distance) to one of the two extremal graphs. These are (i) the union of two cliques on $n/2$ vertices and (ii) the balanced complete bipartite graph. In these two cases, we have that either $e(A)$ or $e(A,\overline{A})$ is small for some half-set $A$ (also see \Cref{def:extremal graph}). To do such arguments in proving \Cref{thm:main_spread}, we need a generalization of this characterization as given by the $\alpha$-extremal graph families defined below. 

\begin{definition}[Extremal family of graphs] \label{def:extremal family}
Let $\alpha>0$. We say that a graph family $\mathbb{G}=\set{G_1,\ldots,G_m}$ on a vertex set $V$ consisting of $n$ vertices is $\alpha$-extremal if there exists a half-set $A\subseteq V$ such that 
\begin{equation}
\sum_{i\in [m]} \min\set{e_{G_i}(A), e_{G_i}(A,\overline{A})} \le \alpha mn^2.
\end{equation}
If the above inequality holds for some half-set $A$, we often say $\mathbb{G}$ is $\alpha$-extremal with respect to $A$.    
\end{definition}

The argument presented earlier for dealing with graph families of minimum degree $(1/2+\eta)n$ can be used verbatim to deal with non $\alpha$-extremal graph families for $\alpha>0$. The reason is that the property of being/ or not being $\alpha$-extremal is preserved (up to multiplicative rescaling) under random subsampling of the vertex set and the color set of the family whp (see \Cref{lem:property preserve}). In the event $\mathbb{G}_{C_N}[V_N]$ is non-extremal, we can use the robust Dirac theorem to deal with the last step where we construct $E(H_N)\setminus M_N$.

We often work with edge-minimal Dirac graphs, which gives us more flexibility in making certain arguments. A Dirac graph $G$ is called \textit{edge-minimal} if, for every edge $e\in E(G)$, the graph obtained by removal of $e$ from $G$ is not a Dirac graph. The following lemma will be proved at the end of this paper in \Cref{section:property preserve}.

\begin{lemma}[Non-extremality preserving lemma for family] \label{lem:property preserve}
Let $0 < 1/n,1/m  \ll \alpha' \ll \alpha \le 1$. Let $p, q\in [0,1]$. Let $\mathbb{G}$ be a not $\alpha$-extremal family of $m$ edge-minimal Dirac graphs on a common vertex set $V$ of size $n$. Denote the index set of the graphs in $\mathbb{G}$ by $C$. Then, with probability at least $1-e^{-\Omega(\alpha' qm )} - \alpha (qm)(pn)^2 e^{-\Omega(\alpha' pn)}$, the random graph family $\mathbb{G}_{C(q)}[V(p)]$ is not $\alpha'$-extremal.
\end{lemma}

It thus remains to describe the proof of \cref{thm:main_spread} in the case that $\mathbb{G}$ is $\alpha$-extremal for some small $\alpha$. It turns out that after some clean-up, this case boils down to finding an $O(1/n^2)$-spread probability measure on the set of rainbow Hamilton paths in a family of almost complete bipartite graphs of large minimum degree that share the same bipartition. By leveraging the additional structure, we prove this last statement using more traditional means. The proof of \Cref{thm:main_spread_exceptional} follows in a similar manner. For the detailed proofs, see \Cref{section:extremal families}.

%%%%%%%%%%%%%%%%%%%%%%%%%%%%%%%%%%%%%%%
\subsection{Proof of robust Dirac's theorem} \label{subsec:robustDirac}
To prove \Cref{thm:robustdirac}, we use the following property of non-extremal graphs. This is a simpler version of \cite[Lemma~2.1]{krivelevich2014robust} with a slightly more general minimum degree condition. Similar lemmas for Dirac graphs also appeared in \cite{cuckler2009hamiltonian,komlos1998proof, krivelevich2014robust, sarkozy2008distributing}

\begin{lemma} \label{lem:KLS general} 
Let $0< 1/n \ll \epsilon \ll \alpha \le 1$. Let $G$ be an $n$-vertex graph with minimum degree at least $(\frac{1}{2}-\epsilon)n$. If $A$ is a half-set such that $\min\set{e(A), e(A,\overline{A})} \ge \alpha n^2$, then for every half-set $B$, we have $e(A,B)\ge \frac{\alpha}{3} n^2$. In particular, if $G$ is non $\alpha$-extremal, then for every pair of half-sets $A$ and $B$, we have $e(A,B)\ge \frac{\alpha}{3} n^2$.
\end{lemma}
\begin{proof}
Let $A$ be a half-set such that $\min\set{e(A), e(A,\overline{A})} \ge \alpha n^2$, i.e., $G$ is not $\alpha$-extremal with respect to $A$. For contradiction, let $B$ be a half-set such that $e(A,B)<\frac{\alpha}{3} n^2$. We have $|\overline{A\cup B}| \le |A\cap B| + 1$. Thus, 
\[
|(A\cup B)\cap N(v)| \ge N(v)- |\overline{A\cup B}| \ge \left(\frac{1}{2}-\epsilon\right)n - |A\cap B| - 1.
\] 
Then, 
\[
e(A,B) \ge \sum_{v\in A\cap B} |(A\cup B)\cap N(v)| \ge |A\cap B| \cdot \left(\left(\frac{1}{2}-\epsilon\right)n - |A\cap B| - 1\right).
\]
We split into three cases.

\paragraph{Case I} If $\alpha n \le |A\cap B| \le (\frac{1}{2} - \alpha) n$, then 
\[
e(A,B)\ge |A\cap B| \cdot \left(\left(\frac{1}{2}-\epsilon\right)n - |A\cap B| - 1\right) \ge \frac{\alpha}{3} n^2.
\]
This is a contradiction to our assumption.

\paragraph{Case II}  If $|A\cap B|\le \alpha n$, then 
\[
e(A,\overline{A})\le e(A,B) + |A|\cdot |\overline{A\cup B}| 
< \frac{\alpha}{3} n^2 + \frac{2\alpha}{3} n^2 = \alpha n^2.
\]
This contradicts that $G$ is not $\alpha$-extremal with respect to $A$.

\paragraph{Case III} If $|A\cap B|\ge (\frac{1}{2} - \alpha) n$, then 
\[
e(A)\le e(A,A) 
\le e(A,B) + |A| \cdot (|A| - |A\cap B|) 
< \frac{\alpha}{3} n^2 + \frac{2\alpha}{3} n^2
= \alpha n^2.
\]
This again contradicts that $G$ is not $\alpha$-extremal with respect to $A$. This finishes the proof of \Cref{lem:KLS general}.
\end{proof}

We remark that the majority of Hamiltonicity-type results in random graphs extensively use the rotation-extension method, which was first introduced by P\'osa~\cite{posa1976hamiltonian}. We next use this versatile technique to prove \Cref{thm:robustdirac}. Some similar ideas were used in~\cite{krivelevich2014robust}.

\begin{proof}[{\bf Proof of \Cref{thm:robustdirac}}]
Let $\eta=\max\{\zeta,\epsilon\}$, thus $\delta(G)\ge (1/2-\eta) n$ and $|M|\le \eta n$. We first show that $G$ is connected. For the sake of contradiction, suppose $G$ is not connected. Since $\delta(G)\ge (1/2 - \eta)n$, the graph $G$ contains exactly two connected components $A,B$ with size at least $(1/2 - \eta)n$. Suppose $|A|\ge |B|$. Choose a half-set $S\subseteq V(G)$ containing $B$ and let $T=V(G)\setminus S$. Then $e_G(S,T)\le |S\setminus B|\cdot |T| \le 2\eta n^2 < \alpha n^2$, which is a contradiction. Thus, $G$ is connected.  

For a subgraph $H$ of $G\cup M$, we call it \textit{good} if for all $uv\in M$, if $u\in V(H)$ or $v\in V(H)$, then $uv\in E(H)$. To prove \Cref{thm:robustdirac}, it is enough to find a good Hamilton cycle in $G\cup M$. Consider a longest good path $P$ in $G$. Thus, it is enough to show the existence of a good cycle using all the vertices in $P$. Indeed, if $|V(P)|=n$, then this cycle would give us a desired Hamilton cycle; otherwise, using the connectedness of $G$ it is possible to extend the good cycle to a good path longer than $P$ as follows. Since $G$ is connected, there is an edge $uv$ where $u$ is a vertex from the cycle and $v$ lies outside of the cycle. Let $u'$ be a neighbor of $u$ in the cycle such that $uu'\notin M$. Then, if $uv\in M$ or $v\notin V(M)$, then we can simply delete the edge $uu'$ from the cycle and add $uv$ to obtain a good path longer than $P$, contradicting that $P$ is a longest good path. Similarly, if $uv\notin M$ and $v\in V(M)$, then let $v'$ be such that $vv'\in M$. Since the cycle is good, $v'$ is outside of the cycle, and now deleting $uu'$ and adding the path $uvv'$ to the cycle gives a good path longer than $P$, a contradiction.

By a similar argument, we can show that any longest good path has the property that the neighborhood of its endpoints must lie in the path. We will repeatedly use this fact. Suppose $P = (v_0,v_1,\ldots,v_{\ell})$. If $v_iv_{\ell}$ is an edge in $G$ and $v_iv_{i+1}\notin M$, then the path $P'=(v_0,v_1,\ldots,v_i,v_{\ell},v_{\ell - 1},\ldots,v_{i+1})$ is also a longest good path. It is common in the literature to say that $P'$ is obtained from $P$ by a \textit{rotation} with \textit{fixed endpoint} $v_0$, \textit{pivot point} $v_i$, and \textit{broken edge} $v_iv_{i+1}$. We can again apply a rotation on $P'$ with the fixed endpoint $v_0$ and avoiding the edges in $M$ as broken edges to get another longest good path. Thus, repeatedly applying such rotations on $P$ with fixed endpoint $v_0$ generates many good paths using all the vertices in $P$. Define $S_0:=\{v_{\ell}\}$. For $t\ge 1$, define $S_t$ to be the set of vertices $v\in V(P)$ such that there is a good path using all the vertices in $P$ with an endpoint $v$ that can be found by applying at most $t$ such rotations on $P$ with the fixed endpoint $v_0$. From the minimum degree condition, we have 
\begin{equation} \label{eq:size of S_1}
|S_1|\ge d(v_{\ell}) -|M| \ge (1/2-2\eta)n.
\end{equation}

We next claim that $|V(P)| \ge (1-10\eta)n$. Suppose not, then by the fact that $d(v_0),d(v_{\ell})\ge (1/2 -\eta)n$ and that every neighbor of $v_0,v_{\ell}$ must lie on $P$ and that $|M|\le \eta n$, we have that there must exist $j\in [\ell -1]$ such that $v_0 v_{j+1} \in E(G)$, $v_{\ell} v_j \in E(G)$, and $v_j v_{j+1} \notin E(M)$. Then, the cycle given by $v_0 v_{j+1} v_{j+2} \ldots v_{\ell} v_j v_{j-1} \ldots v_0$ is a desired cycle using all the vertices on $P$. Thus, we have the following:
\begin{equation} \label{eq:size of V(P)}
|V(P)| \ge (1-10\eta)n.
\end{equation}

If $|S_3|\ge (1/2 +\eta)n$, then $S_3\cap N(v_0)\neq \emptyset$, thus we can find a good cycle as desired. Thus, we can assume that $|S_3|<(1/2 +\eta)n$. For a set $A=\{v_{i_1},\ldots,v_{i_k}\}$, we define the sets $A^+:=\{v_{i_1 +1},\ldots,v_{i_k +1}\}$ and $A^-:=\{v_{i_1 -1},\ldots,v_{i_k -1}\}$. We ignore indices that are not between $0$ and $\ell$.

\begin{claim} \label{cl:size of S_3}
If $|S_3|<(1/2 +\eta)n$, then $|S_3^+ \cup S_3^-|\le (1/2 + 20\eta)n$.
\end{claim}
\begin{proof}
Assume that $|S_3|<(1/2 +\eta)n$. For $i\in [\ell]$, let $V_i:=\{v_0,\ldots,v_i\}$. Let $k$ denote the minimum index such that $|S_1\cap V_k|=4\eta n$. Thus, $v_k\in S_1$ and the edge $v_{k-1}v_{\ell}\in E(G)$. We claim that there exists $j\le k-2$ such that $v_{j-1}v_k, v_{j-1}v_\ell\in E(G)$ and $v_{j-1}v_j\notin M$. To prove this claim it suffices to show that $v_k$ has a neighbor $v_{j-1}$ among $v_1,v_2,\ldots,v_{k-3}$ such that $v_j \in S_1$ (if $v_j \in S_1$, then $v_{j-1}v_j\notin M$ ). Assume not. Then by considering the P\'osa rotations on the path $v_0,v_1,\ldots,v_{k-1},v_\ell,v_{\ell-1},\ldots,v_k$ that fix $v_0$ and give a good path we have
$$|S_3|\ge |S_2|\ge (|N(v_k)|-|M|)+|S_1\cap V_{k-3}| \ge (1/2-2\eta)n + 4\eta n-3>(1/2+\eta) n,$$
which contradicts our initial assumption. Thus we may assume that there exists $j<k-1$ such that $v_{j-1}v_k, v_{j-1}v_\ell\in E(G)$ and $v_{j-1}v_j\notin M$. Observe that in this case we have two longest good paths $P_1:=(v_0,\ldots,v_{j-1},v_{\ell},v_{\ell -1},\ldots,v_j)$ and $P_2:=(v_0,\ldots,v_{j-1},v_k,v_{k+1},\ldots,v_{\ell},v_{k -1},v_{k-2},\ldots,v_j)$ with the same endpoints that traverse the part in $V:=V_{\ell}\setminus V_k$ in opposite ways.  This ensures that the set $C:=(N(v_j)\cap V)\setminus V(M)$ is in both $S_3^+$ and $S_3^-$. As $|S_3|\le (1/2+\eta)n$ is an upper bound on both $|S_3^+|$ and $|S_3^-|$, to finish the proof, it is enough to show that $|C|\ge (1/2-18\eta)n$. This immediately follows from the following (where we use \eqref{eq:size of S_1} at the fourth inequality).
\begin{align*}
(1/2+\eta)n \ge |S_3| &\ge |S_1 \cap V| + |S_3\setminus V| 
\\&\ge \left(|S_1|- |V_k\cap N(v_\ell)^+|\right) + |N(v_j)\setminus \left(V\cup V(M)\right)| 
\\&\ge ((1/2-2\eta)n-4\eta n ) + |N(v_j)\setminus V| - 2|M| \ge (1/2-8\eta)n + |N(v_j)\setminus V|,
\end{align*}
thus $|N(v_j)\setminus V|\le 9\eta n$, and $|C|\ge |N(v_j)|-2|M|-|N(v_j)\setminus V|\ge (1/2-\eta)n-2\eta n-9\eta n=(1/2-12\eta)n$.
\end{proof}

Let $D:= V(P) \setminus (S_3^+ \cup S_3^- \cup V(M))$. Let $D',S_1'$ be half-sets such that both $|D'\cap D|$ and $|S_1'\cap S_1|$ are maximized. Then, by \eqref{eq:size of S_1}, \eqref{eq:size of V(P)}, and \Cref{cl:size of S_3}, we have $|D'\setminus D|\le 32\eta n$ and $|S_1'\setminus S_1|\le 2\eta n$. By \Cref{lem:KLS general}, we have $e_G(D', S_1')\ge \alpha n^2/3$, thus there is an edge from a vertex $v_t\in D$ to a vertex $v_s\in S_1$. As $v_t\in D\subseteq \overline{S_3^+\cup S_3^-\cup V(M)}$, we have that $v_t$ is not adjacent to $v_s$ on $P$ and $v_tv_{t+1},v_tv_{t-1} \notin M$. Thus, the edge $v_tv_s$ can be used to perform a P\'osa rotation on $v_0\ldots v_{s-1}v_\ell v_{\ell-1}\ldots,v_s$ that fixes $v_0$ and gives a good path. Thus if $t<s$ then $v_{t+1}\in S_2$, else $v_{t-1}\in S_2$. In both cases, we have that $v_t\in S_2^+\cup S_2^-\subseteq S_3^+\cup S_3^-$, which contradicts the fact that $v_t\in D$. This finishes the proof of \Cref{thm:robustdirac}.
\end{proof}

\subsection{Concentration inequalities}
In our proof, we use the following standard concentration inequalities. 
\begin{theorem} [Chernoff bound, see~\cite{chernoff1952measure,mitzenmacher2017probability}] \label{chernoff} 
Let $0\le p\le 1$ and $X_1,\ldots,X_n$ be independent $Bernoulli(p)$ random variables. Let $X = \sum_{i=1}^n X_i$ and $\mu = \mathbb{E}(X) = np $. Then for any $0 < \delta < 1$, we have 
\[
\Pr\left[|X - \mu| \ge \delta \mu\right] \le 2 e^{-\frac{\delta^2 \mu}{3}}.
\]
\end{theorem}

\begin{theorem}[{\cite[Theorem 1.9]{Warnke2016}}] \label{lem:concentration_permutation}
Let $N \in \mathbb{N}$ and let $S_N$ be the set of all bijections from  $[N]$ to $[N]$. Let $M>0$. Suppose that $f \colon S_N \rightarrow \mathbb{R}$ is such that for all $\pi, \pi' \in S_N$ which differ in exactly two values, we have
    \[
        |f(\pi) - f(\pi')| \le  M.
    \]
    If $\pi \in S_N$ is chosen uniformly at random, then
    \[
        \Pr(\left|f(\pi)- \mathbb{E}(f(\pi))\right| \ge t ) \le 2\exp\left(-\frac{t^2}{8NM^2}\right).
    \]
\end{theorem}

%%%%%%%%%%%%%%%%%%%%%%%%%%%%%%%%%%%%%%%%%%%%%%%%%%%%%%%%%%%%%%%%%%
\section{The vortex and the absorber} \label{section:vortex}
In this section, we first define the vortex and the absorber used for the proof of \Cref{thm:main_spread}. We then prove, in \Cref{lem:vortex}, that we can sample them from an appropriate spread probability measure. In the following, we often abuse notation by denoting both the vertex set and color set by $[n]$.

\begin{definition}[Vortex] \label{def:vortex}      
Let $0<\beta ,\delta, \epsilon <1<N$, and let $\mathbb{G}=\{G_1,\ldots,G_n\}$ be a collection of graphs on~$[n]$. We say that a pair of partitions $(\cV,\cC)$, $\cV=(V_1,\ldots,V_N)$ and $\cC=(C_1,\ldots,C_N)$ of $[n]$ is a $(\beta,\delta, \epsilon, N, \mathbb{G})$ vertex-color vortex if the following are satisfied.
\begin{itemize}
\item[(i)] $|V_{i+1}|= (1\pm \beta/10) \beta^i n$ for $1\le i\le N-1$, hence $|V_{i+1}|= (1\pm \beta) \beta |V_i|$ for $1\le i\le N-1$; 
\item[(ii)] $|C_i| = (1\pm \beta^4)(1-\beta^3)|V_i|$ for $2\le i\le N-1$;
\item[(iii)] $\delta^{-1}\le |C_N| /|V_N| \le 2\delta^{-1}$;
\item[(iv)] $d_{G_j}(v,V_l)\ge (1/2-\epsilon)|V_l|$ for $l\in\{i,i+1 \}$, $v\in V_i \cup V_{i+1}  $ and $j\in C_i\cup C_{i+1}$,  $1\le i\le N-1$.  
\end{itemize}
We denote the set of such vortices by $\cV\cC(\beta,\delta, \epsilon, N, \mathbb{G})$.
\end{definition}

\begin{definition}[Absorber]\label{def:absorber}
Let $0<\beta ,\delta, \epsilon <1<N$, and let $\mathbb{G}=\{G_1,\ldots,G_n\}$ be a collection of  graphs on~$[n]$. For partitions $\cV=(V_1,\ldots,V_N)$ and $\cC=(C_1,\ldots,C_N)$ of $[n]$ such that $(\cV,\cC) \in \cV\cC(\beta,\delta, \epsilon, N, \mathbb{G})$, a set of colors $C_{abs}\subseteq C_1$ and a set of edges $M_{abs}$ spanned by $V_1$, we say that the pair $(M_{abs},C_{abs})$ is $(V_N,C_N)$ absorbing if for every subset $A\subseteq C_N$ of size $(1-\beta^3)|V_N|$, there exists a bijection $\phi:M_{abs}\to C_{abs}\cup (C_N\setminus A)$ such that $e\in G_{\phi(e)}$ for every $e\in M_{abs}$.
\end{definition}

Since we use deterministic arguments to find a Hamilton cycle in the last part of the partition $V_N$, we must ensure that $|V_N|$ is small enough. This is because, as indicated by \eqref{eq:toy_spread}, the spreadness of the resulting distribution will depend on $|V_N|$. In our application, it turns out that taking $|V_N|=o(\log n)$ works. We choose to set $|V_N|$ as a sufficiently large constant (this will always be a function of a sufficiently large number~$L$) as this simplifies the underlying calculations.

\begin{definition}[Vortex with absorbing power] \label{def:vortexabsorber}
Let $0<\alpha, \alpha', \beta ,\delta, \epsilon <1 <L,N$. Let $\mathbb{G}=\{G_1,\ldots,G_n\}$ be a collection of graphs on $[n]$ that is not $\alpha$-extremal. We define $\mathcal{S}(\alpha,\alpha',\beta,\delta,\epsilon,N,\mathbb{G})$ to be the set of $4$-tuples $(\cV,\cC,M_{abs},C_{abs})$, with $\cV=(V_1,\ldots,V_N)$ and $\cC=(C_1,\ldots,C_N)$, such that  $(\cV,\cC)\in \cV\cC(\beta,\delta,\epsilon,N, \mathbb{G})$, $(M_{abs},C_{abs})$ is $(V_N,C_N)$ absorbing,    $|M_{abs}|,|C_{abs}|\le L^4$ and $\mathbb{G}_{C_i}[V_i]$ is not $\alpha'$-extremal for every $i\in [N]$. 
\end{definition}

In the following lemmas and in \Cref{thm:main_spread_non_extremal}, we slightly weaken the minimum degree condition imposed on the graph family $\mathbb{G}$ from $n/2$ to $(1/2-\eta)n$. This modification will ease the deduction of \Cref{thm:main_distinct_packing} in the non-extremal case. The following is the main result of this section.
\begin{lemma}[Vortex-absorber lemma] \label{lem:vortex}
Let $0<1/n \ll 1/L , \eta  \ll \epsilon \ll \delta, \beta \ll  \alpha' \ll \alpha \le 1$. Let $N=N(\beta,n,L)$ be the minimum integer such that $\beta^{N-1} n \in [L, 2\beta^{-1} L]$. Let $\mathbb{G}=\{G_1,\ldots,G_n\}$ be a collection of graphs on $[n]$ of minimum degree at least $(1/2-\eta)n$ that is not $\alpha$-extremal. Then, there exists a distribution $\mathcal{D}_{VC}$ on elements $(\cV,\cC,M_{abs},C_{abs})$ of $\cS(\alpha,\alpha',\beta,\delta,\epsilon, N,\mathbb{G})$ such that if $(\cV,\cC,M_{abs},C_{abs})$ is generated according to $\mathcal{D}_{VC}$ then, with $\cV=(V_1,\ldots,V_N)$ and $\cC=(C_1,\ldots,C_N)$, the following holds. For every $m\in [n]$, quadruple of sets $A=\{a_1,\ldots,a_m\}\subseteq V(G)$, $B=\{b_1,\ldots,b_m\}\in [N-1]^m$, $C=\{c_1,\ldots,c_m\}\subseteq [n]$, $D=\{d_1,\ldots,d_m\}\in [N-1]^m$, and quadruple of sets $U_1,U_N,Q_1,Q_N\subseteq [n]$, we have 
\begin{align}\label{eq:partitions}
\Pr_{\mathcal{D}_{VC}}\left( U_N\subseteq V_N \bigwedge Q_N \subseteq C_N \bigwedge U_1 \subseteq V(M_{abs}) \bigwedge Q_1 \subseteq C_{abs} \bigwedge \left(\bigwedge_{i\in [m]} \set{ a_i \in V_{b_i}  \text{ and } c_i \in C_{d_i}}\right) \right) \nonumber
\\\le 2\bfrac{L^4}{n}^{|U_1|+|U_N|+|Q_1|+|Q_N|}\prod_{i=1}^m\frac{2|V_{b_i}|}{n} \prod_{i=1}^m \frac{2|C_{d_i}|}{n}.
\end{align}
\end{lemma}

Before proving this lemma, we collect a couple of auxiliary lemmas that will be useful in our proof.
\begin{lemma}\label{lem:complement of extremal graph}
Let $0<1/n \ll \eta \ll \alpha \le 1$. Let $G_1$ and $G_2$ be two graphs on a vertex set $V$ of size $n$ with at most $0.1\alpha n^2$ common edges. Suppose that $G_1$ has minimum degree at least $(1/2-\eta)n$ and is $\alpha$-extremal with respect to a half-set $U\subseteq V$. Then, $G_2$ is $3\alpha$-extremal with respect to $U$. 
\end{lemma}
\begin{proof}
Let the graph $G^*$ be the complement of $G_1$. Then, $G_2$ can be turned into a subgraph of $G^*$ by removing at most $0.1\alpha n^2$ edges. Now, since $G_1$ has minimum degree at least $(1/2-\eta)n$ and $G_1$ is $\alpha$-extremal with respect to $U$, the graph $G^*$ is $(2\alpha +\eta)$-extremal with respect to $U$. As a consequence, since $G_2$ can be turned into a subgraph of $G^*$ by removing at most $0.1\alpha n^2$ edges, $G_2$ is $(2\alpha +\epsilon +0.1\alpha)$-extremal with respect to $U$. Thus, $G_2$ is $3\alpha$-extremal with respect to $U$. 
\end{proof}

\begin{lemma}\label{lem:preserve minimum degree}
Let $ \eta, \epsilon \in (0,1)$ be such that $\eta\le \epsilon/10$. Let $\mathbb{G}=\{G_1,\ldots,G_n\}$ be a collection of graphs on a vertex set $V$ of size $n$ with minimum degree at least $(1/2-\eta)n$. Let $p_1,q_1,p_2\in [0,1]$ and let $C=[n]$. Let $V_1, V_2\subseteq V$ be disjoint sets and $C_1\subseteq C$ sampled according to $V(p_1), V(p_2), C(q_1)$, respectively. Then, with probability at least $1 - 2(e^{-\epsilon^2 np_1/3\cdot 100^2}+ e^{-\epsilon^2 nq_1/3\cdot 100^2}+n^2p_1q_1(e^{-\epsilon^2 np_1/12}+e^{-\epsilon^2 np_2/12}))$, the following hold:
\begin{itemize}
    \item[(I)] $||V_1|-np_1|\le \epsilon np_1/100$ and $||C_1|-nq_1|\le \epsilon nq_1/100$,
    \item[(II)] for every $v\in V_1$, $c\in C_1$, and $i\in [2]$ we  have that  $d_{G_c[V_i]}(v)\ge (1/2 - \epsilon) |V_i|$. (In particular, $\mathbb{G}_{C_1}[V_1]$ is a family of graphs with minimum degree at least $(1/2-\epsilon)|V_1|$).
\end{itemize}
\end{lemma}
\begin{proof}
Denote the events that (I) and (II) hold by $\cE_I$ and $\cE_{II}$ respectively. By the Chernoff bound (\Cref{chernoff}), the probability that $\cE_I$ does not hold is at most $ 2e^{-\epsilon^2 np_1/3\cdot 100^2}+ 2e^{-\epsilon^2 nq_1/3\cdot 100^2}$. Thereafter, note that for $v\in V$ and $c\in C$, in the event $v\in V_1$ and $c\in C_1$, we have that for $i\in [2]$ the quantity $d_{G_c[V_i]}(v)$ dominates a $Bin((1/2-\eta)n ,p_i)$ random variable. In particular, $d_{G_c[V_i]}(v)<(1/2 - \epsilon) |V_i|$ with probability at most $ 2e^{-\epsilon^2 np_i/12}$, by the Chernoff bound. It follows that $\cE_{II}$ fails to hold with probability at most
\begin{align*}
   &\sum_{i\in [2]} \sum_{v\in V} \sum_{c\in C} \Pr(v\in V_1 \text{ and }c\in C_1 \text{ and }d_{G_c[V_i]}(v)<(1/2 - \epsilon) |V_i|)
   \\&= \sum_{i\in [2]} \sum_{v\in V} \sum_{c\in C} \Pr(v\in V_1 \text{ and }c\in C_1)   \Pr(d_{G_c[V_i]}(v)<(1/2 - \epsilon) |V_i| \big| v\in V_1 \text{ and }c\in C_1 )
   \\&\le \sum_{i\in [2]} \sum_{v\in V} \sum_{c\in C} p_1q_1  2e^{-\epsilon^2 np_i/12}
   =2n^2p_1q_1(e^{-\epsilon^2 np_1/12}+e^{-\epsilon^2 np_2/12}). 
\end{align*}
\end{proof}

We are now ready to prove \Cref{lem:vortex}.
\begin{proof}[{\bf Proof of \Cref{lem:vortex}}]
We choose $\alpha''$ such that the following hierarchy of constants holds. 
\[
0 < 1/n \ll 1/L, \eta \ll \epsilon \ll \delta , \beta \ll \alpha' \ll \alpha'' \ll \alpha \le 1.
\]
We first construct the pairs $(M_{abs},V_{abs})$ and $(V_N,C_N)$. Then we later extend them to an element of $\cS(\alpha,\alpha',\beta,\delta,\epsilon, N,\mathbb{G})$. Denote the vertex set of $\mathbb{G}$ by $V$ and the color set by $C$. Let $p=L^4/(2n)$ and $V'=V(p)$ and $C'=C(p)$. Let $\cE$ be the event that the following hold.
\begin{itemize}
    \item $\max\{||V'|-L^4/2|, ||C'|-L^4/2| \}\le \beta^2L^4$, hence $L^4/4 \le |V'|, |C'| \le L^4$.
    \item $G_c[V']$ has minimum degree at least $(1/2 -\epsilon/2)|V'|$ for every color $c\in C'$.
    \item The family $\mathbb{G}_{C'}[V']$ is not $50\alpha''$-extremal.
\end{itemize}
\Cref{lem:property preserve,lem:preserve minimum degree} imply that $\cE$ holds with probability at least $8/10$. 
\begin{claim}\label{claim:constructing the absorber}
If the event $\cE$ holds, then there is a subset $C^*\subseteq C'$ with $|C^*|\ge \alpha'' |C'|/2$ such that the family $\mathbb{G}_{C^*}[V']$ is not $\alpha''$-extremal and the following property holds.
\begin{itemize}
\item[{\normalfont (P)}] For every pair of colors $c_1,c_2\in C^*$, there are at least $\alpha'' |C^*|$ colors $c\in C^*\setminus \{c_1,c_2\}$ such that the graph $G_{c}[V']$ intersects $G_{c_i}[V']$ in at least $0.1\alpha'' |V'|^2$ edges for each $i\in [2]$.
\end{itemize}
\end{claim}
\begin{proof}
If the property~(P) holds with $C^*=C'$, then set $C^*=C'$. If (P) does not hold with $C^*=C'$, then there exist two colors $a,b\in C'$ and a subset $C''\subseteq C'\setminus \{a,b\}$ such that $|C''|\ge(1-\alpha'')|C'|$, and for every $c\in C''$, the graph $G_c[V']$ intersects at least one of $G_a[V']$, $G_b[V']$ in at most $0.1\alpha''|V'|^2$ edges. Let $A\subseteq C''$ (and $B$ respectively) be the set of all colors $c\in C''$ such that $G_c[V']$ intersects $G_a[V']$ (and $G_b[V']$ respectively) in at most $0.1\alpha''|V'|^2$ edges. Note that $A\cup B= C''$. Without loss of generality, assume that $|A|\ge |B|$. We next show that if the event $\cE$ holds, then at least one of the following holds.
\begin{itemize}
    \item[(1)] For every $c_1,c_2\in C''$, the graphs $G_{c_1}[V'],G_{c_2}[V']$ intersect in at least $0.1\alpha'' |V'|^2$ edges.
    \item[(2)] The family $\mathbb{G}_A[V']$ is not $\alpha''$-extremal.
    \item[(3)] $|B|\ge \alpha'' |C'|$ and $\mathbb{G}_B[V']$ is not $\alpha''$-extremal. 
\end{itemize}
Assume that $\cE$ occurs. Assume further that (1) and (2) do not hold and let $V_A\subseteq V'$ be a half-set such that $\mathbb{G}_A[V']$ is $\alpha''$-extremal with respect to $V_A$. If $|B|<\alpha''|C'|$ then 
\begin{align*}
    \sum_{G\in \mathbb{G}_{C'}[V']} \min\{e_G(V_A),e_G(V_A,\overline{V_A})\} & \le \sum_{G\in \mathbb{G}_{A}[V']} \min\{e_G(V_A),e_G(V_A,\overline{V_A})\}\} + |B||V'|^2+|C'\setminus C''||V'|^2
    \\& \le \alpha'' |A| |V'|^2+ 2\alpha'' |C'| |V'|^2 \le 3\alpha''|C'||V'|^2,
\end{align*}
thus $\mathbb{G}_{C'}[V']$ is $10\alpha''$-extremal contradicting the occurrence of $\cE$. Therefore, $|B|\ge \alpha''|C'|$. Now furthermore, for the sake of a contradiction, assume that $\mathbb{G}_B[V']$ is $\alpha''$-extremal with respect to a half-set $V_B\subseteq V'$. 

Under the assumption that (2) does not hold (i.e., $\mathbb{G}_A[V']$ is $\alpha''$-extremal with respect to $V_A$), there exists $G'\in \mathbb{G}_A[V']$ such that ${\min\{e_{G'}(V_A),e_{G'}(V_A,\overline{V_A})\} \le \alpha''|V'|^2}$.
Therefore, since the minimum degree of $G'$ is at least $(1/2-\epsilon)|V'|$, and $G',G_a[V']$ intersect in at most $0.1\alpha''|V'|^2$ edges, and $\epsilon\ll \alpha''$, by \Cref{lem:complement of extremal graph}, we have that $G_a[V']$ is $3\alpha''$-extremal with respect to $V_A$. By another application of \Cref{lem:complement of extremal graph} with the same reasoning, every graph $G\in \mathbb{G}_A[V']$ is $9\alpha''$-extremal with respect to $V_A$. 

By definition of $B$, every graph $G\in \mathbb{G}_B[V']$ intersects $G_b[V']$ in at most $0.1\alpha'' |V'|^2$ edges, and intersects the complement of $G_b[V']$, denoted by $\overline{G_b[V']}$, in at least $(1/4-\alpha'')|V'|^2$ edges. Therefore, since the graph $\overline{G_b[V']}$ has at most $(1/4+\epsilon)|V'|^2$ edges, for every graph $G\in \mathbb{G}_B[V']$, the graph $\overline{G_b[V']}$ contains at most $(\alpha''+\epsilon) |V'|^2$ edges that do not belong to $G$. Thus, for every $G_1,G_2\in \mathbb{G}_B[V']$, we have that 
\begin{align*}
    |E(G_1)\triangle E(G_2)| &\le |E(G_1)\cap E(G_b[V'])|+ |E(G_2)\cap E(G_b[V'])|+ |(E(G_1)\triangle E(G_2))\cap E( \overline{G_b[V']})|
    \\&\le 0.1\alpha''|V'|^2+ 0.1\alpha''|V'|^2+|E( \overline{G_b[V']})\setminus E(G_1)|+| E( \overline{G_b[V']})\setminus E(G_2)|
    \\& \le 0.2\alpha''|V'|^2+ (\alpha''+\epsilon) |V'|^2+ (\alpha''+\epsilon) |V'|^2 \le 3\alpha''|V'|^2.
\end{align*}
Since (1) does not hold, assume that $c_1,c_2\in C''$ are such that $G_{c_1}[V'], G_{c_2}[V'] $ intersect in fewer than $0.1\alpha'' |V'|^2$ edges. Then $c_1$ and $c_2$ do not belong to the same set in $\{A,B\}$, since, as seen by the calculation above, for $c,c'\in B$ (similarly for $c,c'\in A)$, $G_c[V']$ and $G_{c'}[V']$ intersect in at least $(1/4-\epsilon-3\alpha'')|V'|^2 > 0.1\alpha'' |V'|^2$ edges. Thus, we may assume that $c_1\in A$ and $c_2\in B$.

Now, since $G_{c_1}[V']$ has minimum degree at least $(1/2-\epsilon)|V'|$ and is $9\alpha''$-extremal with respect to $V_A$ (as $c_1\in A$), and $G_{c_1}[V'], G_{c_2}[V']$ intersect in fewer than $0.1\alpha'' |V'|^2$ edges, by \Cref{lem:complement of extremal graph}, $G_{c_2}[V']$ is $27\alpha''$-extremal with respect to $V_A$. Furthermore, for every $G\in \mathbb{G}_B[V']$, since $c_2\in B$, we have $|E(G)\triangle E(G_{c_2})|\le 3\alpha''|V'|^2$ by the calculation above, thus $G$ is $(27\alpha'' +3\alpha'')\le 30\alpha''$-extremal with respect to $V_A$. Recall that every graph in $\mathbb{G}_A[V']$ is $9\alpha''$-extremal with respect to $V_A$. It follows that $\mathbb{G}_{C''}[V']$ is $30\alpha''$-extremal and consequently, $\mathbb{G}_{C'}[V']$ is $31\alpha''$-extremal with respect to $V_A$. This contradicts the occurrence of $\cE$. This shows that if $\cE$ holds, then one of the events among (1)--(3) must hold.

Note that the property (P) always holds with $C^*=A$ (similarly with $C^*=B$). Indeed, for $c_1,c_2,c\in A$, each of $G_{c_1}[V'],G_{c_2}[V']$, and $G_{c}[V']$ can be turned into a subgraph of the complement of $G_a[V']$ by removing at most $0.1\alpha''|V'|^2$ edges. Since each of these graphs has minimum degree at least $(1/2-\epsilon)|V'|$, the number of edges $G_{c}[V']$ has in common with each of $G_{c_1}[V'],G_{c_2}[V']$ is at least 
\[
3|V'| (1/2-\epsilon)|V'|/2 - 2\cdot 0.1\alpha''|V'|^2 - \binom{|V'|}{2} \ge (1/4 -3\epsilon/2 -0.2\alpha'')|V'|^2 \ge 0.1\alpha'' |V'|^2.
\]
Finally, if (1) holds, then set $C^*=C''$; otherwise, set $C^*=A$ if (2) holds, else set $C^*=B$. In every case, the property (P) holds, the family $\mathbb{G}_{C^*}[V']$ is not $\alpha''$-extremal, and $|C^*|\ge \alpha''(1-\alpha'')|C'|\ge \alpha'' |C'|/2$. This completes the proof of \Cref{claim:constructing the absorber}.
\end{proof}

We finally let $V_N = V'(p')$, where $p' = \beta^{N-1}  /(L^4/2)$, and let $C_N = C^*(q')$, where $q'=1.5\delta^{-1} \beta^{N-1}n/|C^*|$. If the event $\cE$ holds, then the facts that $|V'|,|C'|\ge L^4/4$ and $|C^*|\ge \alpha'' |C'|/2$, and the condition $1/L \ll \alpha'',\beta,\delta$ imply that $p',q'<1$. Let $\cE_N$ be the event that the following hold:
\begin{itemize}
    \item[(i)] the event $\cE$ holds; 
    \item[(ii)] $|V_N|= (1\pm \beta/10)\beta^{N-1}n$ (recall that $N$ is chosen so that $\beta^{N-1}n\in [L,2\beta^{-1}L]$);         
    \item[(iii)] $\delta^{-1}\le |C_N|/|V_N|\le 2\delta^{-1}$ (thus $|C_N|\le 5\beta^{-1}\delta^{-1}L$);
    \item[(iv)] the graph family $\mathbb{G}_{C_N}[V_N]$ is not $\alpha'$-extremal and has minimum degree at least $(1/2-\epsilon)|V_N|$.
\end{itemize}
First recall that $\Pr(\cE)\ge 8/10$. Thereafter in the event $\cE$, by \cref{claim:constructing the absorber}, we have that $\mathbb{G}_{C^*}[V']$ is not $\alpha''$-extremal. Since $V_N=V'(p')$ and $C_N=C^*(q')$, \Cref{lem:property preserve} gives that conditioned on the event $\cE$, the family $\mathbb{G}_{C_N}[V_N]$ is not $\alpha'$-extremal with probability at least $99/100$. Moreover, since $V_N=V'(p')=V(pp')$, \Cref{lem:preserve minimum degree} gives that $|V_N|= (1\pm \beta/10)\beta^{N-1}n$ and the graph family $\mathbb{G}_{C'}[V_N]\supseteq \mathbb{G}_{C_N}[V_N]$ has minimum degree at least $(1/2-\epsilon)|V_N|$ with probability at least $99/100$. Finally, since $C_N = C^*(q')$, conditioned on the event $|V_N|= (1\pm \beta/10)\beta^{N-1}n$, the Chernoff bound gives that $\delta^{-1}\le |C_N|/|V_N|\le 2\delta^{-1}$ with probability at least $99/100$. Everything together gives that the event $\cE_N$ holds with probability at least $7/10$.

\begin{claim}\label{clm:absorber}
If the event $\cE_N$ holds, then there is a set of colors $C_{abs}\subseteq C^*\setminus C_N$ and a matching $M_{abs}$ spanned by $V'\setminus V_N$ such that $|M_{abs}|, |C_{abs}|\le L^4$ and $(M_{abs},C_{abs})$ is $(V_N,C_N)$ absorbing.
\end{claim}
\begin{proof}
Let $V^*=V'\setminus V_N$. Since $\cE_N$ holds and $1/L \ll \alpha'', \beta, \delta$, we have
\begin{equation}\label{eq:VNCN}
|V_N|\le |C_N|\le 5\beta^{-1}\delta^{-1}L \le \alpha''^2 L^4/4 \le\alpha''^2 |V'|\le 2\alpha''^2 |V^*|,
\end{equation}
where the last inequality uses the already established fact that $|V_N|\le \alpha''^2 |V'|$. Since the property (P) holds, by \eqref{eq:VNCN}, the property (P*) also holds.
\begin{itemize}
\item[(P*)] For every pair $c_1,c_2\in C^*$, there are at least $\alpha'' |C^*|$ colors $c\in C^*\setminus \{c_1,c_2\}$ such that the graph $G_{c}[V^*]$ intersects $G_{c_i}[V^*]$ in at least $0.05\alpha'' |V^*|^2$ edges for each $i\in [2]$.
\end{itemize}
(P*) implies that for every $c\in C_N\subseteq C^*$, there exist at least $0.05\alpha''|V^*|^2|C^*|$ pairs $(c',e)\in C^*\times E(G_{c}[V^*])$ such that $e$ belongs to both $G_{c}[V^*]$ and $G_{c'}[V^*]$. Thus, the number of edges of $G_{c}[V^*]$ that can be colored with at least $2|C_N|^2$ colors from $C^*$ is at least $(0.05\alpha''|V^*|^2|C^*| - 2|C_N|^2 \cdot |V^*|^2/2)/ |C^*| \ge 0.01\alpha''|V^*|^2$, where we use the following inequality. 
\begin{equation}\label{eq:CstarvsCN}
0.05\alpha''|C^*|\ge 0.05(\alpha'')^2|C'|/2 \ge 0.05(\alpha'')^2 L^4/8 \ge 2(5\beta^{-1}\delta^{-1}L)^2\ge 2|C_N|^2.
\end{equation}
Thus, since $|C_N|\le 2\alpha''^2|V^*|$ (by \eqref{eq:VNCN}), there exists a $C_N$-transversal matching $M$ on $V^*$ such that every edge in $M$ can be colored with at least $2|C_N|^2$ colors from $C^*$. 

Let $R,S\subseteq C_N$ be an arbitrary partition of $C_N$ into two disjoint sets where $S$ has size $(1-\beta^3)|V_N|$. For every $r\in R$, let $e_r$ be the edge with color $r$ in $M$. Now for each pair of colors $r\in R$, $s\in S$, we will identify a pair of edges $e_1(r,s)$, $e_2(r,s)$ and a pair of colors $c_1(r,s)$, $c_2(r,s)$ such that
\begin{itemize}
    \item the edges $e_r$ and $e_1(r,s)$ can be colored with the color $c_1(r,s)$;
    \item the edges $e_1(r,s)$ and $e_2(r,s)$ can be colored with the color  $c_2(r,s)$;
    \item the edge $e_2(r,s)$ can be colored with the color $s$. 
\end{itemize}
Thus the set $\{e_r,e_1(r,s),e_2(r,s)\}$ can be rainbow-colored using either the colors $r$, $c_1(r,s)$, and $c_2(r,s)$ or the colors $s$, $c_1(r,s)$, and $c_2(r,s)$. We will also ensure that these edges are pairwise non-adjacent and the colors are pairwise different. 

For that, let $(r_i,s_i)$, $i\in [|R||S|]$ be an ordering of the pairs in $R\times S$, and assume that so far, we have chosen suitable edges $e_1(r_j,s_j),e_2(r_j,s_j)$ and colors $c_1(r_j,s_j),c_2(r_j,s_j)$ for every $j<i$. In particular, so far, we have ruled out at most $2|C_N|^2-2$ colors. Call these colors unavailable and the rest available. We first choose the colors $c_1(r_i,s_i)$ and  $c_2(r_i,s_i)$. Since $e_{r_i}$ belongs to $M$, $e_{r_i}$ can be colored with at least $2|C_N|^2$ colors from $C^*$. Let $c_1(r_i,s_i)$ be one of these colors that is available. Then, given the color $c_1(r_i,s_i)$, by (P*) and \eqref{eq:CstarvsCN}, there exist at least $\alpha''|C^*|\ge 2|C_N|^2$ colors $c\in C^*$ such that $G_c[V^*]$ intersects each of $G_{c_1(r_i,s_i)}[V^*]$ and $G_{s_i}[V^*]$ in at least $0.05 \alpha'' |V^*|^2$ edges. Choose $c_2(r_i,s_i)$ to be an available such color that is distinct from $c_1(r_i,s_i)$. We continue and choose the edges $e_1(r_i,s_i)$ and $e_2(r_i,s_i)$. Choose the edge $e_1(r_i,s_i)$ from $G_{c_1(r_i,s_i)}[V^*]\cap G_{c_2(r_i,s_i)}[V^*]$ and the edge $e_2(r_i,s_i)$ from $G_{c_2(r_i,s_i)}[V^*]\cap G_{s_i}[V^*]$ subject to the condition that the set of edges chosen so far induces a matching (this condition rules out at most $2|C_N|^2|V^*|\le 2(5\beta^{-1}\delta^{-1}L)^2|V^*| \le 0.01\alpha'' (L^4/8)|V^*|\le 0.01 \alpha'' |V^*|^2$ out of the at least $0.05 \alpha'' |V^*|^2$ available edges). We set $M_{abs}=\bigcup_{i\in [|R||S|]}\{e_1(r_i,s_i),e_2(r_i,s_i)\} \cup\{e_r:r\in R\}$ and $C_{abs}=\bigcup_{i\in [|R||S|]}\{c_1(r_i,s_i),c_2(r_i,s_i)\}$.

Note that for every subset $A\subseteq C_N$ of size $|S|$, there exists a bijection $\phi:M_{abs}\to C_{abs}\cup (C_N\setminus A)$ such that $e\in G_{\phi(e)}$ for $e\in M_{abs}$. Indeed, first let $\psi:S\setminus A \to  A\setminus S$ be a bijection. For every $r\in R,s\in S$ such that $\{r,s\}$ is not a subset of $S \triangle A$, or $\{r,s\}$ is a subset of $S \triangle A$ but $\psi(s)\neq r$, set $\phi(e_i(r,s))=c_i(r,s)$ for $i\in [2]$. Thereafter, for $c\in R\setminus (S\triangle A)$, set $\phi(e_c)=c$. Finally, for each pair $r\in A\setminus S, s\in S\setminus A$ such that $\psi(s)=r$, set $\phi(e_r)=c_1(r,s)$, $\phi(e_1(r,s))=c_2(r,s)$, and $\phi(e_2(r,s))=s$. Then, $\phi(M_{abs})=C_{abs}\cup (C_N\setminus A)$, as desired. This finishes the proof of \Cref{clm:absorber}.
\end{proof}

Now we construct the pair of partitions $(\cV,\cC)$. For that, first generate $V_N, C_N$ as described before, and then if $\cE_N$ holds, then generate $M_{abs}, C_{abs}$ according to \Cref{clm:absorber}. Thus with $p=L^4/(2n)$ and $p'=\beta^{N-1}/(L^4/2)$ we have that $V_N,V(M_{abs})\subseteq V'=V(p)$, $V_N=V'(p')=V(pp')$ and $C_N,C(M_{abs})\subseteq C'=C(p)$. 
Now, let $V^-=V\setminus V'$ and  $C^-=C\setminus C'$. Let $\mathbf{p}=(p_1,\ldots,p_{N-1})$ be such that $p_i=\beta^{i-1}$ for $2\le i\le N-1$ and $p_1=1-\sum_{i=2}^{N-1}p_i$.
Also, let $\mathbf{q}=(q_1,\ldots,q_{N-1})$ be such that $q_i=(1-\beta^3) p_{i}$ for $2\le i\le N-1$ and $q_1=1-\sum_{i=1}^Nq_i$. 
To construct $\cV=(V_1,\ldots,V_N)$, first generate $(V_1',\ldots,V_{N-1}')$ by independently assigning each vertex $v\in V^-$ to the set $V_i'$ with probability $p_i$. Then set $V_1=V_1'\cup (V'\setminus V_N)$, and $V_i=V_i'$ for $2\le i\le N-1$. Similarly, construct the partition $(C_1',\ldots,C_{N-1}')$ of the colors by assigning independently each color $c\in C^-$ to the set $C_i'$ with probability $q_i$. Then set $C_1=C_1'\cup (C'\setminus C_N)$, and $C_i=C_i'$ for $2\le i\le N-1$. 

We will first show that the generated quadruple $(\cV,\cC,M_{abs},C_{abs})$ belongs to   $\cS(\alpha,\alpha',\beta,\delta,\epsilon,N,\mathbb{G})$ with probability at least $1/2$. For this, recall that $\Pr(\cE_N)\ge 7/10$. If $\cE_N$ holds, then the property (iii) in \Cref{def:vortex} holds, $\mathbb{G}_{C_N}[V_N]$ is not $\alpha'$-extremal and has minimum degree at least $(1/2-\epsilon)|V_N|$, $|V_N|=(1\pm  \beta/10)\beta^{N-1}n$, and by \Cref{clm:absorber}, the sets $M_{abs}, C_{abs}$ satisfy the relevant properties in \Cref{def:absorber,def:vortexabsorber}. Thereafter note that, with $p_1'=(1-p)p_1+p(1-p')$, $p_N'=pp'$ and $p_i'=(1-p)p_i$ for $2\le i\le N-1$, we have that $V_i=V(p_i')$ for $i\in [N]$. In addition, with $q_i'=(1-\beta^3)p_i'$ for $2\le i\le N-1$ and $q_1'=(1-p)q_1+p$ we have that $C_i=C(q_i')$ for $2\le i\le N-1$ and $C_1\subseteq C_1'\cup C'= C(q_1')$. \Cref{lem:preserve minimum degree} gives that the event $\cE_N$  and the properties (i), (ii), and (iv) in \Cref{def:vortex} hold with probability at least 
\[
\frac{7}{10} - \sum_{i=1}^{N-1} 2\left(e^{-\epsilon^2 np_i'/3\cdot 100^2}+ e^{-\epsilon^2 nq_i'/3\cdot 100^2}+n^2p_i'q_i'\left(e^{-\epsilon^2 np_i'/12}+e^{-\epsilon^2 np_{i+1}'/12}\right)\right) 
\ge \frac{6}{10}.
\]
(Formally, here for the index $i=1$ we apply \Cref{lem:preserve minimum degree} with the set  $C_1'\cup C'$ in place of $C_1$, hence $d_{G_j}(v,V_l)\ge (1/2-\epsilon)|V_l|$ for $l\in\{1,2 \}$, $v\in V_1 \cup V_2$, and $j\in C_1'\cup C'\cup C_2$.)
\Cref{lem:property preserve} gives that, with probability at least $9/10$, $\mathbb{G}_{C_1'\cup C'}[V_1]$ is not $2\alpha'$-extremal, and for every $2\le i \le N-1$, we have $\mathbb{G}_{C_i}[V_i]$ is not $\alpha'$-extremal. Note that if $\cE_N$ occurs (in particular if $|C'|\le L^4$), then the facts $|C_1'|\ge \beta n/2$ and $\mathbb{G}_{C_1'\cup C'}[V_1]$ is not $2\alpha'$-extremal imply that $\mathbb{G}_{C_1}[V_1]$ is not $\alpha'$-extremal. The above argument ensures that the generated quadruple $(\cV,\cC,M_{abs},C_{abs})$ belongs to $\cS(\alpha,\alpha',\beta,\delta,\epsilon,N,\mathbb{G})$ with probability at least $1/2$. 

We define $\mathcal{D}_{VC}$ to be the distribution on elements of $(\cV,\cC,M_{abs},C_{abs})\in \cS(\alpha,\alpha',\beta,\delta,\epsilon,N,\mathbb{G})$ such that $\Pr_{\mathcal{D}_{VC}}(\cV,\cC,M_{abs},C_{abs})$ equals the probability that $(\cV,\cC,M_{abs},C_{abs})$ is output by the above procedure divided by the probability that it outputs an element of $\cS(\alpha,\alpha',\beta,\delta,\epsilon,\mathbb{G})$. The latter probability is at least $1/2$ and corresponds to the leading constant $2$ in \eqref{eq:partitions}. Now, a vertex $v$ belongs to $V_N\cup V(M_{abs})$ only if it belongs to $V'$. Thus, each vertex belongs to $V_N\cup V(M_{abs})$ with probability at most $p\le L^4/n$ and to $V_i$ with probability at most $p_i' \le 2|V_i|/n$ independently from the rest of the vertices and the partition $\mathcal{C}$ (the inequality $p_i' \le 2|V_i|/n$ follows from $(\cV,\cC,M_{abs},C_{abs})\in \cS(\alpha,\alpha',\beta,\delta,\epsilon,\mathbb{G})$ - see \Cref{def:vortex}~(i)). Similarly, each color belongs to $C_N\cup C_{abs}$ with probability at most $p\le L^4/n$ and to $C_i$ with probability at most $q_i' \le 2|C_i|/n$ independently from the rest of the colors and the partition $\mathcal{V}$. The remaining terms in \eqref{eq:partitions} follow from the above independence. This finishes the proof of \Cref{lem:vortex}.
\end{proof}

%%%%%%%%%%%%%%%%%%%%%%%%%%%%%%%%%%%%%%%%%%%%%%%%%%%%%%%
\section{Proof of the non-extremal case}

In this section, we prove \Cref{thm:main_spread,thm:main_distinct_packing} when $\mathbb{G}$ is a non-extremal family. We start by describing and analyzing the cover-down step. 

\subsection{The cover-down step} 
We will construct an $O(1/n^2)$-spread distribution on a disjoint union of paths. We will find these paths by sampling a perfect matching in an auxiliary bipartite graph.

\begin{definition}\label{def:bipartite_auxiliary_graph}
Let  $\mathbb{G}=\{G_1,\ldots,G_{m}\}$ be a graph collection with $|V(\mathbb{G})|\le m$. For an injection $\sigma$ from $V(\mathbb{G})$ to $[m]$, we define $B_\sigma$ to be the bipartite graph on $X\cup Y$, where both $X,Y$ are copies of $V(\mathbb{G})$, and for $x\in X$, $y\in Y$ the edge $xy$ belongs to $B_\sigma$ if and only if $xy$ belongs to $G_{\sigma(x)}$. 
\end{definition}
In the above definition, $\sigma$ determines the ``color" of the edges incident to a given vertex $v\in X\subseteq V(B_\sigma)$. In addition, since $\sigma$ is an injection, every color in $[m]$ appears on the edges incident to at most one vertex of~$X$. If we let $|V(\mathbb{G})|=n$, then $B_{\sigma}$ is a bipartite graph on $2n$ vertices and a matching of $B_{\sigma}$ of size $(1-\alpha)n$ corresponds to a $\mathbb{G}$-rainbow collection of paths and cycles that span in total $(1-\alpha)n$ edges.

Starting with a non $\alpha$-extremal family we first show that for a uniformly random injection $\sigma$, the graph $B_\sigma$ inherits a non-extremal feature. Namely, there is no small edge-cut that bipartitions the vertex set (we call such bipartite graphs non-extremal). This is done in \cref{lem:sample_bijection}. We later sample a large matching of $B_{\sigma}$.
\begin{definition}(Extremal bipartite graphs)
Let $\alpha>0$. We say that a bipartite graph $G$, with bipartition $X\cup Y$, is $\alpha$-extremal if there exist half-sets $A\subseteq X$ and $B\subseteq Y$ such that $e(A,B)\le \alpha |X||Y|$. 
\end{definition}

\begin{lemma}\label{lem:sample_bijection} 
Let $0<1/n \ll \epsilon  \ll  \alpha \ll 1$ and $\mathbb{G}=\{G_1,\ldots,G_{m}\}$ be a graph collection with $|V(\mathbb{G})|=n\le m\le 2n$. Suppose that $\mathbb{G}$ is a non $\alpha$-extremal family of graphs with minimum degree at least $(1/2 -\epsilon)n$. If $\sigma$ is an injection from $V(\mathbb{G})$ to $[m]$ chosen uniformly at random, then whp $B_\sigma$ is not $(\alpha^7/100)$-extremal.
\end{lemma}
\begin{proof}
Let $V:=V(\mathbb{G})$ and $C:=[m]$ denote the index set of $\mathbb{G}$. For a half-set $T\subseteq  V$ and a vertex $v\in V$, let $D(v,T)$ be the set of colors $c\in C$ such that $v$ has at most $\alpha^7 n/10$ neighbors in $T$ in $G_c$. For a half-set $T\subseteq  V$, let $D(T)$ be the set of vertices $v\in V$ such that $|D(v,T)|\ge n/10^4$. Call a half-set $T\subseteq V$ \textit{dangerous} if $|D(T)|\ge n/10^4$. We first deal with the non-dangerous sets.

We denote by $X\cup Y$ the bipartition of the vertex set of $B_\sigma$, as given by \Cref{def:bipartite_auxiliary_graph}.

\begin{claim}\label{claim:non_dangerous}
Whp there do not exist half-sets $S\subseteq X$ and $T\subseteq Y$ such that $T$ is not dangerous and $e_{B_\sigma}(S,T)\le \alpha^7 n^2/100$.
\end{claim}
\begin{proof}
Let $T\subseteq V$ be a non-dangerous half-set. Fix a set $\{v_1,\ldots,v_{9n/10}\}\subseteq V$. Since $T$ is not dangerous, among $\{v_1,\ldots,v_{9n/10}\}$ there exist at least $8n/10$ vertices $v_i$ such that $|D(v_i,T)|\le n/10^{4}$. In the event that for at least $6n/10$ of these vertices $v_i$ we have that $\sigma(v_i)\notin D(v_i,T)$, then as every half-set $S$ contains at least $n/10$ of these vertices, we have that $e_{B_\sigma}(S,T) > (\alpha^7 n/10)\cdot n/10 = \alpha^7 n^2/100$. 

Now, for each vertex $v_i$ with $i\in [9n/10]$ such that $|D(v_i,T)|\le n/10^{4}$, we have that $\sigma(v_i)\in D(v_i,T)$ with probability at most $10^{-4}n/10^{-1}n= 1/1000$, independently from $\sigma(v_1),\dots \sigma(v_{i-1})$. Thus, the probability that there exist a half-sets $S\subseteq V$, $T\subseteq Y$ so that $T$ is not dangerous and $e_{B_\sigma}(S,T)\le \alpha^7 n^2/100$ is at most
\[
2^n \binom{8n/10}{2n/10} \bfrac{1}{1000}^{2n/10}\le 2^n \bfrac{8en/10}{2n/10}^{2n/10} \bfrac{1}{1000}^{2n/10} =\bfrac{2^5 \cdot 4e }{1000}^{2n/10}=o(1).
\]
\end{proof}
We let $\mathcal{T}$ be a maximal set of dangerous sets so that $|T\cap T'|\le (1/2-\alpha^7/2)n$ for every distinct $T,T'\in \mathcal{T}$.
\begin{claim}\label{claim:reduce number}
Suppose that for every $T\in \mathcal{T}$ and half-set $S\subseteq X$ we have that $e_{B_\sigma}(S,T)\ge \alpha^7 n^2 $. Then, $e_{B_\sigma}(S,T)\ge \alpha^7 n^2/2$ for every dangerous half-set $T\subseteq Y$ and half-set $S\subseteq X$. 
\end{claim}
\begin{proof}
Let $T$ be a dangerous half-set. Then, there exist $T'\in \mathcal{T}$ such that $|T\cap T'|\ge (1/2-\alpha^7/2)n$. Thus, under the given assumption, for every half-set $S\subseteq Y$ we have that
\[
e_{B_\sigma}(S,T)\ge e_{B_\sigma}(S,T')- |S||T'\setminus T|\ge \alpha^7 n^2 - n(\alpha^7 n/2) \ge  \alpha^7 n^2 /2.
\]
\end{proof}
We now bound the size of $\mathcal{T}$.
\begin{claim}\label{claim:t is small}
$|\mathcal{T}| \le 2\cdot 10^{8}$.
\end{claim}
\begin{proof}
For $v\in V$, $c\in C$, and half-sets $T,T'$ such that $|T\cap T'|\le (1/2-\alpha^7/2)n$, since $d_{G_c}(v)\ge (1/2-\epsilon)n$, we have that $v$ has at least $(\alpha^7/2-\epsilon)n/2 > \alpha^7 n/10$ neighbors in one of the sets $T,T'$ in $G_c$. Thus, there exists at most one set $T\in \mathcal{T}$ such that $v$ has at most $\alpha^7 n/10$ neighbors in $T$ in $G_c$. On the other hand, by definition, for each element of $\mathcal{T}$ there exist at least $10^{-8}n^2$ such vertex-color pairs. Thus $|\mathcal{T}| \cdot 10^{-8}n^2 \le nm\le 2 n^2$, giving that $|\mathcal{T}| \le 2\cdot 10^{8}$.
\end{proof}
We finally prove that whp no set in $\mathcal{T}$ induces a small cut in $B_\sigma$ whp. 
\begin{claim}\label{claim:dangerous}
Whp there do not exist dangerous half-set $T\subseteq Y$ and half-set $S\subseteq X$ such that ${e_{B_\sigma}(S,T)< \alpha^7 n^2/2}$.
\end{claim}
\begin{proof}
We will show that for fixed $T\in \mathcal{T}$, whp we have that $e_{B_\sigma}(S,T)\ge \alpha^7 n^2$ for every half-set $S\subseteq X$. Then, the union bound and \cref{claim:t is small} imply that whp for every $T\in \mathcal{T}$ and half-set $S\subseteq X$, we have that $e_{B_\sigma}(S,T)\ge \alpha^7 n^2$. Then, \cref{claim:reduce number} implies the statement of this claim.  

Fix $T \in \mathcal{T}$. For $c\in C$, let $V_c$ be the set of vertices $v\in V$ such that $d_{G_c}(v,T)\ge \alpha^4 n$. Let $B$ be the set of colors $c\in C$ such that $|V_c|\le (1/2 + 3\alpha^2)n$. Note that for every color $b\in B$, by possibly extending the set of vertices $V\setminus V_c$, there exists a half-set $S\subseteq V$ such that 
\[
e_{G_c}(T,S)\le 3\alpha^2 n \cdot n/2 + n/2 \cdot a^4n\le 2\alpha^2 n^2.
\]
Thus, \Cref{lem:KLS general} implies that $\min\{e_{G_c}(T),e_{G_c}(T,\overline{T}) \}\le 6\alpha^2 n^2$ for every $c\in B$. Since $\mathbb{G}$ is not $\alpha$ extremal we have that
\[
\alpha n^2m \le \sum_{G\in \mathbb{G}} \min\{e_{G_c}(T),e_{G_c}(T,\overline{T})\}
\le 6\alpha^2 n^2\cdot |B|+(n^2/2)(m-|B|).
\]
Thus $|B|\le (1-\alpha)m$. For every $c\notin B$, we have that $|V_c|> (1/2 + 3\alpha^2) n$. On the other hand, for $c\in B$, since $G_c$ has minimum degree at least $(1/2-\epsilon)n$, we have that $|V_c|\ge ((1/2-\epsilon)n|T|-\alpha^4 n^2)/|T|\ge (1/2-\epsilon-3\alpha^4)n$. It follows that
\begin{align*}
\sum_{c\in C} |V_c|&\ge (1/2 + 3\alpha^2) n \cdot \alpha m+ (1/2-\epsilon-3\alpha^4)n \cdot (1-\alpha)m 
\\&= nm/2 +3\alpha^3 nm -(\epsilon+3\alpha^4) (1-\alpha)nm \ge  nm/2 +2\alpha^3 nm.
\end{align*}

We can generate $\sigma$ by first choosing a set $C'\subseteq C = [m]$ of size $n$ and then a bijection from $V$ to $C'$, both choices done uniformly at random. \Cref{lem:concentration_permutation} (by setting $[N]$ as the color set $[m]$ and defining $C'$ to be the image of $[n]$ under the random bijection from $[m]$ to $[m]$) and the inequality $m\le 2n$ imply that whp $\sum_{c\in C'} |V_c|\ge n^2/2 +3\alpha^3 n^2/2$.

Now, let $Z$ be the number of colors $c\in C'$ that are mapped by $\sigma^{-1}$ to an element of $V_c$. Then, 
\[
\mathbb{E}\left(Z \; \bigg| \; \sum_{c\in C'} |V_c|\ge n^2/2 +3\alpha^3 n^2/2 \right)\ge (1/2+3\alpha^3/2)n.
\]
\Cref{lem:concentration_permutation} gives that $Z\ge (1/2+\alpha^3)n$ whp. In the event that $Z\ge (1/2+\alpha^3)n$ we have that every half-set $S$ contains at least $\alpha^3 n$ vertices $v$ such that $d_{G_{\sigma(v)}}(v,T)\ge \alpha^4 n$. Thus $e_{B_\sigma}(S,T) \ge \alpha^3 n \cdot \alpha^4 n = \alpha^7 n^2$, completing the proof.
\end{proof}
Claims~\ref{claim:non_dangerous} and~\ref{claim:dangerous} give that whp $B_\sigma$ is not $(\alpha^7/100)$-extremal. This completes the proof of \cref{lem:sample_bijection}. 
\end{proof}

\begin{definition}[Cover-down structure]
Given two disjoint vertex sets $U,V$ and a matching $M$ on $U$, we say that $\cP$ is a $(U,M)$-internal $V$-external path covering, $(U,V,M)$-PC for short, if $\cP$ is a collection of vertex-disjoint paths whose union covers both $U$ and $M$, and for $P\in \cP$ the endpoints of $P$ lie in $V$ and the rest of its vertices in $U$. Furthermore, given a graph collection $\mathbb{G}=\{G_1,\ldots,G_{m}\}$ on $U\cup V$, we define a $(U,V,M,\mathbb{G})$-PC to be a $(U,V,M)$-PC such that the set of edges of the underlying paths that do not belong to $M$ is $\mathbb{G}$-transversal. 
\end{definition}

We will construct a spread measure on $(U,V,M,\mathbb{G})$-PCs in the following lemma. 
\begin{lemma}[Cover-down lemma] \label{lem:cover_down}
Let $0< 1/n  \ll \epsilon, \delta, \gamma \ll  \alpha \ll 1$. Let $C$ be a color set which we identify with $[|C|]$. Let $U,V$ be disjoint sets and $\mathbb{G}=\{G_1,\ldots,G_{|C|}\}$ be a graph family on $U\cup V$ such that 
\begin{itemize}
    \item[(i)] $\mathbb{G}[U]$ is not $\alpha$-extremal and every graph in $\mathbb{G}[U]$ has minimum degree at least $(1/2-\epsilon)|U|$,
    \item[(ii)] $d_{G_i}(u,V)\ge |V|/3$ for all $u\in U$ and $i\in [|C|]$.
\end{itemize}
In addition, let $M$ be a matching on $U$ of size at most $\gamma|U|$. Assume that $|U|=n $, $|V|\le |U|$, and $|C|-|U|+|M|=\delta |U| \le |V|/8$. Then there exists a $(100\delta^{-1}/|V|^2)$-spread distribution $\mathcal{D}_2$ on the $\mathbb{G}$-transversal set of edges $E$ with the property that $E\cup M$ spans a $(U,V,M,\mathbb{G})$-PC.
\end{lemma}
\begin{proof}
We first construct the distribution $\mathcal{D}_2$ and then show that it is $(100\delta^{-1}/|V|^2)$-spread. 

Let $U':= U \setminus V(M)$. First, let $\sigma'$ be an injection from $U'$ to $C$ chosen uniformly at random and let $C'=\sigma'(U')$ be its image. Let $U''\subseteq V(M)$ be such that for each edge $e$ in $M$ the set $U''$ contains exactly one of the endpoints of $e$. Let  $\sigma''$ be an injection from $U''$ to $C\setminus C'$, chosen uniformly at random. Also set $C''=\sigma''(U'')$. Concatenate the maps $\sigma'$ and $\sigma''$ to the map $\sigma$ from $U'\cup U''$ to $C$.

Let $\tilde{B}_\sigma$ be the auxiliary graph on $X\cup Y$ where $X$ is a copy of $U'\cup U''$, $Y$ is a copy of $U\setminus U''$, and for every $x\in X,y\in Y$ the edge $xy$ belongs to $E(\tilde{B}_\sigma)$ if and only if $xy$ belongs to $G_{\sigma(x)}$. We can think that for $x\in X$, every edge incident to $x$ in $\tilde{B}_\sigma$ has color $\sigma(x)$. Let $p=\delta^{-1}/n$ and let $(\tilde{B}_{\sigma})_p$ be the random subgraph of $\tilde{B}_\sigma$ obtained by retaining each edge with probability $p$ independently. Let $M'$ be a maximum matching of $(\tilde{B}_{\sigma})_p$. Note that $M'$ corresponds to a $\mathbb{G}$-rainbow set of edges $E'$ with the property that every vertex in $U$ is incident to at most 2 edges in $E'\cup M$. Let $\mathcal{E}$ be the event that $M'$ has size at least $(1-\delta/2)|U| - |M|$ and $E'\cup M$ spans at most $\delta|U|/2$ many cycles.

Throughout this paragraph, assume the event $\mathcal{E}$ holds. Then, the number of connected components of the graph on $U$ induced by $E'\cup M$ is exactly $|U| - |M| - |M'| + (\text{\# cycles spanned by $E'\cup M$}) \le \delta |U|$, where the inequality uses that $\mathcal{E}$ holds. Now, by removing an edge from each cycle spanned by $E'\cup M$ that does not belong to $M$, and potentially a small number of additional edges not in $M$, we obtain a set of edges $E''\subseteq E'$ such that $E''\cup M$ induces a linear forest $\mathcal{P}'$ of $\delta|U|$ many paths (here we consider a vertex to be a path of length~$0$) where $\mathcal{P}'$ spans $U$. Let $C'''\subseteq C$ be the set of colors not appearing on $E''$. Then, $|C'''| = |C| - |E(\mathcal{P}')| = |C| - (|U| - |M| - \delta |U|) = 2\delta |U|$, where the last equality follows from an assumption of \Cref{lem:cover_down}. Let $U'''$ be the multiset of $2\delta |U|$ endpoints of paths in $\mathcal{P}'$ (for each path $v\in \mathcal{P}'$ of length $0$, there are two copies of $v$ in $U'''$). Let $\sigma'''$ be a bijection from $U'''$ to $C'''$ chosen uniformly at random. For every $u\in U'''$, we will now (randomly) choose an edge from $G_{\sigma'''(u)}$ between $u$ and $V$ so that no two of these edges are incident to a vertex in $V$. Formally, set $E'''=\emptyset$ initially and let $U''' = \{u_1,\ldots,u_\ell\}$. Sequentially, for every $i\in [\ell]$, choose an edge of $G_{\sigma'''(u_i)}$ from $u_i$ to $V$ that is not incident to any edge in $E'''$ (constructed so far) uniformly at random, and add it to $E'''$. With the final set $E'''$ obtained this way, let $E=E''\cup E'''$. By construction, in the event $\mathcal{E}$, the set $E\cup M$ spans a $(U,V,M,\mathbb{G})$-PC.

We let $\mathcal{D}_2$ be the distribution on $\mathbb{G}$-rainbow sets of edges so that the probability that $\mathcal{D}_2$ assigns to such a set $F$ equals the probability that the above procedure terminates with $E=F$ conditioned on the event $\mathcal{E}$, divided by the probability that $\mathcal{E}$ occurs. 
\begin{claim}
Suppose that the event $\mathcal{E}$ occurs with probability at least $1/2$. Then $\mathcal{D}_2$ is $(100\delta^{-1}/|V|^2)$-spread.  
\end{claim}
\begin{proof}
Let $F$ be a $\mathbb{G}$-rainbow set of edges with colors given by the injection $c:F \to C$. Let $F''\subseteq F$ denote the set of all edges with both endpoints in $U$, and $F'''\subseteq F$ denote the set of all edges with one endpoint in $U$ and the other in $V$. If the event $F\subseteq E$ holds, then we must have $F = F''\cup F'''$ (which we assume from now on without loss of generality) and $F''\subseteq E''$ and $F'''\subseteq E'''$. These happen only if the following occur:
\begin{itemize}
    \item[($\mathcal{E}_1$)] for every $e\in F''$, $\sigma$ maps one of the endpoints of $e$ to the color $c(e)$;
    \item[($\mathcal{E}_2$)] for every $e\in F''$, $e$ survives the sparsification step, i.e. it belongs to $(\tilde{B}_{\sigma})_p$;
    \item[($\mathcal{E}_3$)] for every $e\in F'''$, $\sigma'''$ maps one of the endpoints of $e$ to the color $c(e)$;
    \item[($\mathcal{E}_4$)] for every $e\in F'''$, $e$ is chosen (randomly) incident to one of its endpoints. 
\end{itemize} 
Recall that $\sigma$ is distributed uniformly at random from $U'\cup U''$ to $C$. Let $U'\cup U'' = \{u_1,\ldots,u_{|U'\cup U''|}\}$. Then one may generate  $\sigma$ by choosing $\sigma(u_i)$ uniformly at random from $(U'\cup U'')\setminus \{\sigma(u_1),\ldots,\sigma(u_{i-1})\}$ for every $i\in [|U'\cup U''|]$. It follows that for $\{u_1,\ldots,u_k\}\subseteq U'\cup U''$ and $\{c_1,\ldots,c_k\}\subseteq C$, we have 
\[
\Pr(\sigma(u_i) = c_i \text{ for every }i\in[k])= \prod_{i=0}^{k-1}\frac{1}{|C|-i} \le \bfrac{1}{|C|!}^{k/C}\le \bfrac{3}{|C|}^k.
\]
Thus, since there are $2^{|F''|}$ ways to specify the endpoints in the event $\mathcal{E}_1$, we have that $\Pr(\mathcal{E}_1)\le (6/|C|)^{|F''|}$. Similarly, $\Pr(\mathcal{E}_3)\le (6/2\delta|U|)^{|F'''|}$. It is clear that $\Pr(\mathcal{E}_2)\le p^{|F''|}$. For the event $\mathcal{E}_4$, note that item (ii) and the inequality $\delta|U|\le |V|/8$, imply that for every vertex $u\in U'''$, given the color $\sigma'''(u)$, there are at least $|V|/3 -2\delta |U| \ge |V|/12$ many choices for choosing the edge of color $\sigma'''(u)$ from $u$ to $V$. Therefore, 
\begin{align*}
\frac{\Pr(F\subseteq E \; | \; \mathcal{E})}{\Pr(\mathcal{E})}&\le 2\cdot \bfrac{6}{|C|}^{|F''|}  p^{|F''|} \bfrac{6}{2\delta|U|}^{|F'''|} \bfrac{12}{|V|}^{|F'''|} 
\\&\le 2\cdot \bfrac{6p}{|C|}^{|F''|} \bfrac{36}{\delta|U||V|}^{|F'''|}\le 2\bfrac{50 \delta^{-1}}{|V|^2}^{|F|}.
\end{align*} 
Hence $\mathcal{D}_2$ is $(100\delta^{-1}/|V|^2)$-spread.
\end{proof}
In order to complete the proof of \Cref{lem:cover_down}, by the previous claim, it remains to prove that the event $\mathcal{E}$ occurs with probability at least $1/2$. This follows from the next couple of claims. 

\begin{claim}\label{claim:perfect matching}
$(\tilde{B}_\sigma)_p$ spans a matching of size at least $(1-\delta/2)|U| - |M|$ with probability at least $3/4$.    
\end{claim}

\begin{claim}\label{claim:few cycles}
$E'\cup M$ spans fewer than $\delta |U|/2$ cycles with probability at least $3/4$. 
\end{claim}

\begin{proof}[{\bf Proof of \Cref{claim:perfect matching}}]
Using (i) and $|V(M)|\le 2\gamma |U|$, we deduce that every graph in $\mathbb{G}[U']$ has minimum degree at least $(1/2-\epsilon-2\gamma)|U|\ge (1/2-\epsilon-2\gamma)|U'|$. Furthermore, $\mathbb{G}[U']$ is not $(\alpha/2)$-extremal. Indeed, for any half-set $A'\subseteq U'$, there exists a half-set $A\subseteq U$ such that $|A\triangle A'|\le |U\triangle U'| = |U\setminus U'| \le 2\gamma|U|$, and therefore
\begin{align*}
\sum_{c\in C}\min\{e_{G_c}(A'), e_{G_c}(A',U'\setminus A')\}&\ge \sum_{c\in C}\min\{e_{G_c}(A), e_{G_c}(A,U\setminus A)\} - |U\setminus U'||U||C|
\\&\ge \alpha |C||U|^2 -2\gamma |C||U|^2 > \alpha |C||U|^2/2.  
\end{align*}

Let $\cE'$ be the event that $B_{\sigma'}$ is not $((\alpha/2)^7/100)$-extremal (with $B_{\sigma'}$ defined as in \Cref{def:bipartite_auxiliary_graph}). \Cref{lem:sample_bijection} implies that $\cE'$ occurs with probability at least $99/100$. We claim that if the event $\cE'$ holds, then the graph $\tilde{B}_\sigma$ is not $(\alpha/10)^7$-extremal. Indeed, for each pair of half-sets $S\subseteq X$ and $T\subseteq Y$ there exists a pair of sets $S'\subseteq X$ and $T'\subseteq Y$ so that $|S\triangle S'|, |T\triangle T'|\le \gamma |U|$ and $S',T'$ are half-sets of $U'$. Thus,
\begin{align*}
e_{\tilde{B}_\sigma}(S,T) &\ge e_{B_{\sigma'}}(S',T')-(|S\triangle S'|+|T\triangle T'|)|U|
\\&\ge ((\alpha/2)^7/100)|U'|^2 -2\gamma |U|^2 > (\alpha/10)^7 |U|^2 \ge (\alpha/10)^7 |X||Y|.  
\end{align*}

Note that (i) implies that every vertex in $X$ has degree at least $(1/2-\epsilon-\gamma)n$ in $\tilde{B}_\sigma$. Let $\cE''$ be the event that $\tilde{B}_\sigma$ has minimum degree at least $(1/2-\epsilon-2\gamma)n$. \Cref{lem:concentration_permutation} implies that with probability at least~$9/10$ every vertex in $Y$ has degree at least $(1/2-\epsilon-2\gamma)n$ in $\tilde{B}_\sigma$. Thus $\cE''$ occurs with probability at least~$9/10$. 

Recall that $|U|=n$ and $M'$ is a maximum matching of $(\tilde{B}_\sigma)_p$. For every $k\in [n-|M|]$, let $\cE_k$ be the event that there exists a pair of such sets $S\subseteq X, T\subseteq Y$ with $|S|=k$, $|T|=k - \delta  n/2$, and there is no edge from $S$ to $Y\setminus T$ in $(\tilde{B}_\sigma)_p$ (thus $N_{(\tilde{B}_\sigma)_p}(S)\subseteq T$). Then, since $|X|=|Y|=n-|M|$, Hall's theorem implies that if $M'$ has size smaller than $(1-\delta/2)n-|M|$, then there exists $\delta n/2 \le k\le n-|M|$ such that the event $\cE_k$ occurs. 

We first deal with the case $\delta n/2 \le k\le (1/2- \epsilon -\gamma-10\delta^{1/2})n$. For such $k$, using the fact that $p=\delta^{-1}/n$ and every vertex in $X$ has degree at least $(1/2-\epsilon-\gamma)n$ in $\tilde{B}_\sigma$, we have
\begin{align*}
    \Pr(\cE_k) &\le \binom{n-|M|}{k}\binom{n-|M|}{k-\delta n/2} (1-p)^{k((1/2-\epsilon-\gamma)n-k)} \le \binom{n}{k}^2 e^{-10\delta^{1/2} np k}
    \\&\le \bfrac{en}{k}^{2k} \left(e^{-10\delta^{1/2}np}\right)^k \le \left(4e^2\delta^{-2} e^{-10\delta^{-1/2}}\right)^k =o(1/n). 
\end{align*}
Thereafter, for $(1/2- \epsilon -\gamma-10\delta^{1/2})n\le k \le (1/2+\epsilon+2\gamma + 10 \delta^{1/2})n$, note that in the event $\cE'$, for $S\subseteq X$ and $T\subseteq Y$ such that $|S|=k$ and $|T|=k-\delta n/2$, using that $\epsilon,\delta,\gamma\ll \alpha$, we have that $e_{\tilde{B}_\sigma}(S,T)\ge (\alpha/10)^7 n^2/2$. Thus, 
\begin{equation*}
    \Pr(\cE_k \; | \; \cE') \le \binom{n-|M|}{k}\binom{n-|M|}{k-\delta n/2} (1-p)^{(\alpha/10)^7 n^2/2} \le 2^n\cdot 2^n \cdot e^{-p (\alpha/10)^7 n^2/2} =o(1/n),
\end{equation*}
where the last equality follows from $(pn)^{-1}=\delta \ll \alpha$.

Finally, for $(1/2+\epsilon+ 2\gamma + 10 \delta^{1/2})n \le k\le n-|M|$, note that in the event $\cE''$ the graph $\tilde{B}_\sigma$ has minimum degree at least $(1/2-\epsilon-2\gamma)n$. Hence, for $S\subseteq X$ and $T\subseteq Y$ such that $|S|=k$ and $|T|=k-\delta n/2$, the number of edges in $\tilde{B}_\sigma$ from $Y\setminus T$ to $S$ is at least 
\[
|Y\setminus T|(\delta(\tilde{B}_\sigma) - |X\setminus S|) \ge ((n-|M|)-(k-\delta n/2))((1/2-\epsilon-2\gamma)n-(n-k)) \ge (n-|M|-k+\delta n/2) \cdot 10 \delta^{1/2}n.
\]
Thus, we have
\begin{align*}
    \Pr(\cE_k \; | \; \cE'') &\le \binom{n-|M|}{k}\binom{n-|M|}{k-\delta n/2} (1-p)^{10 \delta^{1/2}n \cdot (n-|M|-k+\delta n/2)}
    \\&\le \binom{n-|M|}{n-|M|-k+\delta n/2}^2(1-p)^{10 \delta^{1/2}n \cdot (n-|M|-k+\delta n/2)}
    \\&\le \left(\bfrac{e(n-|M|)}{n-|M|-k+\delta n/2}^2 e^{-p\cdot 10\delta^{1/2}n}\right)^{n-|M|-k+\delta n/2}
    \\&\le \left( \bfrac{en}{\delta n/2}^2 e^{-10\delta^{-1/2}}\right)^{\delta n/2} =o(1/n). 
\end{align*}
Now, by a union bound over $k$ between $\delta n/2$ and $n-|M|$, whp it holds that conditioned on the events $\cE'$ and $\cE''$, none of the events $\cE_k$ occur. It follows that $(\tilde{B}_\sigma)_p$ contains a matching of size at least $(1 -\delta /2)n - |M|$ with probability at least $1-o(1)-\Pr(\cE')-\Pr(\cE'')\ge 3/4$, as desired.    
\end{proof}

\begin{proof}[{\bf Proof of \Cref{claim:few cycles}}]
A cycle $C$ of length $k$ that is spanned by $E'\cup M$ can be described by choosing $\ell\le k/2$ vertices $u_1,\ldots,u_\ell \in V(M)$, $k-2\ell$ vertices $v_1,v_2,\ldots,v_{k-2\ell} \in U\setminus V(M)$, and choosing a cyclic permutation of these vertices $w_1,w_2,\ldots,w_{k-\ell},w_{k-\ell+1}=w_1$ such that the edges are as follows
\[
E(C)=\{w_iw_{i+1}: w_i\notin V(M)\} \cup \{w_ix:w_i\in V(M),w_ix\in M\} \cup \{xw_{i+1}:w_i\in V(M),w_ix\in M\}.
\]
The edges of such a cycle belong to $E'\cup M$ with probability at most $p^{k-\ell}$. Thus the expected number of cycles of length at most $8\delta^{-1}$ spanned by $E'\cup M$ is at most
\[
\sum_{k=1}^{8\delta^{-1}}\sum_{\ell=1}^{k/2} |U|^{k-\ell}p^{k-\ell}=O(1).
\]
Thus Markov's inequality gives that $E'\cup M$ spans at most $\delta n/8$ cycles of length at most $8\delta^{-1}$ with probability at least $3/4$. In this event $E'\cup M$ spans at most $\delta n/8 +  n/(8\delta^{-1}) < \delta n/2$ cycles in total, as desired.
\end{proof}
This finishes the proof of the cover-down lemma, i.e., \Cref{lem:cover_down}.
\end{proof}

%%%%%%%%%%%%%%%%%%%%%%%%%%%%%%%%%%%%%%%%%%%%%%%%%%%%%%%%%%%%%%%%%%%%%%%%%%%
\subsection{Putting everything together} \label{subsec:puttingeverythingtogether}
In this final subsection, we prove the following. 
\begin{theorem}\label{thm:main_spread_non_extremal}
Let $0 <1/n \ll 1/c \ll \alpha \le 1$ and $0\le \eta \ll \alpha$. For every non $\alpha$-extremal collection $\mathbb{G}=\{G_1,\ldots,G_n\}$ of graphs on $n$ vertices with minimum degree at least $(1/2-\eta)n$, there is a $(c/n^2)$-spread probability measure on the set of Hamilton $\mathbb{G}$-transversals.
\end{theorem}

\begin{proof}
Note that we can (and do) assume $1/c \ll \eta$ without loss of generality. Let $L, \delta, \epsilon, \beta, \alpha'$ be such that $0<1/n \ll 1/c \ll 1/L, \eta \ll \epsilon \ll \delta, \beta \ll \alpha' \ll \alpha \le 1$. Also, let $N=N(\beta,n,L)$ be the minimum integer such that $\beta^{N-1} n \in [L, 2\beta^{-1} L]$. Sample an element $(\cV,\cC,M_{abs},C_{abs})$ of  $\cS(\alpha,\alpha',\beta,\delta,\epsilon, N,\mathbb{G})$ according to the distribution $\mathcal{D}_{\cV\cC}$ given at the statement of \Cref{lem:vortex}. We will now note a few useful equalities involving the sizes of the sets in $\cV = (V_1,\ldots,V_N)$ and $\cC = (C_1,\ldots,C_N)$. 
Observe that $\sum_{j=1}^{N}|V_j|=n=\sum_{j=1}^{N}|C_j|$, thus for every $i\in [N-1]$, we have $\sum_{j=1}^{i}(|C_j|-|V_j|)=\sum_{j=i+1}^{N}(|V_j|-|C_j|)$.
Using this, for every $i\in [N-1]$, we have
\begin{align}
\sum_{j=1}^{i}(|C_j|-|V_j|)+(|M_{abs}|-|C_{abs}|)
&= \sum_{j=i+1}^{N}(|V_j|-|C_j|)+(|M_{abs}|-|C_{abs}|) \label{eq:simple relation}
\\&= \sum_{j=i+1}^{N-1} (\beta^3\pm \beta^4)|V_j| + \beta^3 |V_N| \nonumber
\\&= (\beta^3\pm 3\beta^4)|V_{i+1}| = (\beta^4\pm 4\beta^5)|V_i|. \label{eq:basic}
\end{align}
In the above, the second equality uses the following two facts. The first one is that ${|V_j|-|C_j|= (\beta^3\pm \beta^4)|V_j|}$ whenever $2\le j\le N-1$ (this holds because $(\cV,\cC,M_{abs},C_{abs})$ belongs to $\cS(\alpha,\alpha',\beta,\delta,\epsilon,N,\mathbb{G})$, see \cref{def:vortex}). 
The second one is that $(|V_N|-|C_N|)+(|M_{abs}|-|C_{abs}|)=\beta^3|V_N|$ which follows from the fact that the pair $(M_{abs},C_{abs})$ is $(V_N,C_N)$ absorbing (see \Cref{def:absorber}), and thus for every subset $A\subseteq C_N$ of size $(1-\beta^3)|V_N|$, there exists a bijection $\phi:M_{abs}\to C_{abs}\cup (C_N\setminus A)$. The equality in \Cref{eq:basic} follows from the fact that $|V_{i+1}|= (1\pm \beta/10) \beta^i n$ for $i\in [N-1]$ (see \cref{def:vortex}).

For $i\in [N-1]$, we will inductively construct a set $\cP_i$ of paths of positive lengths that is $(\cup_{j\in [i]} V_j, V_{i+1}, M_{abs})$-PC, and has the following property. Let $E_i=\bigcup_{P\in \cP_i}E(P)$, then $E_i\setminus M_{abs}$ is $(\cup_{j\in [i]}C_j)\setminus C_{abs}$-transversal. At the $N$-th step, we will join the paths in $(\cup_{j\in [N-1]} V_j, V_N, M_{abs})$-PC into a single Hamilton cycle $H$ on $V(\mathbb{G})$ using edges with colors in $C_N$. Finally, we will use the color-absorbing property of $(M_{abs},C_{abs})$ to produce a rainbow coloring of $H$. 

To do this, let $\cP_0=M_{abs}$. For each $i\in [N-1]$ and given $\cP_{i-1}$, let $M_i$ be the matching on $V_i$ such that $e\in M_i$ if and only if there exists a path $P\in \cP_{i-1}$ that joins the endpoints of $e$. We will first construct $\cP_1$ given the matching $M_1$ on $V_1$. For that, we can apply \Cref{lem:cover_down} to sample $\cP_1$, a $(V_1,V_2,M_1,\mathbb{G}_{C_1\setminus C_{abs}})$-PC with $(\beta^{-5}/|V_2|^2)$-spread distribution. For this, note that conditions (i) and (ii) of \Cref{lem:cover_down} are satisfied since $(\cV,\cC,M_{abs},C_{abs})$ belongs to $\cS(\alpha,\alpha',\beta,\delta,\epsilon,N,\mathbb{G})$, $|M_1| = |M_{abs}| \le 2L^4$, and $|C_1\setminus C_{abs}|-|V_1|+|M_1|= |C_1| - |V_1| + |M_{abs}| - |C_{abs}| = (\beta^4\pm 4\beta^5)|V_1|\le |V_2|/8$ (by \eqref{eq:basic}). 

Thereafter, for $2\le i \le N$, inductively, we assume that we are given $\cP_{i-1}$ (and thus also $M_i$) with the desired properties, and we will construct $\cP_i$. Note that $|\cP_{i-1}| = |M_i|$. Since $\cP_{i-1}$ is a $(\cup_{j\in [i-1]}V_j,V_i,M_{abs})$-PC, we have that $|E_{i-1}| = |M_i| +\sum_{j=1}^{i-1}|V_j|$. Therefore, since the edge set $E_{i-1}\setminus M_{abs}$ is $(\cup_{j \in [i-1]}C_j)\setminus C_{abs}$-transversal, we have that
\begin{equation}\label{eq:M1}
|M_i|+\sum_{j=1}^{i-1}|V_j|-|M_{abs}|=\sum_{j=1}^{i-1}|C_j|-|C_{abs}|
\text{ thus, }
|M_i|=\sum_{j=1}^{i-1}(|C_j|-|V_j|)+(|M_{abs}|-|C_{abs}|).
\end{equation}
The above equality and \Cref{eq:basic} imply that if $2\le i\le N$, then
\begin{equation}\label{eq:M3}
    |M_i| = (\beta^3\pm 3\beta^4)|V_i|
    \text{\;\; and \;\; }
    |C_i|-|V_i|+|M_i| = (\beta^4\pm 4\beta^5)|V_i|.
\end{equation}

Now if $2\le i\le N-1$, then given the matching $M_i$ on $V_i$, we can apply \Cref{lem:cover_down} to sample $\cP_i'$, a $(V_i,V_{i+1},M_i,\mathbb{G}_{C_i})$-PC with $(\beta^{-5}/|V_{i+1}|^2)$-spread distribution. For this, note that conditions (i) and (ii) of \Cref{lem:cover_down} are satisfied since $(\cV,\cC,M_{abs},C_{abs})$ belongs to $\cS(\alpha,\alpha',\beta,\delta,\epsilon,N,\mathbb{G})$ and the relations in \eqref{eq:M3} holds (noting that $(\beta^4\pm 4\beta^5)|V_i| \le |V_{i+1}|/8$). Substituting every edge of each path in $\cP_i'$ that lies in $M_i$ with the corresponding path in $\cP_{i-1}$ gives the set $\cP_i$ with the desired properties.

In the last step, i.e. $i=N$, construct the matching $M_N$ from $\cP_{N-1}$ as before. Thus, $|M_N|\le 2\beta^3|V_N|$ by \eqref{eq:M3}. Moreover since $(\cV,\cC,M_{abs},C_{abs})$ belongs to $\cS(\alpha,\alpha',\beta,\delta,\epsilon,N,\mathbb{G})$, we have that $|C_N|\ge \delta^{-1} |V_N|$, $\mathbb{G}_{C_N}[V_N]$ is not $\alpha'$-extremal, and $G_i[V_N]$ has minimum degree at least $(1/2-\epsilon)|V_N|$ for each $i\in C_N$ (see \Cref{def:vortex,def:vortexabsorber}). We let $G_N$ be the graph on $V_N$ such that $e\in E(G_N)$ if and only if $|\{i\in C_N:e\in E(G_i)\}|\ge |V_N|$. Therefore, for every vertex $v\in V_N$, we have
\[
\delta(G_N)|C _N|+(|V_N|-\delta(G_N))|V_N|\ge \sum_{i\in C_N} d_{G_i[V_N]}(v)\ge (1/2-\epsilon)|V_N||C_N|.
\]
As $|V_N|\le \delta |C_N|$, we have $(|V_N|-\delta(G_N))|V_N|\le \delta|V_N||C_N|$, thus the above inequality implies that $\delta(G_N)\ge (1/2-\epsilon-\delta)|V_N|$. Similarly, for every half set $A\subseteq V_N$, we have that  
\begin{align*}
    \min\{ e_{G_N}(A),e_{G_N}(A,\bar{A})\} &\ge \frac{\sum_{i\in C_N} \min\{ e_{G_i}(A),e_{G_i}(A,\bar{A})\} - |V_N| \binom{|V_N|}{2} }{|C_N|} \\&\ge \frac{\alpha'|C_N||V_N|^2 - \delta |C_N| |V_N|^2/2 }{|C_N|} >   \alpha'|V_N|^2/2,
\end{align*}
where the second inequality uses that $\mathbb{G}_{C_N}[V_N]$ is not $\alpha'$-extremal and $|C_N|\ge \delta^{-1} |V_N|$. Thus $G_N$ is not $(\alpha'/2)$-extremal. Thus, we can use \Cref{thm:robustdirac} to conclude that there is a Hamilton cycle $H_N$ of $G_N$ covering $M_N$. Replacing each edge of $M_N\subseteq E(H_N)$ with the corresponding paths in $\cP_{N-1}$ gives a Hamilton cycle $H$ with the property that $E(H)\setminus (M_{abs} \cup (E(H_N)\setminus M_N))$ is $[n]\setminus (C_{abs}\cup C_N)$ transversal. We then greedily color $E(H_N)\setminus M_N$ so that it is $C_N$-rainbow. This is possible as there are initially at least $|V_N|$ choices for a color in $C_N$ for each edge in $G_N$. Note that the number of colors in $C_N$ used in coloring the edges in $E(H_N)\setminus M_N$ is precisely $|V_N| - |M_N| = |C_N| - |M_{abs}| + |C_{abs}| = (1-\beta^3)|V_N|$, where the first equality uses \Cref{eq:simple relation,eq:M1}. Thus, we finally use the unused colors in $C_N$ and the colors in $C_{abs}$ to color the edges in $M_{abs}$  in a rainbow fashion (see \cref{def:absorber,def:vortexabsorber}).

By construction, $E(H)$ is the edge set of a rainbow Hamilton cycle $H$ (here, every element in $E(H)$ denotes a colored edge). We now turn our attention to showing that the above procedure produces an $O(1/n^2)$-spread distribution on the collection of all Hamilton $\mathbb{G}$-transversals.
For that, let $F$ be a set of colored edges on the vertex set $V(\mathbb{G})$. Note that we can assume that $F$ is rainbow and it induces a Hamilton cycle or a disjoint union of paths because otherwise, $F$ cannot be a subset of $E(H)$, which immediately implies that  $\Pr(F\subseteq E(H))=0$. First, we partition $F$ into two sets $F_{abs}$ and $F'$. For each such partition, we will consider the event that $F_{abs}\subseteq M_{abs}$ and $F'\subseteq E(H)\setminus M_{abs}$ which we denote by $\cE(F_{abs},F')$. 

Observe that if $|F_{abs}|> 2L^4$ or $F_{abs}$ is not a matching, then the event $\cE(F_{abs},F')$ has probability zero as $|M_{abs}|\le 2L^4$ and $M_{abs}$ is a matching. 
Furthermore if $F'=\emptyset$, then $\cE(F_{abs},F')$ occurs only if $V(F_{abs})\subseteq V(M_{abs})$ and the colors of the edges of $F_{abs}$ belong to $C_{abs}$. Hence, if $F'=\emptyset$, then $\Pr(\cE(F_{abs},F')) \le 2(L^4/n)^{3|F_{abs}|}\le (1/n^2)^{|F|}$ by \cref{lem:vortex}. To this end, fix a partition of $F$ into $F_{abs}, F'$ with $F'\neq \emptyset$ such that ${|F_{abs}|\le 2L^4}$ and $F_{abs}$ is a matching. Suppose the set $F'$ spans a set of vertex-disjoint paths $P_1,\ldots,P_k$. 
Let $\ell_i$ be the length of $P_i$, $P_i=v_1^iv_2^i,\ldots,v_{\ell_i+1}^i$, and $c(i,j)$ be the color of the edge $v_j^iv_{j+1}^i$ for $i\in [k]$, $j\in [\ell_i]$, and $q(i,j)$ be the index $r\in [N]$ satisfying $v_j^i\in V_r$. 

Every edge $e\in E(H)$ needs to satisfy that both endpoints of $e$ must lie in some $V_i\cup V_{i+1}$ for some $i\in [N-1]$. In addition, for every edge $e\in E(H)\setminus M_{abs}$ that lies in $V_i\cup V_{i+1}$ and intersects $V_i$, we must have that $e$ belongs to some path in $\cP_i$ and $e$ does not belong to any path in $\cP_{i-1}$, hence its color must belong to $C_i$. Thus, in the event $F'\subseteq E(H)\setminus M_{abs}$, we can associate to each path $P_i$ the index $q(i)$, and an $\ell_i$-tuple $t_i\in\{-1,0,1\}^{\ell_i}$ such that the following are satisfied. 
\begin{itemize}
\item[(I)] $q(i)=q(i,1)$ for $i\in [k]$; 
\item[(II)] $t_i(j)=q(i,j)-q(i,j+1)$ for $i\in [k], j\in[\ell_i]$.
\end{itemize}
Thus, (I) and (II) determine $q(i,j)$ for $i\in [k]$, $j\in [\ell_i+1]$, i.e. the index $r$ so that the vertex $v_j^i$ belongs to $V_r$. In addition, they determine the index $r$ such that color $c(i,j)$ belongs to $C_r$, this equals $\min\{q(i,j), q(i,j+1)\}$. We denote the pair $(q(i),t_i)$ associated to $P_i$ by $q(P_i)$ and let $q(F')=\{q(P_i)\}_{i\in [k]}$. 

Suppose $\Pr(F'\subseteq E(H))>0$. Then let $\mathcal{Q}(F')$ be the set of all $\{q(i),t_i\}_{i\in [k]}$ for which (I) and (II) may be satisfied. 
For every $Q=\{q(i),t_i\}_{i\in [k]} \in \mathcal{Q}(F')$ and $l\in [N]$, let 
\[
n_l(Q)=|\{(i,j):q(i,j+1)=l,i\in[k],j\in\{1,\ldots,\ell_i\}|,
\]  
\[
m_l(Q)=|\{(i,j):l=\min\{q(i,j), q(i,j+1)\},i\in[k],j\in[\ell_i]\}|.
\]

For sets $F_{abs}$, $F'$ and $Q\in \mathcal{Q}(F')$, we let $\cE(F_{abs}, F',Q)$ be the event that $\cE(F_{abs}, F')$ occurs and $\mathcal{Q}(F')=Q$.

For completeness, if $\Pr(F'\subseteq E(H))=0$, then let $\mathcal{Q}(F')=\emptyset$. 

Thus, in the event $\cE(F_{abs}, F',Q)$, the quantity $n_l(Q)$ counts the number of vertices in $V_l$ that are not the first vertex of a path and are spanned by $F'$, while $m_l(Q)$ counts the number of edges spanned by $F'$ that receive a color from $C_l$ (equivalently, they are spanned by $\cP_l$ but not $\cP_{l-1}$). Then,
\begin{align}\label{eq:0}
    \Pr(F\subseteq E(H))&\le \sum_{F_{abs}\sqcup F'=F } \; \sum_{Q\in \mathcal{Q}(F')}
    \Pr(F'\subseteq E(H)|\cE(F_{abs},F',Q)) \Pr(\cE(F_{abs}, F',Q))
\end{align}
By \Cref{lem:vortex}, the probability of the event $\cE(F_{abs}, F',Q)$ is at most 
\begin{align}\label{eq:1}
2 \bfrac{L^4}{n}^{3|F_{abs}|+n_N(Q)+m_N(Q)}
\prod_{i\in [k]}\bfrac{2|V_{q(i)}|}{n}    
\prod_{i\in [N]}\bfrac{2|V_i|}{n}^{n_i(Q)}
\prod_{i\in [N]}\bfrac{2|C_i|}{n}^{m_i(Q)} \nonumber
\\ \times \bfrac{n}{2|V_1|}^{2|F_{abs}|}
\bfrac{n}{2|V_N|}^{n_N(Q)}
\bfrac{n}{2|C_N|}^{m_N(Q)} 
\end{align}
The term $(n/2|V_1|)^{2|F_{abs}|}$ is due to potentially double-counting vertices that are spanned by both $F_{abs}$ and $F'$ and thus, these vertices belong to $V_1$. We next simplify the above expression using the following observations. First, since $1/n\ll 1/L$, we have $(L^4/{n})^{3|F_{abs}|}\le (1/n)^{2.5|F_{abs}|}$. Second, since  $|V_1|\ge (1-\sum_{i\ge 1}2\beta^i)n\ge n/2$, we have $(n/2|V_1|)^{2|F_{abs}|}\le 1$. Third, 
\[
\bfrac{L^4}{n}^{n_N(Q)+m_N(Q)} 
\bfrac{n}{2|V_N|}^{n_N(Q)} 
\bfrac{n}{2|C_N|}^{m_N(Q)} \le \bfrac{L^4}{2}^{(n_N(Q)+m_N(Q))} \le \frac{L^{L^2}}{2}.
\] 
At the last inequality we used the fact that due to the choice of $N$ and \Cref{def:vortex}, we have $n_N(Q)\le |V_N|\le (1+\beta/10)\beta^{N-1}n\le 3\beta^{-1}L\le L^2/8$ and $m_N(Q)\le |C_N|\le 2\delta^{-1}|V_N|\le (2\delta^{-1})(3\beta^{-1}L)\le L^2/8$. Combining the above gives,
\begin{align}\label{eq:1.1}
2 \bfrac{L^4}{n}^{3|F_{abs}|+n_N(Q)+m_N(Q)}
\bfrac{n}{|V_1|}^{2|F_{abs}|}
\bfrac{n}{2|V_N|}^{n_N(Q)}
\bfrac{n}{2|C_N|}^{m_N(Q)} \le L^{L^2} n^{-2.5|F_{abs}|}.
\end{align}

Recall that $c(i,j) \in \min\{q(i,j), q(i,j+1)\}$ and $|q(i,j)- q(i,j+1)|\le 1$, thus $c(i,j)\in \{q(i,j+1), q(i,j+1)-1\}$. 
It follows that $\sum_{i=1}^l n_i(Q)\le \sum_{i=1}^l m_i(Q)$ for every $\ell\in [N]$. We also have $\sum_{i=1}^N n_i(Q) = \sum_{i=1}^N m_i(Q)$. Thus, since $|V_i|$ is decreasing with respect to $i$, we have that 
\begin{align}\label{eq:1.2}
    \prod_{i\in [N]} \bfrac{2|V_i|}{n}^{n_i(Q)}
    \prod_{i\in [N]} &\bfrac{2|C_i|}{n}^{m_i(Q)} 
    \le \prod_{i\in [N]}\bfrac{2|V_i|}{n}^{m_i(Q)}
    \prod_{i\in [N]}\bfrac{2|C_i|}{n}^{m_i(Q)} \nonumber
    \\&\le (\delta)^{-m_N(Q)} \prod_{i\in [N]}\bfrac{8\beta^{-1}|V_i|}{n}^{2m_i(Q)} \le L^{L^2} \prod_{i\in [N]}\bfrac{8\beta^{-1}|V_i|}{n}^{2m_i(Q)},
\end{align} 
where at the second inequality we used that $|C_N|\le 2\delta^{-1}|V_N|$ and $|C_i|\le 2|V_i|$ for every $i\in [N-1]$, and at the last inequality we used that $1/L \ll \delta$ and $m_N(Q)\le L^2$ as established right before \eqref{eq:1.1}. Now, \eqref{eq:1.1} and \eqref{eq:1.2} give that \eqref{eq:1} can be bounded above by
\begin{equation}\label{eq:2}
L^{2L^2} n^{-2.5|F_{abs}|} 
\prod_{i\in [k]}\bfrac{2|V_{q(i)}|}{n} \prod_{i\in [N]}\bfrac{8\beta^{-1}|V_i|}{n}^{2m_i(Q)}.
\end{equation}

Thereafter, as for every $i\in [N-1]$, each $\cP'_i$ (which contains all edges in $\cP_i$ that are not in $\cP_{i-1}$) is sampled from an $(\beta^{-5}/|V_{i+1}|^2) \le (\beta^{-8}/|V_i|^2)$-spread distribution, we have
\[
\Pr(F'\subseteq E(H)|\cE(F_{abs},F',Q)) \le \prod_{i\in [N-1]}\bfrac{\beta^{-4}}{|V_i|}^{2m_i(Q)} \le L^{L^2} \prod_{i\in [N]}\bfrac{\beta^{-4}}{|V_i|}^{2m_i(Q)},
\]
where we used $|V_N|\le L^2/8$ and $m_N(Q)\le L^2/8$ as established right before \eqref{eq:1.1}. The above inequality together with \eqref{eq:2} yields the following.
\begin{align}\label{eq:3}
\Pr(F'\subseteq E(H)|\cE(F_{abs},F',Q)) \Pr(\cE(F_{abs},F',Q)) 
&\le L^{3L^2} n^{-2.5|F_{abs}|} \prod_{i\in [k]}\bfrac{2|V_{q(i)}|}{n} \bfrac{8\beta^{-5}}{n}^{2|F'|} \nonumber
\\&\le L^{3L^2} \bfrac{8\beta^{-5}}{n}^{2|F|} \prod_{i\in [k]}\bfrac{2|V_{q(i)}|}{n},
\end{align}

Combining \eqref{eq:0} and \eqref{eq:3} give
\begin{align*} 
\Pr(F\subseteq E(H))&\le L^{3L^2} \bfrac{8\beta^{-5}}{n}^{2|F|} 
\sum_{F_{abs}\sqcup F'=F } \; \sum_{Q\in \mathcal{Q}(F')}
\prod_{i\in [k]}\bfrac{2|V_{q(i)}|}{n}.
\end{align*} 
Observe that for every $i\in [k]$, there are $3^{|F'|}\le 3^{|F|}$ ways to choose the $t_i$'s. In addition, $\sum_{i=1}^N 2|V_i|/n = 2$ and $k\le |F|$. Thus, we have
\begin{align*} 
\Pr(F\subseteq E(H))&\le L^{3L^2} \bfrac{8\beta^{-5}}{n}^{2|F|} \sum_{F_{abs}\cup F'=F} 3^{|F|} \left(\sum_{i\in [N]} \frac{2|V_i|}{n}\right)^k
\\&\le L^{3L^2} \bfrac{8\beta^{-5}}{n}^{2|F|}  2^{|F|} 3^{|F|} 2^{|F|} 
\le \bfrac{c}{n^2}^{|F|},
\end{align*}
where we used $1/c \ll 1/L, \beta$ and $F$ is non-empty. This finishes the proof of \Cref{thm:main_spread_non_extremal}.
\end{proof}

\begin{proof}[{\bf Proof of \Cref{thm:main_distinct_packing} when $\mathbb{G}$ is a non-extremal family}]
Let $0 < 1/n \ll \zeta \ll 1/c \ll \alpha \le 1$. Starting with a non $\alpha$-extremal family $\mathbb{G}= \{G_1,\ldots,G_n\}$ of Dirac graphs on $n$ vertices,  recursively construct a set of $\zeta n^2$ Hamilton $\mathbb{G}$-transversals as follows. Let $j\in [\zeta n^2]$. Given $j-1$ Hamilton $\mathbb{G}$-transversals constructed so far, let $\mathbb{G}^j=\{G_1^j,\ldots,G_n^j\}$ be the graph family where $G_i^j$ is the subgraph of $G_i$ that contains all the edges that are not spanned by the Hamilton $\mathbb{G}$-transversals constructed so far. If every graph in $\mathbb{G}^j$ has minimum degree at least $(1/2-3c\zeta)n$, then sample a Hamilton $\mathbb{G}^j$-transversal according to the $(c/n^2)$-spread measure given by \Cref{thm:main_spread_non_extremal}. Note that a Hamilton $\mathbb{G}^\ell$-transversal sampled according to a $(c/n^2)$-spread measure contains an edge of $G_i$ incident to $v$ with probability at most $c/n$, for every $i\in [n]$, $v\in V(\mathbb{G})$, and $\ell\le j$. Thus, by the union bound, the probability that for a fixed $j\in [\zeta n^2]$, some graph in $\mathbb{G}^{j}$ has minimum degree less than $(1/2-3c\zeta)n\le \delta(\mathbb{G})-2c\zeta n$ is bounded above by $n^2\Pr(Bin (\zeta n^2,c/n)>2c\zeta n)$, which equals $o(n^{-2})$ by the Chernoff bound. Hence the minimum degree condition is satisfied through the $\zeta n^2$ steps whp. Thus, this procedure generates $\zeta n^2$ pairwise edge-disjoint Hamilton cycles whp, as desired.
\end{proof}

%%%%%%%%%%%%%%%%%%%%%%%%%%%%%%%%%%%%%%%%%%%%%%%%%%%%%%%%%%%%
\section{Extremal families of Dirac graphs} \label{section:extremal families}

In this section, we prove the following two theorems, which together imply Theorems~\ref{thm:main_spread}, \ref{thm:main_spread_exceptional}, and~\ref{thm:main_distinct_packing} in the extremal cases.

\begin{theorem}\label{thm:extremalsimple}
Let $0 < 1/c, \zeta \ll \alpha \ll 1$. For every non-exceptional but $\alpha$-extremal family $\mathbb{G}=\{G_1,\ldots,G_n\}$ of Dirac graphs on $n$ vertices, there exists $(c/n^2)$-spread probability measure on the set of $\mathbb{G}$-transversal Hamilton cycles. In addition, there exists a collection of $\zeta n^2$ pairwise edge-disjoint Hamilton $\mathbb{G}$-transversals.
\end{theorem}

\begin{theorem}\label{thm:superextremal}
Let $0< 1/c, \zeta \ll 1$. Let $\mathbb{G}=\{G_1,\ldots,G_n\}$ be a collection of Dirac graphs on $n$ vertices that is exceptional, and let $A=A(\mathbb{G})$ and $C=C(\mathbb{G})$. Let $(e,i)$ be such that $e\in E(G_i)$ and either $i\in C$ and $e\notin A\times \overline{A}$ or $i\in \overline{C}$ and $e\in A\times \overline{A}$. Also let $\mathbb{G}^{-i}(A,C)=\{G_1',\ldots, G_{i-1}',G_{i+1}',\ldots,G_n'\}$ where $G_j'$ is the bipartite subgraph of $G_j$ with bipartition $A\times \overline{A}$ (also denoted as $G_j[A\times \overline{A}]$) if $j\in C$ and $G_j'=G_j-G_j[A\times \overline{A}]$ otherwise. Then there exists a $(c/n^2)$-spread probability measure on the set of $\mathbb{G}^{-i}$-transversal $(E,\phi)$ where $E$ spans a Hamilton path on $V$ that joins the endpoints of $e$. In addition, there exist a collection of $\min\{\zeta n^2, r(\mathbb{G})\}$ such pairs $(e,i)$ and $\mathbb{G}^{-i}$-transversal Hamilton paths that are pairwise edge-disjoint.
\end{theorem}

Note that \Cref{thm:extremalsimple}, together with \Cref{thm:main_spread_non_extremal}, finishes the proof of \Cref{thm:main_spread}. Moreover, \Cref{thm:superextremal} obviously implies \Cref{thm:main_spread_exceptional}. Additionally, \Cref{thm:extremalsimple,thm:superextremal} imply \Cref{thm:main_distinct_packing} when $\mathbb{G}$ is an extremal family. 

After a ``clean up" procedure, we reduce both the above results to finding an $O(1/n^2)$-spread measure on the set of Hamilton paths with fixed endpoints that are $\mathbb{G}'$-transversal, where $\mathbb{G}'$ is a family of almost complete bipartite graphs that share the same bipartition. Then, we make use of the lemma below. For a given graph $G$ and $u,v\in V(G)$, a Hamilton path of $G$ with the two endpoints $u,v$ is referred to as an $u$-$v$ Hamilton path. 

\begin{lemma}\label{lem:bipartite}
Let $0\le \epsilon \ll 1$. Let $\mathbb{G}$ be a graph family of $2n-1$ bipartite graphs indexed by $C=[2n-1]$ on the vertex set $A\cup B$ with $|A|=|B|=n$. Assume that $\mathbb{G},A,B,C$ satisfy the following conditions.
\begin{enumerate}
\item $e_{G_i}(A,B)\ge (1 - \epsilon)|A||B|$ for every $i\in C$,
\item \label{eq:min deg of A} $\sum_{i\in C} e_{G_i}(v,B)\ge (1 - \epsilon)|B||C|$ for every $v\in A$, and 
\item \label{eq:min deg of B} $\sum_{i\in C} e_{G_i}(v,A)\ge (1 - \epsilon)|A||C|$ for every $v\in B$.
\end{enumerate}
Then, for every $a\in A$ and $b\in B$, there exists an $O(1/n^2)$-spread measure on the set of $a$-$b$ Hamilton paths that are $\mathbb{G}$-transversal. 
\end{lemma}
We prove \cref{lem:bipartite} by constructing a path $P_1$, a matching $M$, and a linear forest $P_2$ whose union forms a Hamilton path with endpoints $a,b$. $P_1$ will have the property that every edge of $P_1$ sees many colors, i.e. it belongs to many graphs in $\mathbb{G}$. Then we let $C_1$ be the colors that appear on many edges of $P_1$. $M$ will be rainbow and use all the colors in $[2n-1]\setminus C_1$. $P_2$ will cover the rest of the vertices and every edge of $P_2$ will see many colors in $C_1$. At the end, we will match and color the edges in $E(P_1)\cup E(P_2)$ with the colors in $C_1$.

We will first use the following lemma to prove \Cref{lem:bipartite} and then later give a proof of it.
\begin{lemma}\label{claim:ham_spread}
Let $0\le \epsilon \le 0.1$. Let $G$ be a bipartite graph on $A\cup B$ with $|A|=|B|=n$, and minimum degree at least $(1-\epsilon)n$. Let $M$ be a matching on $V(G)$ of size at most $\epsilon n$. Then, there exists an $O(1/n)$-spread measure on the set of Hamilton cycles of $G\cup M$ that spans every edge in $M$.
\end{lemma}

\begin{proof}[{\bf Proof of \Cref{lem:bipartite}}]
Without loss of generality, we can assume that $\epsilon> 0$ and $n$ is sufficiently large relative to $\epsilon$. Let $k$ be such that $0 < 1/n \ll 1/k \ll \epsilon \ll 1$, and let $\delta = \epsilon^{1/2}$. 
Let $G$ be the graph on $A\cup B$ with edge set $E(G)$, where $e\in E(G)$ if and only if $e$ belongs to at least $(2-\delta)n$ graphs of $\mathbb{G}$. By \eqref{eq:min deg of A} and \eqref{eq:min deg of B}, the minimum degree of $G$ is at least $(1-\delta)n$.
We first apply \Cref{claim:ham_spread} with $\epsilon = \delta$ and $M=\{ab\}$ on the graph $G$ to give an $(k/n)$-spread measure on the set of paths $P_1$ such that $a,b\notin V(P_1)$ and $|V(P_1)| = 1.9n$. This easily follows by defining $P_1$ to be a subpath of the path obtained from applying \Cref{claim:ham_spread} that does not contain the vertices $a,b$ and has size $1.9n-1$. We will next condition on such a path $P_1$ and then proceed with the plan as sketched before.

Since our arguments still work with any integer in $1.9n \pm 1$ instead of $1.9n$, we will assume that $1.9n$ is an even number from now on. In particular, $|V(P_1)\cap A| = |V(P_1)\cap B| = 0.95n$. Note that the edges in $E(P_1)$ belong to $G$ and have not received a color yet. Let $C_1$ be the set of all colors $c$ such that $|E(P_1)\cap E(G_c)|\ge 1.8n$. Then, $|C_1|\ge (2-20\delta)n$. Remove some colors from $C_1$ so that it has size exactly $(2-20\delta)n$. Let $u,v$ be the endpoints of the path $P_1$.

Let $V= V(G)\setminus (\{a,b\}\cup V(P_1))$. Note that $|V|= 0.1n -2$. We now construct a rainbow matching $M$ on $V$ using all colors in $C\setminus C_1$ with a $(k/n^2)$-spread measure as follows. By using (1), we have $e(G_i[V])\ge 0.49 |V|^2$ for every $i\in C$. We now form the matching $M$ in a random greedy way. Give an arbitrary ordering to the colors in $C\setminus C_1$ and then sequentially, for every $i\in C\setminus C_1$, let $e_i$ be a uniformly random edge that can be added to the matching $M$ formed so far (i.e., add an edge not incident to any of the edges already in $M$). Notice that for every color, we will have a choice of at least $0.49 |V|^2 - 20\delta n \cdot |V|\ge 0.1^3n^2$ edges. Thus, $M$ is $(k/n^2)$-spread over the colored edges of $\mathbb{G}$. We will next condition on such a matching $M$.

Let $V'= V(G)\setminus (V(P_1)\setminus \{u,v\})$ and $G'=G[V']$. Then $G'$ is bipartite on $0.05n + 1$ vertices on both sides and has minimum degree larger than $|V'|/2 - \delta n \ge 0.49|V'|$ (where we use that the minimum degree of $G$ is at least $(1-\delta)n$). We can use \Cref{claim:ham_spread} with $M = \{ab\}\cup \{uv\}\cup M$ to sample a linear forest $P_2$ from an $(k/n)$-spread probability measure such that $P_2\cup \{ab\}\cup \{uv\}\cup M$ forms a Hamilton cycle of $G'$. 

Now let $P=P_1\cup M \cup P_2$. Note that $P$ is a Hamilton path with endpoints $a$ and $b$. It remains to assign the colors in $C_1$ to the edges in $E(P_1)\cup E(P_2)$ in a $(k/n)$-spread manner. For that, consider the auxiliary balanced bipartite graph $H$ with vertex set $C_1\times (E(P_1)\cup E(P_2))$. An edge $(e,c)$ belongs to $H$ if $e\in E(G_c)$. Each vertex of $H$ in $E(P_1)\cup E(P_2)$ corresponds to an edge of $G$, hence it has degree at least $(2-\delta)n - |C\setminus C_1|\ge 1.8n\ge (1-0.1)|C_1|$ in $H$. Thereafter, each vertex in $C_1$ has degree at least $1.8n-1\ge (1-0.1)|C_1|$. Observe that there exists a natural bijection between perfect matchings of $H$ and $C_1$-rainbow colorings of $E(P_1)\cup E(P_2)$. To construct a perfect matching $M$ of $H$ in a $(k/n)$-spread manner, apply \cref{claim:ham_spread} to construct a Hamilton path $P$ in $H$ in a $(k/n)$-spread manner and then take $M$ to be the maximum matching of $P$ (i.e the one that spans every other edge of $P$ and is incident to the endpoints of $P$). This completes the construction of the $a$-$b$ Hamilton path that is $\mathbb{G}$-transversal. 

We next check that this procedure gives an $O(1/n^2)$-spread measure on the set of such paths. Consider a set $C'\subseteq C$ and edges $e_i = a_ib_i \in E(G_i)$ for every $i\in C'$ and let $E' = \{e_i : i\in C'\}$. We are interested in the probability that $E'\subseteq E(P)$. Fix a partition of the set $E'$ into three sets $E_1, E_2, E_3$. Then, the probability that edges in $E_1$ without the corresponding colors all lie in $P_1$ is at most $(k/n)^{|E_1|}$. Then, conditioning on this, the probability that all edges in $E_2$ with the colors lie in the rainbow matching $M$ is at most $(k/n^2)^{|E_2|}$. Conditioned on these, the probability that all uncolored edges in $E_3$ lie in $P_2$ is at most $(k/n)^{|E_3|}$. Finally, conditioning on all these, while giving the random coloring of $P_1\cup P_2$, the probability that the uncolored edges of $E_1\cup E_3$ get the corresponding color is at most $(k/n)^{|E_1| + |E_3|}$. Thus, we have 
\[
\Pr(E_1\subseteq E(P_1) \bigwedge E_2\subseteq E(M) \bigwedge E_3\subseteq E(P_2)) \le (k/n)^{2(|E_1| + |E_2| + |E_3|)} = (k/n^2)^{|E_1| + |E_2| + |E_3|}.
\]
Thus, by a union bound over all partitions of $E'$ into three sets, we have $\Pr(E'\subseteq P) = 3^{|E'|} \cdot (k/n^2)^{|E'|}$. This completes the proof of \Cref{lem:bipartite}.
\end{proof}

\Cref{claim:ham_spread} can be proved in a number of ways (including a similar approach to the proof of the non-extremal case of \Cref{thm:main_spread}). We give a proof using a more classical approach.

\begin{proof}[{\bf Proof of \Cref{claim:ham_spread}}]
Let $G_0$ be a random subgraph of $G$ whose edge set $E(G_0)$ is $\cup_{v\in V(G)}E_v$ where for every vertex $v\in G$, the set $E_v$ is formed by including $100$ edges from $\{vw \in E(G): w\notin V(M)\}$. We first aim to show that the random graph $G' = G_0\cup G_{1000/n}$ (which is the union of $G_0$ and an independently generated $G_{1000/n}$) has the property that $G'\cup M$ whp contains a Hamilton cycle that spans all edges in $M$. 

We first show that the minimum degree condition on $G$ implies good expansion properties of $G_0$. For a graph $H$ and a matching $M$ on $V(H)$, we call $H$ an \textit{$M$-respecting $2$-expander} if $H$ is connected and for every $U\subseteq V(H)$ with $|U|\le |V(H)|/8$, we have $|N_{H}(U)\setminus V(M)|\ge 2|U|$. Let us first show that 
\begin{equation}\label{eq: expansion}
\text{whp for every $U\subseteq A$ or $U\subseteq B$ with $|U|\le n/2$, we have $|N_{G_0}(U)|> \min\{4|U|, n/2\}$.} 
\end{equation}
Indeed, by a union bound, this does not hold with probability at most 
\begin{align*}
\sum_{k\in [n/8]} 2 \binom{n}{k} \binom{n}{4k}\left(\frac{\binom{4k}{100}}{\binom{(1-2\epsilon)n}{100}}\right)^k + \sum_{k\in [n/8, n/2]} 2 \binom{n}{k} \binom{n}{n/2}\left(\frac{\binom{n/2}{100}}{\binom{(1-2\epsilon)n}{100}}\right)^k = o(1).
\end{align*}
Now, \eqref{eq: expansion} easily implies that whp for every $U\subseteq V(G_0)$ with $|U|\le n/4$, we have $|N_{G_0}(U)\setminus V(M)|\ge 2|U|$.
We next argue that whp $G_0$ is connected. Indeed, if $G_0$ is not connected, then there must exist a connected component of $G_0$ of size at most $n$, and thus there must be a non-empty set $U\subseteq A$ or $U\subseteq B$ with $|U|\le n/2$ such that $|N_{G_0}(U)|\le |U|$, which then contradicts \eqref{eq: expansion}. This completes the proof that $G_0$ is whp an $M$-respecting $2$-expander.

A similar argument shows that whp $G_0\cup M$ contains a $2$-factor containing all edges of $M$. For this, one can, for example, decompose the graph $G_0$ into two edge-disjoint random graphs $G'_0$ and $G''_0$ where each of these graphs is formed by including $50$ edges from $\{vw \in E(G): w\notin V(M)\}$. Then, one can argue using Hall's theorem and a property analogous to \eqref{eq: expansion} that each of $G'_0$, $G'0[V(G)\setminus V(M)]$  contains a perfect matching whp. By taking the union of these two perfect matchings and $M$, we can obtain a desired $2$-factor in $G_0\cup M$ spanning all edges in $M$.

Next, we will use a bipartite variant of the classical P\'osa rotation-extension argument (\cite{posa1976hamiltonian}). The rotation-extension technique was extended to bipartite graphs by Frieze~\cite{frieze1985limit}, which was later simplified by Bollob\'as and Kohayakawa~\cite{bollobas1991hitting}. We follow this later approach. We use the same terminology as in the proof of \Cref{thm:robustdirac}. In particular, a spanning subgraph $H$ of $G\cup M$ is called \textit{good} if for all $uv\in M$, we have $uv\in E(H)$. 

For a graph $H$, we define $\ell(H)$ to be the size of a largest component of a good $2$-factor of $H$. For a graph $H$, we call an edge $e\in \binom{V(H)}{2}$ a \textit{good booster} if $\ell(H+e) > \ell(H)$ or the graph $H+e$ contains a good Hamilton cycle. Now, the following two claims finish the proof of the fact that the random graph $G_0\cup G_{1000/n}\cup M$ contains a good Hamilton cycle whp.

\begin{claim}
Let $G_0$ be an $n\times n$ bipartite graph containing a good $2$-factor with respect to a matching $M$ on $V(G_0)$. If $G_0$ is an $M$-respecting $2$-expander, then $G_0$ has at least $n^2/32$ good boosters.
\end{claim}
\begin{proof}
We first consider a good 2-factor $F$ in $G_0$ such that $\ell(G_0)$ equals the size of a largest component of $F$. Let $C$ denote such a largest component of size $\ell(G_0)$. We next show that there is a good spanning subgraph $F'$ of $G_0$ such that 
\begin{enumerate}
\item $F'$ is a union of a path $P$ and a $2$-factor of $G_0\setminus V(P)$.
\item $V(P)\supsetneq V(C)$. 
\end{enumerate} 
Since $G_0$ is a connected graph, there exists an edge $v_1v_2\in E(G_0)$ such that $v_1\in V(C)$ and $v_2\in V(G_0)\setminus V(C)$. Let $C'$ denote the component of $F$ containing $v_2$. For $i\in [2]$, let $x_i$ and $y_i$ denote the neighbors of $v_i$ in $F$. Note that since $M$ is a matching, at most one edge from $\{v_ix_i, v_iy_i\}$ belong to $M$ for each $i\in [2]$. Without loss of generality, let $v_ix_i \notin E(M)$ for each $i\in [2]$. Thus, by deleting the edges $v_1x_1$ and $v_2x_2$ from $C$ and $C'$, and adding the edge $v_1v_2$, we get a good path $P$ on $V(C)\cup V(C')$. We can now consider $F'$ to be the union of $P$ and the $2$-factor of $G_0\setminus V(P)$ obtained from deleting the components $C$ and $C'$ from $F$.

We now fix a good spanning subgraph $F'$ of $G_0$ satisfying (1) and (2), where the length of $P$ is as large as possible. Due to the assumption that the length of $P$ is as large as possible, we have the following property. 
\begin{itemize}
\item[$(\star)$] Every good path $P'$ in $G_0\cup M$ with vertex set $V(P') = V(P)$ satisfies that the endpoints of $P'$ do not have a neighbor outside $V(P)$.
\end{itemize}
For a vertex $u\in V(P)$, let $S(u)$ be the set of endpoints not equal to $u$ of good paths with the vertex set $V(P)$ with $u$ as one endpoint. Let $v$ be an endpoint of $P$. For every $x\in S(v)$, let $P_x$ be a $v$-$x$ path with $V(P_x) = V(P)$. Let $E = \{xy : x\in S(v) ~ \text{and} ~ y\in S(x)\}$. We will show that $E$ is a desired set of at least $n^2/32$ good boosters. To see that $|E|\ge n^2/32$, we will use the following variant of the standard P\'osa's lemma. 
\begin{lemma}[Lemma~7.7 in \cite{behague2025colour}]
Let $G_0$ be a graph, and let $M$ be a matching on $V(G_0)$. Let $H$ be an $M$-respecting $2$-expander. Let $P$ be a good path in $G_0\cup M$ satisfying $(\star)$, which has $v$ as an endpoint. Then, there are at least $|V(G_0)|/8$ vertices $u\in V(P)$ for which there is an $M$-respecting $v$-$u$ path in $G_0\cup M$ with vertex set $V(P)$.
\end{lemma}
We remark that this lemma is proved with a stronger assumption that $P$ is a longest good path in $G_0$. However, we can replace this with the assumption that $(\star)$ holds, as the proof only uses this weaker property.

By the above lemma, we have $|S(v)|\ge n/4$ and for each $x\in S(P,v)$, we have $|S(x)|\ge n/4$. This immediately implies that $|E|\ge n^2/32$. The only thing left to show is that every edge of $E$ is a good booster. Let $e\in E$. Clearly, the graph $G_0+e$ contains a good cycle $C_e$ on the vertex set $V(P)$. If $V(P) = V(G_0)$, then the graph $G_0+e$ contains a Hamilton cycle. If not, then the cycle $C_e$ along with the components of $F'$ minus $P$, forms a good $2$-factor of $G_0$ where $\ell(G_0+e)\ge \ell(C_e) > \ell(C) = \ell(G_0)$, as desired.
\end{proof}

\begin{claim}
If $G_0$ is an $n\times n$ bipartite graph with at least $n^2/32$ good boosters with respect to a matching $M$ on $V(G_0)$, then whp $G_0\cup G_{1000/n}\cup M$ contains a Hamilton cycle containing $M$.
\end{claim}
\begin{proof}
Let $m=10n$ and $p=100/n^2$. For every $i\in [m]$, let $G_i$ be an independent copy of $G_{100/n}$. Then, 
\[
\Pr(G_0\cup G_{1000/n}\cup M ~ \text{contains a good Hamilton cycle}) \ge \Pr(\cup_{i=0}^m G_i \cup M ~ \text{contains a good Hamilton cycle}).
\]
For every $j\in [m]$, define a Bernoulli random variable $X_j$ such that $X_j = 1$ if $G_j$ contains a good booster for the graph $\cup_{i=0}^{j-1} G_i$, and otherwise $X_j = 0$. Note that the event $\sum_{j\in [m]} X_j \ge n$ implies that the graph $\cup_{i=0}^m G_i \cup M$ contains a good Hamilton cycle. Thus, it is sufficient to show that whp $\sum_{j\in [m]} X_j \ge n$.  

Observe that $\Pr(X_j = 0) \le (1-p)^{n^2/32} \le 1/2$. Thus, a standard application of Azuma-Hoeffding inequality yields that the probability that $\sum_{j\in [m]} X_j$ is less than $n \le m/2 - n$ is at most $2e^{-\frac{n^2}{2m}} = o(1)$, as desired.
\end{proof}

{\bf Wrapping up the proof of \Cref{claim:ham_spread}.} Let $p^*$ be the probability that $G_0\cup G_{1000/n}\cup M$ does not contain a Hamilton cycle spanning all edges in $M$. We have shown that $p^*=o(1)$. Denote by $\mathcal{F}_G$ the family of subgraphs $H\subseteq G$ such that $H\cup M$ contains a Hamilton cycle spanning all edges in $M$.
We finally describe a procedure to sample Hamilton cycles with the desired spreadness property. For every graph $H\in \mathcal{F}_G$, fix a good Hamilton cycle $C_H$ in $H\cup M$. Now, consider the distribution on the good Hamilton cycles $C$, give it a measure of 
\[
\frac{1}{1-p^*}\sum_{H\in \mathcal{F}_G : C_H = C} \Pr(G_0\cup G_{1000/n} = H).
\]
This gives an $O(1/n)$-spread measure as every set of edges $E$ belongs to $G_0\cup G_{1000/n}$ with probability at most $(2000/n)^{|E|}$.
\end{proof}

\subsection{Proof of \Cref{thm:extremalsimple}}

Let $0 < \zeta \ll \alpha \ll \epsilon \ll 1$.  
Suppose we are given a non-exceptional, $\alpha$-extremal graph family $\mathbb{G}$ of edge-minimal Dirac graphs with respect to a fixed half-set $A_0\subseteq V(\mathbb{G})$. We say that a graph $G\in \mathbb{G}$ is $(1-\epsilon)$-bipartite if $e_G(A_0)+e_G(\overline{A_0})\le \epsilon n^2$. We let $m(\mathbb{G})$ be the number of graphs in $\mathbb{G}$ that are $(1-\epsilon)$-bipartite. For the proof of \Cref{thm:extremalsimple}, we consider the case distinction based on whether $m(\mathbb{G})\ge (1-\epsilon/100)n$ or not. 

\subsubsection{\textbf{Case-I: {\boldmath $m(\mathbb{G})\ge (1-\epsilon/100)n$}}, i.e., at least $(1-\epsilon/100)n$ graphs in $\mathbb{G}$ are $(1-\epsilon)$-bipartite graphs}

\begin{definition}\label{def:bipartitecleanup}
For a graph family $\mathbb{G}$ on $n$ vertices, let $\cP=\cP(\mathbb{G})$ be the set of all rainbow paths $P$ such that if we let $v_1=v_1(P)$ and $v_2=v_2(P)$ be the endpoints of $P$, and $C$ be the set of all colors not appearing on the edges in $E(P)$, then there exists an equibipartition of $\{v_1,v_2\} \cup (V(\mathbb{G})\setminus V(P))$ into two sets $A,B$ such that the following are satisfied. 
\begin{itemize}
\item[(1)] $v_1\in A$ and $v_2\in B$.
\item[(2)] $e_{G_c}(A,B)\ge (1-\epsilon^{1/2})|A||B|$ for $c\in C$.
\item[(3)] $\sum_{c\in C} d_{G_c}(v,A)\ge (1-\epsilon^{1/2})|A||C|$ for $v\in B$.
\item[(4)] $\sum_{c\in C} d_{G_c}(v,B)\ge (1-\epsilon^{1/2})|B||C|$ for $v\in A$.
\item[(5)] $|A|,|B|,|C|=\Omega(n)$.
\end{itemize}
\end{definition}

\begin{lemma}\label{lem:cleanupbipartite}
Let $\mathbb{G} = \{G_1,\ldots,G_n\}$ be an $\alpha$-extremal family of Dirac  graphs on $n$ vertices satisfying the condition of Case~I. Then, there exists an $O(1/n^2)$-spread measure on $\mathcal{P}(\mathbb{G})$.  In addition, there exists a collection of $\zeta n^2$ elements of $\mathcal{P}(\mathbb{G})$ that are pairwise edge-disjoint.
\end{lemma}

We first prove \Cref{thm:extremalsimple} assuming the above lemma and then prove it right after.

\begin{proof}[{\bf Proof of \Cref{thm:extremalsimple} in Case~I}] 
To prove the first part of \Cref{thm:extremalsimple}, when $m(\mathbb{G})\ge (1-\epsilon/100)n$, first sample $P\in \mathcal{P}(\mathbb{G})$ according to the $O(1/n^2)$-spread measure given by \Cref{lem:cleanupbipartite} and then apply \Cref{lem:bipartite} with $a=v_1(P)$ and $b=v_2(P)$ and graph family $\mathbb{G}'$ given by $G_i[A,B], i\in C$, where $A,B,C$ are as in \Cref{def:bipartitecleanup}. 
To construct $\zeta n^2$ pairwise edge-disjoint Hamilton $\mathbb{G}$-transversals, first apply \Cref{lem:cleanupbipartite} to find a collection $\{P_1,\ldots,P_{\zeta n^2}\}$ of $\zeta n^2$ elements of $\mathcal{P}(\mathbb{G})$ that are pairwise edge-disjoint. Then, for $i\in [\zeta n^2]$, construct a new Hamilton $\mathbb{G}$-transversal, as before, that spans $P_i$ and avoids all the edges that are spanned by the Hamilton transversals constructed so far or by the paths in $\{P_1,\ldots,P_{\zeta n^2}\}$. This can be done since, at each step, we avoid at most $2\zeta n^2$ edges incident to each vertex, thus we can still apply \Cref{lem:bipartite}.
\end{proof}

\begin{proof}[{\bf Proof of \Cref{lem:cleanupbipartite}}] 
We let $C$ be the set of indices of the $(1-\epsilon)$-bipartite graphs in $\mathbb{G}$. Let $B_0=\overline{A_0}$. We first slightly modify the sets $A_0,B_0$. We then build the paths $P_1,\ldots,P_{\zeta n^2}$, where $E(P_i)$ will be the union of at most $5$ edge sets $E_0^i,E_1^i,\ldots,E_4^i$. The purpose of  $E_1^i$, $E_2^i$, and $E_3^i$ is to cover ``bad" colors and vertices, i.e., colors and vertices that do not satisfy conditions (2), (3), and (4).  $E_0^i$ has at most $1$ edge, and its purpose is to address parity issues that could arise.  Finally, $E_4^i$ is used to join the edges in $E_0^i\cup E_1^i\cup E_2^i\cup E_3^i$ into a path~$P$.

\textbf{Step 1: Modify the sets {\boldmath $A_0,B_0$} and construct the edge-set {\boldmath $E_0^i$}.}
We will modify the sets $A_0, B_0$ to obtain a partition $A_1, B_1$ of $V(\mathbb{G})$ such that $\sum_{c\in C}d_{G_c}(v,A_1)\ge 1/3 \sum_{c\in C}d_{G_c}(v)$ for every $v\in B_1$ and $\sum_{c\in C}d_{G_c}(v,B_1)\ge 1/3 \sum_{c\in C}d_{G_c}(v)$ for every $v\in A_1$. To achieve this, starting from the partition $A_0,B_0$, we iteratively relocate a vertex $v\in A_0$ from $A_0$ to $B_0$ if $\sum_{c\in C}d_{G_c}(v,B_0)< 1/3 \sum_{c\in C}d_{G_c}(v)$ (equivalently, $\sum_{c\in C}d_{G_c}(v,A_0)> 2/3 \sum_{c\in C}d_{G_c}(v)$), and also similarly relocate vertices from $B_0$ to $A_0$. Let $A_1$, $B_1$ be the updated partition (this relocation process will terminate after at most $18\alpha n$ steps as shown below). As $\mathbb{G}$ is $\alpha$-extremal with respect to $A_0$ and $G_c$ is $(1-\epsilon)$-bipartite for every $c\in C$, we have $\sum_{c\in C} e_{G_c}(A_0) \le \sum_{c\in C} \min\{e_{G_c}(A_0), e_{G_c}(A_0,B_0)\}\le \alpha n^3$. Thus, $\Phi:= \sum_{c\in C} e_{G_c}(A_0,B_0) \ge |C||A_0|n/2 - 2\alpha n^3 \ge |C| (1/4 -3\alpha)n^2$. Note also that for every $A\subseteq V(\mathbb{G})$, we have $\sum_{c\in C} e_{G_c}(A,V(\mathbb{G})\setminus A) \le |C| n^2/4$.
Each time we relocated a vertex from $A_0$ to $B_0$, the number $\Phi$ (with respect to the updated $A_0, B_0$) increases by at least $|C|n/6$, thus at most $\frac{3\alpha|C|n^2}{|C|n/6} = 18\alpha n$ vertices are relocated. By symmetry, we may assume that $|A_1|\ge |B_1|$. Hence, $0\le |A_1|-|B_1|\le 36\alpha n$ and $|A_0\triangle A_1|\le 18\alpha n$.

Then, relocate from $A_1$ to $B_1$, one by one, vertices $v\in A_1$ such that $\sum_{c\in C}d_{G_c}(v,A_1)\ge 0.0001 n^2$ until $|A_1|-|B_1|\le 1$ or no such vertices exist. Let $A_2,B_2$ be the resultant sets. Note that $|A_1\setminus A_2|\le (|A_1|-|B_1|)/2 \le 18\alpha n$. Since $\sum_{c\in C}d_{G_c}(v,B_2) \ge \sum_{c\in C}d_{G_c}(v,B_1)$ for $v\in A_2$ (as $B_2\supseteq B_1)$  and $\sum_{c\in C}d_{G_c}(v,A_2)\ge \sum_{c\in C} (d_{G_c}(v,A_1) - |A_1\setminus A_2|)$ for $v\in B_2$, we have 
\begin{equation}\label{eq:lower bound on across edges}
\text{$\sum_{c\in C}d_{G_c}(v,B_2) \ge 0.00005 n^2$ for $v\in A_2$ and $\sum_{c\in C}d_{G_c}(v,A_2)\ge 0.00005 n^2$ for $v\in B_2$.}
\end{equation} 
Since $|A_0\triangle A_2|\le |A_0\triangle A_1| + |A_1\setminus A_2| \le 36\alpha n$, 
\begin{equation}\label{eq:number of edges between A2 B2}
\text{the graph $G_c[A_2,B_2]$ spans at least $(1/4-\epsilon-36\alpha)n^2\ge (1/4-2\epsilon)n^2$ edges for every $c\in C$.}
\end{equation}

If $|A_2|>|B_2|$ and $|C|$ is odd, then for $i\in [\zeta n^2]$, let $c^*_i$ be a uniform element of $C$ such that there exists an edge of color $c^*_i$ spanned by $A_2$ not among $\{e_{c^*_1},\ldots,e_{c^*_{i-1}}\}$ and $e_{c^*_i}$ be a random such edge (initially there exist at least $n/4$ eligible edges for each color $c \in C$). Else, if $|A_2|=|B_2|$ and $|C|$ is odd, then for $i\in [\zeta n^2]$, let $e_{c^*_i}$ be a uniform element taken from the set of edges that either have color $c^*_i\in C$ and are not spanned by $A_2\times B_2$ or have color $c^*_i\in [n]\setminus C$ and are spanned by $A_2\times B_2$ not among $\{e_{c^*_1},\ldots,e_{c^*_{i-1}}\}$. As our graph family is not exceptional, initially, there exist at least $0.01n^2$ such edges. In both cases, we let $E_0^i=\{e_{c^*_i}\}$ and $C_0^i=\{c^*_i\}$ for all $i\in [\zeta n^2]$. When $|C|$ is even, we let $E_0^i=C_0^i=\emptyset$ for all $i\in [\zeta n^2]$. Let $E_0=\bigcup_{i\in [\zeta n^2]}E_0^i$.

\textbf{Step 2: Construct the edge-set {\boldmath $E_1^i$} in {\boldmath $A_2$} to ``balance" the sets {\boldmath $A_2,B_2$}.}
Let $d=|A_2|-|B_2|$. Note that $0\le d\le |A_1|-|B_1|\le 36\alpha n$. If $d\in \{0,1\}$, then let $E_1^i=\emptyset$ for all $i\in [\zeta n^2]$. Assume $d\ge 2$ in the following.
We now construct a $C\setminus C_0^1$-rainbow matching $E_1^1$ of size $2 \lfloor d/2 \rfloor$ spanned by $A_2$ respecting an $O(1/n^2)$-spread measure. (The number $2 \lfloor d/2 \rfloor$ is carefully chosen to balance the sets $A_2, B_2$ appropriately.) Note that $A_2$ spans at least $|A_2|d/4$ edges of every color in $C$. For $i\in \{1,\cdots,2 \lfloor d/2 \rfloor\}$, choose a uniformly random edge not in $E_0$ and not incident to the edges added to $E_0^1\cup E_1^1$ so far of a distinct color $c\in C\setminus C_0^1$. 
Then, at step $i\in [2 \lfloor d/2 \rfloor]$, for each vertex $v\in A_2$, as there are in total (counting with multiplicity) at most $0.0001n^2$ edges $uv$ of color in $C$ such that $u\in A_2$ (due to the termination condition of Step~1 and $A_2\subseteq A_1$), there exist at least $(|C|-d)|A_2|d/4 - 2d\cdot 0.0001n^2 -|E_0|\ge 0.2dn^2-\zeta n^2 \ge 0.1dn^2$ eligible edges. Therefore, each colored edge is chosen with probability at most $10n^{-2}$ during this step (conditioned on the event that any other set of colored edges are previously chosen). 

Similarly, for $2\le i\le \zeta n^2$, construct the set $E_1^i$ such that it is disjoint from each of the sets $E_0$ and $E_1^1,\ldots,E_1^{i-1}$ (once again, at each step there exist at least $0.2dn^2-\zeta d n^2 \ge 0.1dn^2$ eligible edges). Then set $E_1=\cup_{i\in [\zeta n^2]}E_1^i$.
In the same way, for $i\in [\zeta n^2]$, perform the following steps always selecting edges that have not been chosen so far.

\textbf{Step 3: Construct an {\boldmath $[n]\setminus (C\cup C_0^i)$}-transversal matching {\boldmath $E_2^i$} disjoint from {\boldmath $A_2\times B_2$}.}
For $i\in [\zeta n^2]$, we now construct, in a random greedy manner, an $[n]\setminus (C\cup C_0^i)$-transversal matching $E_2^i$ disjoint from $A_2\times B_2$ and $E_0\cup E_1$ such that $E_0^i\cup E_1^i\cup E_2^i$ forms a matching.
For every $c\in [n]\setminus (C\cup C_0^i)$, sequentially, if currently there are more edges in $E_0^i\cup E_2^i$ that are spanned by $A_2$ than by $B_2$, then choose a random edge in color $c$ spanned by $B_2$, else choose a random edge spanned by $A_2$. This gives us the desired spreadness because for every $c\notin C$,  the number of edges of $G_c$ spanned by both $B_2$ and $A_2$ is at least 
\[
\min\{e_{G_c}(A_2), e_{G_c}(B_2)\} - \zeta n^2 - 2n(|E_0^i| + |E_1^i| + (n-|C|)) \ge \epsilon n^2/5 - \zeta n^2 - 2n (1 + 18\alpha n + \epsilon n/100)
\ge \epsilon n^2/6,
\]
where the lower bound on $e_{G_c}(A_2)$ (and similarly on $e_{G_c}(B_2)$) can be reasoned as follows. First, $c\notin C$ implies that $G_c$ is not $(1-\epsilon)$-bipartite with respect to $A_0$ and therefore, as $G_c$ is an edge-minimal Dirac graph, we have $e_{G_c}(A_0)\ge \epsilon n^2/4$ (by applying \Cref{lem:complement is also sparse}). Then, $e_{G_c}(A_2)\ge e_{G_c}(A_0)- n|A_0\triangle A_2| \ge  \epsilon n^2/4-30\alpha n^2\ge \epsilon n^2/5$, as desired. The above inequality gives that at each step there are $\Omega(n^2)$ eligible edges to choose from. 

\textbf{The analysis for the number of edges used in {\boldmath $A_2$} and {\boldmath $B_2$}.}
We now show that 
\begin{equation}\label{eq:difference for 0 and 2}
\text{$\left|(E_0^i\cup E_2^i)\cap \binom{A_2}{2}\right| - \left|(E_0^i\cup E_2^i)\cap \binom{B_2}{2}\right|$ equals $1$ if $d$ is odd, and equals $0$ if $d$ is even.}
\end{equation}
To see this, we will analyze all parity cases that can arise. Note that $d$ and $n$ always have the same parity. First, if $|A_2| > |B_2|$ and $|C|$ is odd, then $C_0^i\subseteq C$, the edge in $E_0^i$ is spanned by $A_2$, and thus $|E_0^i\cup E_2^i|$ has the same parity as $d$ or $n$. This ensures the validity of \eqref{eq:difference for 0 and 2}. Second, if $|A_2| = |B_2|$ (i.e., d=0) and $|C|$ is odd and $C_0^i\subseteq C$, then the edge in $E_0^i$ is spanned by $A_2$ or $B_2$, and $|E_0^i\cup E_2^i|$ is even. Hence, \eqref{eq:difference for 0 and 2} holds. Third, if $|A_2| = |B_2|$ and $|C|$ is odd and $C_0^i\not\subseteq C$, then the edge in $E_0^i$ is spanned by $A_2\times B_2$, and $|E_2^i|$ is even. Thus, \eqref{eq:difference for 0 and 2} holds. Finally, if $|C|$ is even, then $E_0^i = \emptyset$ and $|E_2^i|$ has the same parity as $d$ or $n$. This finishes the proof of \eqref{eq:difference for 0 and 2}.

Let $E^i=E_0^i\cup E_1^i\cup E_2^i$. Then, $E^i$ is a rainbow matching of size at most $\epsilon n/50$ that uses all the colors not in~$C$. Furthermore, by \eqref{eq:difference for 0 and 2} and the fact that $|E_1^i\cap \binom{A_2}{2}| = 2\lfloor d/2\rfloor$ and $|E_1^i\cap \binom{B_2}{2}| = 0$, we have 
\begin{equation}\label{eq:difference between number of edges}
\left|E^i\cap \binom{A_2}{2}\right| - \left|E^i\cap \binom{B_2}{2}\right| = |A_2| - |B_2|.
\end{equation}

\textbf{Setup for upcoming steps.}
We will now informally describe and prepare for the final two steps. We will create edge sets $E_3^i\cup E_4^i$ spanned by $A_2\times B_2$ and having a distinct color in $C$, not used so far while constructing $P_i$. Assume that at the end, $\bigcup_{j=0}^4 E_j^i$ spans a path $P_i$ of size at most $0.001n$. The above automatically ensures that $P_i$ is rainbow and every color not used on $E(P_i)$ belongs to $C$. In addition, if we let $v_1(P_i),v_2(P_i)$ be its two endpoints, then we will show that one of them belongs to $A_2$, say $v_1(P_i)$ and the other to $B_2$. If we further let $C(P_i)$ be the set of colors not appearing on $P_i$, and let $A(P_i)=A_2\setminus (V(P_i)\setminus \{v_1(P_i)\})$ and $B(P_i)=B_2\setminus (V(P_i)\setminus \{v_2(P_i)\})$, then we will show the sets $A(P_i),B(P_i),C(P_i)$ witness the fact that $P_i$ belongs to the set $\cP(\mathbb{G})$.

Let $V_A$ be the set of vertices $v\in A_2$ such that $\sum_{c\in C}d_{G_c}(v,B_2) \le (1/2 - 5\epsilon) n^2$.  Similarly let $V_B$ be the set of vertices $v\in B_2$ such that $\sum_{c\in C}d_{G_c}(v,A_2) \le (1/2 - 5\epsilon) n^2$. Thus, by using \eqref{eq:number of edges between A2 B2}, we have $|V_A|,|V_B|\le 5 \epsilon n$.

\textbf{Step 4: Construct a {\boldmath $C$}-rainbow set {\boldmath $E_3^i\subset A_2\times B_2$} to cover the vertices of small degree.} 
We continue and cover the vertices  $V_A\cup V_B$ by short rainbow paths. For that, we greedily match (at random) the vertices in $(V_A\cup V_B)\setminus V(E^i)$ twice to vertices in $(A_2\cup B_2)\setminus (V(E^i)\cup V_A\cup V_B)$ and the vertices in $(V_A\cup V_B)\cap V(E^i)$ to vertices in $(A_2\cup B_2)\setminus (V(E^i)\cup V_A\cup V_B)$ once. The set of (colored) edges $E_3^i$ chosen during this step is constrained such that $E^i\cup E_3^i$ is acyclic and disjoint from $E_0,E_1,E_2$ and $\cup_{j<i}E_3^j$. We then set $F^i=E^i\cup E_3^i$. 
Thus, by using \eqref{eq:lower bound on across edges}, at any point during Step~4, there exists at least $0.00005n^2 -|V(E^i)|n -3|V_A|n-3|V_B|n-\zeta n^2 \ge 10^{-5}n^2$ eligible edges to choose from. This gives the $O(1/n^2)$-desired spreadness of the set $E_3^i$. 

\textbf{Step 5: Construct a {\boldmath $C$}-rainbow set {\boldmath $E_4^i\subset A_2\times B_2$} to merge the edges in {\boldmath $F^i$} into a single path~{\boldmath $P_i$}.} 
Let $V^i$ be a random set of vertices not incident to $F^i$ of size $\epsilon n$ (which we also consider as a set of $\epsilon n$ paths of length $0$) and $Q_1,Q_2,\ldots,Q_w$ be a random ordering of the paths that either are spanned by $F^i$ or belong to $V^i$. Let $v_j^+,v_j^-$ the two endpoints of $Q_j$ for all $j\in[w]$. For every $j\in [w-1]$, we join $v_j^+$ to $v_{j+1}^-$ via a random path of length $2$ if both $v_j^+,v_{j+1}^-$ belong to $A$ or both belong to $B$ and via a random path of length $3$ otherwise. These paths are constrained such that the set of edges chosen so far at iteration $j$ is acyclic, and no edge is chosen twice overall. Finally, if both $v_1^-, v_w^+$ belong to $A$ or both belong to $B$, then add a new random edge incident to $v_w^+$ (all of the edges are of distinct colors). Since every vertex not incident to a single edge of $F^i$ has degree at least $0.49n^2$ (incorporating all remaining colors), there are $\Omega(n^4)$ distinct choices for paths of length 2 and $\Omega(n^6)$ distinct choices for paths of length 3 at each point.

We let $P_i$ be the resultant path with endpoints $x_i$ and $y_i$. We will show that $P_i\in \cP(\mathbb{G})$ for every $i\in [\zeta n^2]$. For that, letting $Z_i:= \{x_i,y_i\}\cup (V(\mathbb{G})\setminus V(P_i))$, we will verify the conditions of \Cref{def:bipartitecleanup} with respect to the bipartition of $Z_i$ into $Z_i\cap A_2$ and $Z_i\cap B_2$. Since $P_i$ is a path with one endpoint in $A_2$ and the other in $B_2$ and $|e(P_i)\cap \binom{A_2}{2}| - |e(P_i)\cap \binom{B_2}{2}| = |A_2| - |B_2|$ (by \eqref{eq:difference between number of edges}), we conclude that $|V(P_i)\cap A_2| - |V(P_i)\cap B_2| = |A_2| - |B_2|$. Thus, $Z_i\cap A_2$ and $Z_i\cap B_2$ form an equibipartition of $Z_i$. Additionally, the conditions (1)--(5) follow from our construction. This shows that $P_i\in \cP(\mathbb{G})$ for every $i\in [\zeta n^2]$. It is a routine to check that the above 5-step procedure for constructing $P_1$ gives an $O(1/n^2)$-spread measure on $\cP(\mathbb{G})$. This completes the proof of \Cref{lem:cleanupbipartite}.
\end{proof}

\subsubsection{\textbf{Case-II: {\boldmath $m(\mathbb{G})\le (1-\epsilon/100)n$}}, i.e., at most $(1-\epsilon/100)n$ graphs in $\mathbb{G}$ are $(1-\epsilon)$-bipartite}

\begin{definition}\label{def:cleanupcliques}
For a graph family $\mathbb{G}$ on $n$ vertices, let $\cP'=\cP'(\mathbb{G})$ be the set of vertex-disjoint pairs of paths $(P_1,P_2)$ with endpoints $v_1(P_1),v_2(P_1)$ and $v_1(P_2), v_2(P_2)$, whose edge set is rainbow and satisfy the following. Let $V(P_1,P_2)=(V(P_1)\cup V(P_2))\setminus \{v_1(P_1),v_2(P_1),v_1(P_2),v_2(P_2)\}$ and $C$ be the set of all colors not appearing in $E(P_1)\cup E(P_2)$. Then there exists a bipartition of $V(\mathbb{G})\setminus V(P_1,P_2)$ into two sets $A,B$ such that the following hold. 
\begin{itemize}
    \item[(1)] $v_1(P_1),v_1(P_2)\in A$ and $v_2(P_1),v_2(P_2)\in B$. 
    \item[(2)] $e_{G_c}(W)\ge (1-\epsilon^{1/2})\binom{|W|}{2}$ for $c\in C$ and $W\in \{A,B\}$. 
    \item[(3)] $\sum_{c\in C} d_{G_c}(v,W)\ge (1-\epsilon^{1/2})|W||C|$ for $W\in \{A,B\}$ and $v\in W$.
    \item[(4)] $|A|,|B|,|C|=\Omega(n)$ and $|A|,|B|$ are even.
\end{itemize}
\end{definition}

\begin{lemma}\label{lem:cleanupcliques}
Let $0 < \zeta \ll \alpha \ll \epsilon \ll 1$. Let $\mathbb{G}=\{G_1,\ldots,G_n\}$ be a family of Dirac graphs on $n$ vertices such that $\mathbb{G}$ is $\alpha$-extremal with respect to the half-set $A_0$ and $m(\mathbb{G})\le (1-\epsilon/100)n$. Then, there exists an $O(1/n^2)$-spread measure on $\{E(P_1)\cup E(P_2):(P_1,P_2)\in \mathcal{P}'(\mathbb{G})\}$. In addition, there exists a collection of $\zeta n^2$ sets $\{E(P_1)\cup E(P_2):(P_1,P_2)\in \mathcal{P}'(\mathbb{G})\}$ that are pairwise edge-disjoint.
\end{lemma}

We first prove \Cref{thm:extremalsimple} assuming the above lemma and then prove it right after.
\begin{proof}[{\bf Proof of \Cref{thm:extremalsimple} in Case~II}]
To prove the first part of \cref{thm:extremalsimple}, when $m(\mathbb{G})\le (1-\epsilon/100)n$, first sample $(P_1,P_2)\in \mathcal{P}'(\mathbb{G})$ according to the spread measure given by \Cref{lem:cleanupcliques} and let $A,B,C$ be as in \Cref{def:cleanupcliques}. Then, partition $C$ into two sets $C_A$ and $C_B$ of size $|A|-1$ and $|B|-1$ respectively. Next, take a random bipartition of $A$ into two equal sets $A_1,A_2$ with $v_1(P_1)\in A_1$ and $v_1(P_2)\in A_2$ and apply \Cref{lem:bipartite} (whose hypotheses can be checked as a routine) to sample a $v_1(P_1)$-$v_1(P_2)$ rainbow Hamilton path $P_A$ on $A$ that is $C_A$-transversal. Similarly sample a $v_2(P_1)$-$v_2(P_2)$ rainbow Hamilton path $P_B$ on $B$ that is $C_B$-transversal. Then concatenate the paths $P_1$, $P_2$, $P_A$, and $P_B$ to get a Hamilton cycle drawn from an $O(1/n^2)$-spread distribution that is $[n]$-transversal.

To construct $\zeta n^2$ pairwise edge-disjoint Hamilton $\mathbb{G}$-transversals, first let $\{S_1,\ldots,S_{\zeta n^2}\}$ be a collection of $\zeta n^2$ sets $\{E(P_1)\cup E(P_2):(P_1,P_2)\in \mathcal{P}'(\mathbb{G})\}$ that are pairwise edge-disjoint. Then, for $i\in [\zeta n^2]$, construct a new Hamilton $\mathbb{G}$-transversal, as before, that spans $S_i$ and avoids all the edges that are spanned by the Hamilton transversals constructed so far or by the paths in $\{S_1,\ldots,S_{\zeta n^2}\}$. This is possible because, at each step, we avoid at most $2\zeta n^2$ edges incident to each vertex.
\end{proof}

\begin{proof}[{\bf Proof of \Cref{lem:cleanupcliques}}]
The proof of this lemma is similar to the proof of \Cref{lem:cleanupbipartite}, so we only present a sketch of it. 
Set $B_0=\ccA_0$. Let $C_1$ be the set of indices of the $(1-\epsilon)$-bipartite graphs. Let $C_2$ be the set of indices $c \in [n]\setminus C_1$ such that $e_{G_c}(A_0,B_0)\ge 0.1 n^2$, and $C_3=[n]\setminus (C_1\cup C_2)$. Since $\mathbb{G}$ is a family of edge-minimal Dirac graphs (and thus \Cref{lem:complement is also sparse} applies for each of these graphs), $\mathbb{G}$ is $\alpha$-extremal with respect to $A_0$, and no graph with index in $C_2$ is $(1-\epsilon)$-bipartite, we have that $|C_2|\le \frac{\alpha n^3}{\epsilon n^2/4} \le \epsilon^2 n$. Therefore, 
\begin{equation}\label{eq:lower bound on C3}
|C_3|\ge (\epsilon/100-\epsilon^2)n.
\end{equation}

\textbf{Step 1: Modify the sets {\boldmath $A_0,B_0$}.}
As in Step~1 of Case~I, we modify the sets $A_0, B_0$ to obtain a partition $A_1, B_1$ of $V(\mathbb{G})$ such that $|A_0\triangle A_1|\le 18\alpha n$ and $\sum_{c\in C_3}d_{G_c}(v,A_1)\ge 1/3 \sum_{c\in C_3}d_{G_c}(v)$ for every $v\in A_1$ and $\sum_{c\in C_3}d_{G_c}(v,B_1)\ge 1/3 \sum_{c\in C_3}d_{G_c}(v)$ for every $v\in B_1$.

In the remainder of the proof, we will focus on proving the first part of \Cref{lem:cleanupcliques}. For the second part of \Cref{lem:cleanupcliques}, for every $i\in [\zeta n^2]$, perform the following five steps, always selecting edges that have not been chosen so far (as in the proof of \Cref{lem:cleanupbipartite}).

\textbf{Step 2: Adjusting the parity of~{\boldmath $C_1\cup C_2$}.}
The first paragraph motivates this step, while the following paragraphs implement it. In order to have a Hamilton cycle, we must use a positive even number of edges in $A_1\times B_1$. We will use all the colors in $C_1\cup C_2$ for edges in $A_1\times B_1$. Thus, if $|C_1\cup C_2|$ is odd, then we need to use one color from $C_3$ to color an edge in $A_1\times B_1$. Moreover, when $C_1\cup C_2 = \emptyset$, then we need to use two colors from $C_3$ to color an edge in $A_1\times B_1$. We will more generally do this when $|C_1\cup C_2|$ is even (to reduce the number of case distinctions).

If $|C_1\cup C_2|$ is odd, then let $c'$ be a random color in $C_3$ and $e'$ be a random edge of $G_{c'}$ in $A_1\times B_1$. For the justification of appropriate spreadness, note that for each of the $\Omega(n)$ colors in $C_3$ there are at least $n/2$ eligible edges from $A$ to $B$, yielding $\Omega(n^2)$ choices in total. We then set $C_0=\{c'\}$ and $E_0=\{e'\}$. 

If $|C_1\cup C_2|$ is even, then let $c', c''$ be two random colors in $C_3$ and $e', e''$ be two non-adjacent random edges of $G_{c'}$ and $G_{c''}$ induced by $A_1\times B_1$. The appropriate spreadness can be argued easily as before. We then set $C_0=\{c', c''\}$ and $E_0=\{e', e''\}$. 

We finally remove the colors in $C_0$ from $C_3$ and let $A'_1 = A_1\setminus V(E_0)$ and $B'_1 = B_1\setminus V(E_0)$. After this modification, we note that $|A_0\triangle A'_1|\le 40\alpha n$ and 
\begin{equation}\label{eq:step2 new}
\text{$\sum_{c\in C_3}d_{G_c}(v,A'_1)\ge 2|C_3|n/7$ for every $v\in A'_1$ and $\sum_{c\in C_3}d_{G_c}(v,B'_1)\ge 2|C_3|n/7$ for every $v\in B'_1$.}
\end{equation}

\textbf{Step 3: Cover vertices of ``small degree" using two {\boldmath $C_3$} rainbow paths.}
Let $V_A$ and $V'_A$ be the sets of vertices $v\in A'_1$ such that $\sum_{c\in C_3}d_{G_c}(v,A'_1) \le (1-\epsilon^2) |C_3||A'_1|$ and $\sum_{c\in C_1}d_{G_c}(v,B'_1) \le (1-\epsilon^2) |C_1||B'_1|$, respectively. The condition $|A_0\triangle A'_1|\le 40\alpha n$ and $\alpha$-extremality of $\mathbb{G}$ imply that $\epsilon ^2 |V_A| |C_3||A'_1| - 40\alpha n^3\le \alpha n^3$ and $\epsilon ^2 |V'_A| |C_1||B'_1| - 40\alpha n^3\le \alpha n^3$. Therefore, using \eqref{eq:lower bound on C3}, we have $|V_A|\le \epsilon^3 n$ and $|V'_A|\le \epsilon^3 n$. Similarly, let $V_B$ and $V'_B$ be the set of vertices $v\in B'_1$ such that $\sum_{c\in C_3}d_{G_c}(v,B'_1) \le (1-\epsilon^2) |C_3||B'_1|$ and $\sum_{c\in C_1}d_{G_c}(v,A'_1) \le (1-\epsilon^2) |C_1||A'_1|$, respectively. Then $|V_B|\le \epsilon^3 n$ and $|V'_B|\le \epsilon^3 n$.

Now,  \eqref{eq:lower bound on C3} and \eqref{eq:step2 new} imply that for every  $C'\subseteq C_3$ of size at most $\epsilon^2 n$, $U_A\subseteq A_1'$ of size at most $\epsilon^2 n$, and pair of vertices $u,v\in A_1'$ there exist at least $|C_3|n/20$ disjoint rainbow paths $u,uw,w,wv,v$ such that $w\in A_1'\setminus U_A$ and the colors of $uw$ and $wv$ belong to $C_3\setminus C_3'$. Thus we can cover the vertices in $V_A\cup V'_A$ by a $C_3$-rainbow path of length $\epsilon^2 n > 2|V_A\cup V'_A|+1$ that is spanned by $A$, in an $O(1/n^2)$-spread manner. Similarly, cover the vertices in $V_B\cup V'_B$ by a $C_3$-rainbow path of length $\epsilon^2 n$ that is spanned by $B$ (using $\epsilon^2 n$ new colors). Let $P_A$ and $P_B$ be the corresponding paths, and remove from $C_3$ all the colors used so far. 

\textbf{Step 4: Construct a {\boldmath $C_2$}-transversal matching {\boldmath $E_1$} spanned by~{\boldmath $A_1\times B_1$}.}
We build a $C_2$-transversal matching $E_1$ spanned by $(A'_1\setminus V(P_A))\times (B'_1\setminus V(P_B))$ using a random greedy way, ensuring the appropriate spreadness conditions, as done before. 

\textbf{Step 5: Construct a {\boldmath $C_1$}-rainbow linear forest {\boldmath $E_2$} spanned by {\boldmath $A_1\times B_1$}.} 
If $|C_1|<\epsilon^2 n$, then, similar to the last step, construct a $C_1$-transversal matching $E_2$ spanned by $(A'_1\setminus V(P_A))\times (B'_1\setminus V(P_B))$ such that $E_1\cup E_2$ is a matching. 
Else if $|C_1|>\epsilon^2 n$, then we will invoke \Cref{lem:bipartite} to construct a $C_1$-transversal path using edges spanned by edges in $(A'_1\setminus (V(P_A)\cup V(E_1)))\times (B'_1\setminus (V(P_B)\cup V(E_1)))$ only and let $E_2$ be its edge set.
Indeed, the $(1-\epsilon)$-bipartiteness of the graphs with color indices in $C_1$ implies (1) for the application of \Cref{lem:bipartite} (with $\epsilon = \epsilon/10$), and the properties derived in Step~3 imply the corresponding properties of (2) and (3).

\textbf{Step 6: Concatenate {\boldmath $(E_0\setminus e')\cup E_1\cup E_2 \cup E(P_A)\cup E(P_B)$} to a rainbow path~{\boldmath $P_1$}.} 
To concatenate $(E_0\setminus e')\cup E_1\cup E_2 \cup E(P_A)\cup E(P_B)$ to a rainbow path avoiding the vertices in $e' := a'b'$ (with $a'\in A_1$ and $b'\in B_1$), we need to join at most $1 + \epsilon^2 n + \epsilon^2 n + 2$ paths that span in total at most $(1-\epsilon/1000) n$ vertices from each of $A_1$ and $B_1$. This can be done as Step~5 of Case~I, using the minimum degree condition on graphs $G_c$ for $c\in C_3$ that is enforced by Step~3. Note that the path $P_1$ has one of its endpoints in $A_1$ and the other in $B_1$, call them $a_1$ and $b_1$ respectively. We will now define $A = A_1\setminus (V(P_1)\setminus \{a_1\})$ and $B = B_1\setminus (V(P_1)\setminus \{b_1\})$. By possibly extending the path $P_1$ by one more edge (in possibly both sides), one can ensure the parity condition in (4) (i.e., $|A|$ and $|B|$ are even). We can now apply \Cref{def:cleanupcliques} with $P_2 = e'$ to see that all properties (1)--(4) are satisfied. This finishes the proof of \Cref{lem:cleanupcliques}.
\end{proof}

\subsection{Exceptional families of Dirac graphs: Proof of \Cref{thm:superextremal}}
Let $\mathbb{G}, A,C,(e,i), \mathbb{G}^{-i}(A,C)$ be as in the statement of \cref{thm:superextremal}. Let $0<\epsilon\ll 1$. Let $C'$ be the set of indices $i\in C$ such that $G_i$ is  $(1-\epsilon)$-bipartite. Note that as $\mathbb{G}$ is exceptional, we have that $|A|=n/2$ and  
\begin{align*}
0.1n^2&\ge r(\mathbb{G}) \ge \sum_{i\in C\setminus C'}\left(e_{G_i}(A)+ e_{G_i}(\overline {A})\right) + \frac{1}{2}\sum_{v\in A}\sum_{i\in C'} d_{G_i}(v,A) + \frac{1}{2} \sum_{v\in \overline{A}}\sum_{i\in C'} d_{G_i}(v,\overline{A}) + \sum_{i\in \overline {C}} e_{G_i} (A,\overline{A})
\\&\ge \epsilon n^2 |C\setminus C'| + \frac{1}{2} \sum_{v\in A}\sum_{i\in C'} d_{G_i}(v,A) + \frac{1}{2} \sum_{v\in \overline{A}}\sum_{i\in C'} d_{G_i}(v,\overline{A})+ (n-|C|) \cdot \frac{n}{2}.
\end{align*}
Thus, $|C\setminus C'|\le 0.1/\epsilon$, $|C|\ge 0.8n$, and $\sum_{i\in C'} d_{G_i}(v,U)\le 0.1n^2$ for $v\in U$ and $U\in \{A,\overline{A}\}$. Now one has to construct a rainbow path $P$  that is $\mathbb{G}^{-i}(A,C)$-rainbow and uses all the colors not in $C'$ (in an $O(1/n^2)$-spread manner) for each pair $(e,j)$ such that $j\in C$. $P$ should start from one of the endpoints of $e$ and have odd length if $i\in C$ and even length otherwise. Then it suffices to apply \cref{lem:bipartite} to extend $P$ to a $\mathbb{G}^{-i}(A,C)$-rainbow Hamilton path that joins the endpoints of $e$. These steps can be done verbatim as in the proof of \Cref{thm:extremalsimple} in Case~I (starting from Step~3). 

For the second part of \cref{thm:superextremal}, let $S$ be a set of  $\min\{\zeta n^2, r(\mathbb{G})\}$ pairs $(e,i)$ such that $e\in E(G_i)$ and either $i\in C$ and $e\notin A\times \overline{A}$ or $i\in \overline{C}$ and $e\in A\times \overline{A}$. Then, one can construct a $\mathbb{G}^{-i}(A,C)$-rainbow Hamilton path $P_{(e,i)}$ that joins the endpoints of $e$ for $(e,i)\in S$ recursively, as above, always choosing edges that have not been chosen so far.

%%%%%%%%%%%%%%%%%%%%%%%%%%%%%%%%%%%%%%%%%%%%%%%%%%%%%%%%%%%%%%%%%%
\section{Random subsets preserve the non-extremalness of a family: Proof of \Cref{lem:property preserve}} \label{section:property preserve} 

For the proof of \Cref{lem:property preserve}, we consider $3$ cases. In the first case, a constant portion of graphs in $\mathbb{G}$ are not $\alpha$-extremal. To deal with this case, we show that if $G$ is a Dirac graph that is not $\alpha$-extremal and $V'=V(p)$, then $G[V']$ is likely not $\beta$-extremal for some $0<\beta\ll \alpha$ (see \Cref{lem:property graph}). This implies that, with non-zero constant probability, a constant portion of graphs in $\mathbb{G}_{C(q)}[V(p)]$ are not $\beta$-extremal, hence the family $\mathbb{G}_{C(q)}[V(p)]$ itself is not $\alpha'$-extremal for some $0<\alpha'\ll \beta$. This case is the hardest to deal with, and most of the arguments given in this section concern it. 

For the remaining two cases, we consider the distinction whether there exist sets $A_G$ for $G\in \mathbb{G}$ such that $G$ is $\alpha$-extremal with respect to $A_G$ and most of these sets are the same up to removing/adding a small number of vertices. 

\subsection{Auxiliary results}
We start with two deterministic auxiliary lemmas about edge-minimal Dirac graphs.
\begin{lemma} \label{lem:edges Dirac} 
Let $G$ be an $n$-vertex edge-minimal Dirac graph. Then every edge of $G$ has an endpoint whose degree is at most $(n+1)/2$.
\end{lemma}
\begin{proof}
Suppose for a contradiction that $G$ has an edge $uv$ such that $d(u),d(v)> (n+1)/2$. Then, $d(u),d(v)\ge (n/2)+1$. Let $G'$ be obtained from $G$ by removing the edge $uv$. Then both $u$ and $v$ have degree at least $n/2$ in $G'$. Consequently, since $G$ is Dirac, every vertex in $G'$ has degree at least $n/2$. Thus $G'$ is Dirac. This contradicts the edge-minimality of $G$.
\end{proof}

\begin{lemma} \label{lem:complement is also sparse}
Let $0 < 1/n \ll \alpha \le 1$. Let $G$ be an $n$-vertex graph with minimum degree at least $(1/2 - \alpha^2)n$. Suppose every pair of adjacent vertices $u,v\in V(G)$ satisfies $\min\{d(u), d(v)\} \le (1/2+\alpha^2)n$. Let $A\subseteq V(G)$ be a half-set such that $e(A)\le \alpha n^2$. Then, $e(\overline{A})\le 3\alpha n^2$.
\end{lemma}
\begin{proof}
Since $G$ has minimum degree at least $(1/2 - \alpha^2)n$, we have 
\[
e(A,\overline{A})\ge (1/2 - \alpha^2)n\cdot |A| - 2e(A) \ge \left(1/4 - 2\alpha^2 - 2\alpha\right)n^2.
\]
Denote by $A'$ the set of all vertices in $\overline{A}$ that have degree at most $(1/2+\alpha^2)n$. The hypothesis implies that no edge in $G$ is spanned by $\overline{A}\setminus A'$, i.e., $e(\overline{A}\setminus A')=0$. Thus, the number of edges with at least one end-point in $\overline{A}$ is at most $(1/2+\alpha^2)n\cdot |A'| + |A|\cdot |\overline{A}\setminus A'|$. Consequently, using the above inequalities, we have 
\[
e(\overline{A})\le (1/2+\alpha^2)n\cdot |A'| + |A|\cdot |\overline{A}\setminus A'| - e(A,\overline{A}) \le (1/2+\alpha^2)n\cdot |\overline{A}| - \left(1/4 - 2\alpha^2 - 2\alpha\right)n^2 \le 3\alpha n^2.
\]
\end{proof}

We next show that the non-extremalness of an edge-minimal Dirac graph is preserved in the subgraph induced by a random vertex subset with good probability. 

\begin{lemma} [Extremality preserving lemma for a graph] \label{lem:property graph}
Let $0 < 1/n   \ll \alpha' \ll \alpha \le 1$. Let $p\in [0,1]$. 
If $G$ is an edge-minimal non $\alpha$-extremal Dirac graph on the vertex set $V$ of size $n$, then the graph induced by $V(p)$ is not $\alpha'$-extremal with probability at least $1-5(pn)^2e^{-\Omega(\alpha' pn)}$.
\end{lemma}

\begin{proof}
Let $\delta$ be such that $0 < 1/n \ll   \alpha' \ll \delta \ll \alpha \le 1$. We consider two cases based on the number of vertices of $G$ with degree significantly larger than $n/2$.

\textbf{Case 1. {\boldmath $G$} has at least {\boldmath $\delta^{1/2} n$} vertices with degree at least {\boldmath $(1/2 + \delta) n$}.} 
The Chernoff bound (\Cref{chernoff}) gives that $|V(p)-pn|\le \delta^2 pn$ with probability $1-e^{-\Omega(\alpha' pn)}$. Thereafter the probability that there is a vertex in $V(p)$ of degree less than $(1/2-\delta^2)np$ is at most
\[
\sum_{v\in V}\Pr(v\in V(p)) \Pr( |N(v)\cap V(p)| <  (1/2-\delta^2)np) \le np \Pr(Bin(n/2,p) <  (1/2-\delta^2)np) \le np e^{-\Omega(\alpha' pn)}, 
\]
by the Chernoff bound (\Cref{chernoff}). Similarly, with probability $1-npe^{-\Omega(\alpha' pn)}$, at least  $8\delta^{1/2} n p/9$ of the vertices with degree at least $(1/2 + \delta) n$ belong to $V(p)$, and those vertices have degree at least$(1/2 + 8\delta/9) np$ in $G[V(p)]$. Thus, with probability $1-npe^{-\Omega(\alpha' pn)}$, $G[V(p)]$ has at most   $(1+\delta^2)pn$ vertices and at least 
\[((1-\delta^2)np \cdot (1/2 + \delta/2) np + (8\delta^{1/2} n p/9)\cdot  (8\delta/9) np)/2\ge (1/4+\delta^{3/2}/2)(np)^2\] edges. 
In such an event, we have that $G' = G[V(p)]$ is not $\alpha'$-extremal. Indeed, if we let $A$ be a half-set of $G'$, then 
\[
e_{G'}(A,\overline{A}) \ge e(G') - (e_{G'}(A) + e_{G'}(\overline{A})) \ge (1/4+\delta^{3/2}/4)(pn)^2 - 2\binom{\lceil (1+\delta^2)pn/2 \rceil}{2} > \alpha' |V(G')|^2
\]
\[
\text{and} \;\; e_{G'}(A) + e_{G'}(\overline{A}) \ge e(G') - e_{G'}(A,B) \ge (1/4+\delta^{3/2}/4)(pn)^2 - ((1+\delta^2)pn)^2/4 > 4\alpha' |V(G')|^2,
\]
thus by \Cref{lem:complement is also sparse}, we have $e_{G'}(A) > \alpha' |V(G')|^2$.

\textbf{Case 2. {\boldmath $G$} has at most {\boldmath $\delta^{1/2} n$} vertices with degree at least {\boldmath $(1/2 + \delta) n$}.} 
\begin{claim}\label{clm:codegree}
There are at least $\delta n^2$ pairs of vertices $u,v$ in $G$ with co-degree satisfying ${\delta n \le d(u,v) \le (1/2 -\delta) n}$.
\end{claim}
\begin{proof}
Suppose for a contradiction that at most $\delta n^2$ pairs of vertices $u,v$ in $G$ satisfy ${\delta n \le d(u,v) \le (1/2 -\delta) n}$. Let $E_1=\{uv \in E(G) :d(u,v)> (1/2 -\delta) n \}$ and $E_2=\{uv\in E(G):d(u,v)<\delta n\}$. It follows from our assumption that $|E(G)\setminus (E_1\cup E_2)|\le \delta n^2$. Therefore,  either $|E_1|\ge 0.1n^2$ or $|E_2|\ge 0.1n^2$. We thus consider the two corresponding cases.

\textbf{Case I: {\boldmath $|E_1|\ge 0.1n^2$}.} 
Let $E_1'$ be the subset of edges of $E_1$ that are not incident to a vertex of degree larger than $(1/2+\delta)n$. Then $|E_1'|\ge 0.1n^2-\delta^{1/2}n^2\ge 0.09n^2$. In addition, since $G$ is edge-minimal, every edge in $E_1'$ is incident to a vertex of degree $\lceil n/2 \rceil$ (see \cref{lem:edges Dirac}). Thus there exists a vertex $v$ of degree $\lceil n/2 \rceil$ and a set of vertices $U\subset N(v)$ of degree at most $(1/2+\delta)n$ such that $|U|=0.09n$ and $d(u,v)>(1/2-\delta)n$ for every $u\in U$.

Let $A=N(v)$ and $A_{good}\subseteq A$ be the set of vertices with at least $0.04n$ neighbors in $U$. Let ${A_{bad}= A\setminus A_{good}}$. Observe that each vertex $u\in U$ has $d(u,v)> (1/2-\delta)n$ neighbors in $A$ and $|A|=\lceil n/2 \rceil$. Thus, by double counting the non-edges from $U$ to $A_{bad}$ in $G$, we have that
\[
0.5|A_{bad}| \cdot (0.09-0.04)n \le \delta n \cdot |U| = 0.09 \delta n^2.
\]
Hence, $|A_{bad}|\le 4\delta n$ and $|A_{good}|\ge n/2 - 4\delta n$. It follows that every pair of vertices $a\in A_{good}, u\in U$ has at least $0.04n-|A\setminus N(u)|\ge 0.04n-\delta n-1>\delta n$ common neighbors in $U\subseteq A$. So, no edge from $A_{good}$ to $U$ belongs to $E_2$ and at most $\delta n^2$ of those edges belong to $E(G)\setminus (E_1\cup E_2)$.
Hence, at least $|A_{good}| - \delta n^2/(0.04 n) \ge (1/2-\delta^{1/2})n$ vertices  $a\in A_{good}$ are incident to a vertex $u\in U$ via an edge in $E_1$. For each such pair $a\in A_{good}$, $u\in U$, the vertex $u$ has at most $(1/2+\delta)n-d(v,u)\le 2\delta n$ neighbors not in $A$ and therefore the number of neighbors of $a\in A_{good}$ in $A$ is at least $d(a,u)-2\delta n\ge (1/2-\delta)n-2\delta n\ge (1/2-3\delta)n$. 
It follows that $A$ spans at least $(1/2-3\delta)(1/2-\delta^{1/2})n/2\ge (1/8-4\delta^{1/2})n^2$ edges. Finally, since  $G$ has at most $\delta^{1/2} n$ vertices with degree at least $(1/2 + \delta) n$, we have
\[
e(A,\overline{A})\le \delta^{1/2} n^2 +  \left(|A|(1/2 + \delta) n-2 \left(1/8-4\delta^{1/2}\right)n^2 \right)< \alpha n^2.
\]
This contradicts the assumption that $G$ is not $\alpha$-extremal.

\textbf{Case II: {\boldmath $|E_2|\ge 0.1n^2$}.} 
The analysis of this case is similar to the analysis of Case~I. In Case~II, there exists a vertex $v$ of degree $\lceil n/2 \rceil$ and a set of vertices $U\subset N(v)$ of degree at most $(1/2+\delta)n$ such that $|U|=0.09n$ and $d(u,v)<\delta n$ for every $u\in U$. Let $A=\overline{N(v)}$ and $A_{good}\subseteq A$ be the set of vertices with at least $0.04n$ neighbors in $U\subset N(v)$, and $A_{bad}= A\setminus A_{good}$. Now every vertex in $U$ has fewer than $\delta n$ neighbors in common with $v$, thus it has at least $(1/2-\delta)n$ neighbors in $A$. By double-counting the non-edges from $U$ to $A_{bad}$, one has $|A_{bad}|\le \delta n|U|/(|U|-0.04n)\le 4\delta n$. Now, since $|E(G)\setminus (E_1\cup E_2)|\le \delta n^2$, each vertex in $U$ has fewer than $\delta n$ neighbors in common with $v$ (thus in $U$), every vertex in $A_{good}$ has at least $0.04n$ neighbors in $U$, every pair of vertices $a\in A_{good}, u\in U$ with $d(a)\le (1/2 + \delta)n$ has at most $d(a) - (d(a,U) - d(u,U)) \le (1/2 + \delta)n - (0.04n - \delta n) < (1/2 - \delta)n$ common neighbors, so such pairs $au$ do not belong to $E_1$. Additionally, since $|A_{good}|\ge \lfloor n/2 \rfloor - 4\delta n$ and $G$ has at most $\delta^{1/2} n$ vertices with degree at least $(1/2 + \delta) n$, we have that at least $\lfloor n/2 \rfloor - 4\delta n - \delta^{1/2}n - \delta n^2/(0.04n) \ge(1/2-2\delta^{1/2})n$ vertices  $a\in A_{good}\subseteq A$ are incident to a vertex $u\in U$ of degree at most $(1/2+\delta)n$ via an edge in $E_2$. Each such vertex $a$ has at most $\delta n$ common neighbors with $u\in U$, thus at most $\delta n + (|A| - d(u,A)) \le 2\delta n$ neighbors in $A$. It follows that $A$ spans at most $2\delta^{1/2} n^2 + 2\delta n^2 < \alpha n^2$ edges, which contradicts the assumption that $G$ is not $\alpha$-extremal.
\end{proof}

Using \Cref{clm:codegree} and the Chernoff bound, we next deduce that $G[V(p)]$ is not $\alpha'$-extremal with the desired probability. Denote by $P$ the set of pairs of vertices satisfying the assertion of \Cref{clm:codegree}. 
Now consider a $p$-random subset $V(p)$ of the vertex set $V$ of $G$. An application of the  Chernoff bound gives the following with error probability $5(pn)^2e^{-\Omega(\alpha' pn)}$. 
\begin{enumerate}
    \item $||V(p)|-pn|\le \delta^3 p n$. \label{eq:size of X}
    \item $G[V(p)]$ has minimum degree at least $(1/2-\delta^3)pn$. 
    \item $|d_{G[V(p)]}(u)-pd_G(u)| \le \delta^3 pn$ for every $u\in V(p)$. In particular, since $G$ is an edge-minimal Dirac graph, if $uv\in E(G[V(p)])$, then $\min\{d_{G[V(p)]}(u),d_{G[V(p)]}(v)\} \le (1/2+\delta^3)pn$.  
    \item The number of pairs in $P$ contained in $V(p)$ is at least $0.9 \delta (pn)^2$ (here we are using that for $u\in V(p)$ if $u$ belongs to at least $\delta^2 n$ pairs in $P$ then the number of vertices $v \in V(p)$ such that $(u,v)\in P$ is concentrated by the Chernoff bound). Denote the set of such pairs by~$P'$.
    \item $|d_{G[V(p)]}(u,v)- p d_G(u,v)|\le \delta^3 pn/4$ for every pair $\{u,v\}\in P'$. 
\end{enumerate}
Indeed, each of the above fails with probability at most $(pn)^2e^{-\Omega(\alpha' pn)}$. For example, the probability that (5) fails, i.e.,  there exists a pair $\{u,v\}\in P'$ such that  $|d_{G[V(p)]}(u,v)- p d_G(u,v)|\le \delta pn/4$ is at most
\begin{align*}
&\sum_{\{u,v\}\in P}\Pr(u,v\in V(P) \text{ and } ||V(P)\cap N(u)\cap N(v)|-p d_G(u,v)|> \delta pn/4)
\\&\le \sum_{\{u,v\}\in P} p^2 \Pr( |Bin(d_G(u,v),p)-p d_G(u,v)|> \delta pn/4)
\le (np)^2 e^{-\Omega(\alpha' pn)},
\end{align*}
where the last inequality follows from the Chernoff bound. 

In the event that $G[V(p)]$ satisfies the above properties (1)-(5), it follows that it is not $\alpha'$-extremal by the following claim. 
\begin{claim}
Let $0 < 1/n \ll \alpha' \ll \delta \ll 1$. Let $G$ be an $n$-vertex graph with the following properties.
\begin{enumerate}
    \item \label{eq:min degree} $G$ has minimum degree at least $(1/2- 2\delta^3)n$.
    \item \label{eq:max degree} For every pair of adjacent vertices $u,v\in V(G)$, we have $\min\{d(u), d(v)\} \le (1/2+2\delta^3)n$.
    \item \label{eq:co degree} There are $\delta n^2/2$ pairs of vertices $u,v$ with co-degree satisfying $\delta n/2\le d(u,v)\le (1/2 - \delta/2)n$.
\end{enumerate}
Then, $G$ is not $\alpha'$-extremal.
\end{claim}
\begin{proof}
We will prove a contrapositive as follows. For that suppose that $G$ is an $n$-vertex $\alpha'$-extremal graph satisfying \eqref{eq:min degree} and \eqref{eq:max degree}. It suffices to prove that $G$ must violate \eqref{eq:co degree}.

Since $G$ is $\alpha'$-extremal graph there exist some half-set $A$ so that $\min\{e(A),e(A,\overline{A})\}\le \alpha' n^2$. 

\textbf{Case I: {\boldmath $e(A,\overline{A}) \le \alpha' n^2$}.} 
By \eqref{eq:min degree}, we know 
\begin{align} \label{eq:lower bound on twice of A}
2e(A)\ge (1/2- 2\delta^3)n \cdot |A| - e(A,\overline{A}) \ge (1/4-\alpha')n^2 - \alpha' n^2 \ge (1/4-2\alpha')n^2.
\end{align}
Denote by $S_1, S_2$ the set of all vertices $v\in A$ satisfying $d_{\overline{A}}(v)\ge \alpha'^{1/3} n$ and $d_A(v)\le (1/2 - \alpha'^{1/3}) n$, respectively. Similarly, denote by $S'_1, S'_2$ the set of all vertices $v\in \overline{A}$ satisfying $d_{A}(v)\ge \alpha'^{1/3} n$ and $d_{\overline{A}}(v)\le (1/2 - \alpha'^{1/3}) n$, respectively.
 
We claim that each of $S_1, S_2, S'_1, S'_2$ has size at most $\alpha'^{1/3} n$. First assume that $|S_1| > \alpha'^{1/3} n$, then 
\[
e(A,\overline{A}) \ge \sum_{v\in S_1} d_{\overline{A}}(v) \ge |S_1|\cdot \alpha'^{1/3} n \ge \alpha'^{2/3} n^2,
\]
a contradiction. Similarly, if $|S_2| > \alpha'^{1/3} n$, then 
\[
2e(A)\le n^2/4 - \sum_{v\in S_2} (|A| - 1 - d_{A}(v)) \le n^2/4 - |S_2|\cdot \alpha'^{1/2} n \le (1/4- \alpha'^{5/6} )n^2,
\]
a contradiction to \eqref{eq:lower bound on twice of A}. Similarly, we can bound the sizes of $S_2$ and $S'_2$ as desired. Now note that every pair of vertices $\{u,v\}$ such that $\delta n/2\le d(u,v)\le (1/2 - \delta/2)n$  must intersect $S_1\cup S_2\cup S_1'\cup S_2'$. Indeed, let $\{u,v\}$ be such a pair, and without loss of generality, we may assume that they have at least $\delta n/4$ common neighbors in $A$. If $u\in \overline{A}$, then $u\in S_1'$. Similarly if $v\in \overline{A}$, then $v\in S_1'$. Thus we may assume that both $u,v$ belong to $A$. Then, as $|N(u)\cup N(v)|\ge 2(1/2- 2\delta^3)n- (1/2 - \delta/2)n \ge (1/2+\delta/4)n$, we have that one of $u,v$ has at least $\delta n/4\ge {\alpha'}^{1/3}n$ neighbors in $\overline{A}$ and therefore belongs to $S_1$. It follows that there are at most $n\cdot 4\alpha'^{1/3}n<\delta n^2/2$ pairs $\{u, v\}$ such that $\delta n/2\le d(u,v)\le (1/2 - \delta/2)n$.

\textbf{Case II: {\boldmath $e(A) \le \alpha' n^2$}.} 
\Cref{lem:complement is also sparse} gives that $e(\overline{A}) \le 3\alpha' n^2$. Let $S$ ($S'$ respectively) be the set of vertices in $A$ (respectively, in $\overline{A}$) with at least $\delta n/3$ neighbors in $A$ (respectively, in $\overline{A}$). Then $|S|,|S'|\le \frac{6\alpha' n^2}{\delta n/3}< \delta^2 n$. Moreover, (1) gives that every vertex in $A\setminus S$ (respectively, $\overline{A}\setminus S'$) has at least $(1/2-0.1\delta)n$ neighbors in $\overline{A}$ (respectively, in $A$). Therefore every pair of vertices $\{u,v\}$ such that $\delta n/2\le d(u,v)\le (1/2 - \delta/2)n$  must intersect $S\cup S'$, hence there are less than $\delta n^2/2$ such pairs.
\end{proof}
This finishes the proof of \Cref{lem:property graph}.
\end{proof}

Two $\alpha$-extremal graphs $G, G'$ on the same vertex set are said to be \textit{$\delta$-good} if there exists a pair of half-sets $A_G, A_{G'}$ such that $G$ and $G'$ are $\alpha$-extremal with respect to $A_G$ and $A_G'$ respectively and $\delta n \le |A_G \cap A_{G'}|\le (1/2 - \delta) n$. We often say that $G, G'$ is $\delta$-good with respect to $A_G, A_{G'}$. In the proof of \Cref{lem:property preserve}, we will use the following lemma for pairs of $\delta$-good $\alpha$-extremal graphs.

\begin{lemma} \label{lem: reduction to pseudo}
Let $0 < 1/n \ll \alpha \ll \delta < 1$. Let $G_1, G_2$ be a pair of $\delta$-good $\alpha$-extremal graphs on $n$ vertices with minimum degree at least $(1/2-\alpha)n$. Then, for every half-set $A$, we have 
\[
\min \set{e_{G_1}(A),e_{G_1}(A,\overline{A})}\ge \alpha n^2 \;\;\; \text{or} \;\;\; \min \set{e_{G_2}(A),e_{G_2}(A,\overline{A})}\ge \alpha n^2.
\] 
\end{lemma}
\begin{proof}
Let $A_1,A_2$ be such that $G_1,G_2$ is $\delta$-good with respect to $A_1,A_2$. Then, $|A_1\cup A_2|\ge (1/2+\delta)n$ and $|A_1\cap A_2|\le (1/2-\delta)n$. In particular, for every half-set $A$, we have $\delta n/2\le |A\cap A_1|\le (1/2-\delta/2)n$ or $\delta n/2\le |A\cap A_2|\le (1/2-\delta/2)n$. 

Without loss of generality, assume that $\delta n/2\le |A\cap A_1|\le (1/2-\delta/2)n$. Since $G_1$ is $\alpha$-extemal with respect to $A_1$, either $e_{G_1}(A_1)\le \alpha n^2$  or  $e_{G_1}(A_1,\overline{A_1})\le \alpha n^2$. If $e_{G_1}(A_1)\le \alpha n^2$, then for $B\in \{A,\overline{A}\}$, $|\overline{A_1}\setminus B|\le (1/2-\delta/2)n$ and therefore,
\begin{align*}
e_{G_1}(A\cap A_1, B\cap \overline{A_1}) 
&\ge \sum_{v\in A\cap A_1} (d_{G_1}(v)-|\overline{A_1}\setminus B|)-2e_{G_1}(A_1) 
\\&\ge |A\cap A_1|\left((1/2-\alpha)n- (1/2-\delta)n \right) - 2\alpha n^2 \ge \delta^2n^2/4 - 2\alpha n^2 >2\alpha n^2. 
\end{align*}
Thus, $\min \set{e_{G_1}(A),e_{G_1}(A,\overline{A})}\ge \min \set{e_{G_1}(A\cap A_1, A\cap \overline{A_1}),e_{G_1}(A \cap A_1,\overline{A}\cap \overline{A_1})} \ge \alpha n^2$. Similarly, if  $e_{G_1}(A_1,\overline{A_1})\le \alpha n^2$, then for $B\in \{A,\overline{A}\}$,
\begin{equation*}
e_{G_1}(A\cap A_1, B\cap A_1) 
\ge \sum_{v\in A\cap A_1} (d_{G_1}(v)-|A_1\setminus B|)-e_{G_1}(A_1,\overline{A_1}) >2\alpha n^2. 
\end{equation*}
Thus, $\min \set{e_{G_1}(A),e_{G_1}(A,\overline{A})}\ge \alpha n^2$.
\end{proof}

The next lemma states that the goodness of a pair is likely to be preserved under vertex sparsification.
\begin{lemma} \label{lem: delta good pairs is reserved}
Let $0 < 1/n \ll  \alpha \ll \delta < 1$. Let $G_1, G_2$ be a $\delta$-good pair of  $\alpha$-extremal graphs on the vertex set $V$ of size $n$ vertices with minimum degree at least $(1/2-\alpha)n$. Let $ p \in [0,1]$.  Then,  $G_1[V(p)], G_2[V(p)]$ is a $\delta/2$-good pair of $(\alpha/10)$-extremal graphs with minimum degree at least $(1/2-\alpha/10)np$ with probability at least $1-10npe^{-\Omega(\alpha np)}$.
\end{lemma}
\begin{proof}
Let $A_1,A_2$ be such that $G_1,G_2$ is $\delta$-good with respect to $A_1,A_2$. An application of the Chernoff bound gives that the following hold with probability at least $1-10npe^{-\Omega(\alpha np)}$.
\begin{itemize}
    \item[(1)] $G_i[V(p)]$ has minimum degree at least $(1/2-\alpha/3)np$ for $i=1,2$;
    \item[(2)] $||V(p)|-np|\le \alpha np/ 1000$;
    \item[(3)] $||V(p)\cap A_i|-np/2|\le \alpha np/ 1000$ for $i=1,2$;
    \item[(4)] $ 2\delta np/3 \le |(A_1\cap V(p)) \cap (A_2\cap V(p))|\le (1/2- 2\delta/3)np$; and
    \item[(5)] $||N_{G_i}(v)\cap W|p - |N_{G_i}(v)\cap W\cap V(p)| \le \alpha np/100 $ for $v\in V(p)$, $W\in \{A,\overline{A}\}$ and $i=1,2$. In particular,
\begin{itemize}
     \item[(5i)] $ |e_{G_i}(A_i)p^2 - e_{G_i[V(p)]}(A_i\cap V(p))| \le \alpha n^2p^2/10 $ for $i=1,2$;
    \item[(5ii)] $ |e_{G_i}(A_i,\overline{A_i})p^2 - e_{G_i[V(p)]}(A_i\cap V(p), \overline{A_i}\cap V(p) )| \le \alpha n^2p^2/10$ for $i=1,2$; and
\end{itemize}  
\end{itemize}   

Suppose that (1)-(5) hold. Let $i\in \{1,2\}$. (1) and (2) imply that $G_i[V(p)]$ has minimum degree at least  $(1/2-\alpha/10)|V(p)|$. Thereafter (2) and (3) imply that $A_i\cap V(p)$ can be turned to a half-set of $B_i$ of $G_i[V(p)]$ by adding or removing at most $\alpha np/ 500$ vertices. Thus, (2), and (5) imply that $G_i[V(p)]$ is $(\alpha/10)$-extremal with respect to $B_i$. Finally (4) implies that  $ 2\delta np/3 -2\alpha np/ 500 \le |B_1\cap B_2| \le (1/2- 2\delta/3)np+2\alpha np/500$. Therefore, $G_1[V(p)], G_2[V(p)]$ is a $\delta/2$-good pair with respect to $B_1,B_2$.
\end{proof}

\subsection{Putting everything together}

\begin{proof}[{\bf Proof of \Cref{lem:property preserve}}]
Let $\delta$ be such that $0 < 1/n,1/m \ll \alpha' \ll \delta \ll \alpha \le 1$. Let $C'\sim C(q)$ and $V'\sim V(p)$. Consider the new family $\mathbb{G}':= \mathbb{G}_{C'}[V']$. We split into three cases. 

\textbf{Case I: There are at least {\boldmath $\delta m$} graphs in {\boldmath $\mathbb{G}$} that are not {\boldmath $\delta$}-extremal.} 
Denote the index set of these non $\delta$-extremal graphs by $\tilde{C}\subseteq C$. Remove a number of elements from $\tilde{C}$ so that $|\tilde{C}|=\delta m$. Also, let $\delta'$ be such that $\alpha'\ll \delta'  \ll \delta$.

By the Chernoff bound, with probability $1-e^{-\Omega(\delta qm )}$, we have $||C'|-qm|\le 0.25qm$ and $||C'\cap \tilde{C}|-\delta q m|\le 0.25\delta qm$. Then, by \Cref{lem:property graph} we have that for each $G\in \mathbb{G}_{\tilde{C}}$ the graph $G[V']$ is not $\delta'$-extremal with probability $1- 5pn e^{-\Omega(-\delta' pn)}$. In the event that the above hold, for every half-set $A$ of $V'$,
\begin{align*} 
   \sum_{G\in \mathbb{G}'} \min\{ e_{G}(A), e_{G}(A,\overline{A}) \} 
   & \ge \sum_{i\in C'\cap \tilde{C}} \min\{ e_{G_i}(A), e_{G_i}(A,\overline{A}) \}   
   \\&  \ge  0.75 \delta q m \cdot \delta' |V'|^2 \ge (0.5 \delta \delta')|C'||V'|^2 \ge \alpha' |C'| |V'|^2.
\end{align*}
This proves that $\mathbb{G}'$ is non $\alpha'$-extremal with probability at least
\[
1-e^{-\Omega(\delta qm )} - (1.25\delta qm)(5(pn)^2)e^{-\Omega(\delta' pn)} \ge 1-e^{-\Omega(\alpha' qm )} - \alpha (qm)(pn)^2 e^{-\Omega(\alpha' pn)},
\]
as desired.

\textbf{Case II: There are at least {\boldmath $\delta m^2$} many {\boldmath $\delta$}-good pairs of {\boldmath $\delta$}-extremal graphs.} 
Then there are at least $0.5 \delta m$ graphs that belong to $0.5\delta m$ such pairs (call these graphs \textit{useful}). An application of the Chernoff bound gives that with probability at least $1-e^{-\Omega(\delta mq)}$, there will be at least $q\cdot 0.4 \delta m$ useful graphs $G$ in $\mathbb{G}_{C'}$. A second application of the Chernoff bound gives that for each useful graph $G$, conditioned on $G\in \mathbb{G}_{C'}$, with probability at least $1-e^{-\Omega(\delta mq)}$ there exists $q\cdot 0.4 \delta m$ graphs $H$ with the property that $G,H$ is a good pair and $G,H\in \mathbb{G}_{C'}$.
Let $\mathcal{S}$ be a maximal set of pairwise disjoint $\delta$-good pairs in $\mathbb{G}$ that belong to the family $\mathbb{G}_{C'}$ (as proper subsets). By the above we have that $|\mathcal{S}|$ has size at least $0.2 \delta qm$ with probability at least $1-e^{-\Omega(\delta mq)}$. Remove some pairs from $\mathcal{S}$ so that it has size exactly $0.2\delta qm$. Then,  \Cref{lem: delta good pairs is reserved} implies that with probability $1-(0.2\delta qm)(10np)e^{-\Omega(\alpha np)}$, for every pair $G,H$ in $\mathcal{S}$ we have that $G[V'], H[V']$ is a $\delta/2$-good pair of $(\alpha/10)$-extremal graphs with minimum degree at least $(1/2-\alpha/10)|V'|$.
In such an event,  \Cref{lem: reduction to pseudo} gives that for every half-set $A$ of $V'$, and pair $\{G,H\}\in \mathcal{S}$ we have that $\sum_{W\in\{G,H\}} \min \set{e_{W}(A),e_{W}(A,\overline{A})}\ge \delta'|V'|^2$. In extension, we have that 
\[
\sum_{G\in\mathbb{G'}} \min \set{e_{G}(A),e_{G}(A,\overline{A})}\ge \delta'|V(p)|^2||\mathcal{S}|\ge 0.2\delta \delta' qm |V'|^2 
\]
In the event $|C'|\le 2qm$ (this occurs with probability at least $1-e^{-\Omega(qm)}$, by the Chernoff bound), the right-hand side is larger than $\alpha' |C'||V'|^2$. Thus $\mathcal{G'}$ is not $\alpha'$-extremal with probability at least $1- e^{-\Omega(\delta mq)} -  2(\delta qm)(np)e^{-\Omega(\alpha np)} \ge  1-e^{-\Omega(\alpha' qm )} - \alpha (qm)(np) e^{-\Omega(\alpha' pn)}$. 

\textbf{Case III: There are at least {\boldmath $(1-\delta) m$} graphs in {\boldmath $\mathbb{G}$} that are {\boldmath $\delta$}-extremal and there are less than {\boldmath $\delta m^2$} many {\boldmath $\delta$}-good pairs of {\boldmath $\delta$}-extremal graphs.} 
For every $\delta$-extremal graph $G\in \mathbb{G}$, we fix a set $A_G$ such that $G$ is $\delta$-extremal with respect to $A_G$. There must be a graph $G'\in \mathbb{G}$ and a subfamily $\tilde{\mathbb{G}}$ of at least $(1-2\delta)m$ many $\delta$-extremal graphs such that every graph $G\in \tilde{\mathbb{G}}$ satisfies $|A_G\cap A_{G'}|\ge (1/2 - \delta)n$ or $|A_G\cap A_{G'}|\le \delta n$.
In both cases, one can transform the bipartition $A_G,\overline{A_G}$ into $A_{G'},\overline{A_{G'}}$ by swapping at most $\delta n$ pairs of vertices. Each such swap increases $\min\set{e_{G}(A), e_{G}(A,\overline{A})}$ (with respect to the current set $A$) by at most $2n$. Hence, each graph in $\tilde{\mathbb{G}}$ is $3\delta$-extremal with respect to $A_G$. Thus, by setting $A=A_G$, we have 
\begin{align*}
\sum_{G\in \mathbb{G}} \min\set{e_{G}(A), e_{G}(A,\overline{A})} &= \sum_{G\in \mathbb{G}\setminus \tilde{\mathbb{G}}} \min\set{e_{G}(A), e_{G}(A,\overline{A})} + \sum_{G\in \tilde{\mathbb{G}}} \min\set{e_{G}(A), e_{G}(A,\overline{A})} \\
&\le 2\delta m \cdot n^2 + m \cdot 3\delta n^2 \le \alpha m n^2,
\end{align*}
which contradicts the fact that $\mathbb{G}$ is not $\alpha$-extremal.

All the above cases either conclude that the family $\mathbb{G}'$ whp is not $\alpha'$-extremal with probability at least $1-e^{-\Omega(\alpha' qm )} - \alpha qm e^{-\Omega(\alpha'^2 pn)}$  or give a contradiction. This finishes the proof of \Cref{lem:property preserve}.
\end{proof}

%%%%%%%%%%%%%%%%%%%%%%%%%%%%%%%%%%%%%%%%%%%%%%%%%%%%%%%%%%%%%%%%%%%%%%%%%%%%%%%%%%%%%%%%%%%
\section{Concluding remarks}
Using our proof mechanisms, we can establish the following transversal version of the robust Dirac's theorem (i.e., \Cref{thm:robustdirac}), which is also only a slight strengthening of \Cref{thm:main_spread_non_extremal}. Since we believe this result may be useful in the future for dealing with transversal Hamiltonicity problems, we include a brief description on how to adapt the proof of \Cref{thm:main_spread_non_extremal} to prove it.

\begin{theorem}\label{thm:hamilton path 2}
Let $0 < 1/n \ll 1/c, \eta, \zeta \ll \alpha \le 1$ and $m = \zeta n$. Let $M$ be a matching of size at most $m$ on a vertex set $V$ of size $n$. Let $\mathbb{G}$ be a non $\alpha$-extremal family of $n-|M|$ graphs on $V$ with minimum degree at least $(1/2 - \eta) n$. Then, there is a $(c/n^2)$-spread probability measure on the $\mathbb{G}$-transversal sets of edges $F$ such that $F\cup M$ forms a Hamilton cycle on $V$.
\end{theorem}
\begin{proof}[Proof sketch]
The above theorem can be easily proved almost verbatim as \Cref{thm:main_spread_non_extremal} with the following minor alterations. First, note that since $\mathbb{G}$ is not $\alpha$-extremal we have that $\mathbb{G}[V\setminus V(M)]$ is not $(\alpha/2)$-extremal. Thus one can prove \cref{lem:vortex} for elements of $\cS(\alpha/2,\alpha',\beta,\delta,\epsilon, N,\mathbb{G}[V\setminus V(M)])$ with the additional property that $d_{G_c}(v,V_1)\ge (1/2-\epsilon)|V_1|$ for every vertex $v\in V(M)$ and every color $c\in [n-|M|]$ (here the hierarchy is as follows: $0<1/n \ll 1/L , \eta, \zeta  \ll \epsilon \ll \delta, \beta \ll  \alpha' \ll \alpha \le 1$). Note that for such elements the graph family $\mathbb{G}_{C_1}[V_1]$ is not $\alpha'$-extremal, and  has minimum degree at least $(1/2 - \epsilon) n$, thus $\mathbb{G}_{C_1}[V_1\cup V(M)]$ is not $\alpha'/2$-extremal, and has minimum degree at least $(1/2 - 2\epsilon) n$. The second alteration is to set $\cP_0=M_{abs}\cup M$ at the first application of \Cref{lem:cover_down} in the proof of \Cref{thm:main_spread_non_extremal}. 
\end{proof}

\section*{Acknowledgement}
We are grateful to an anonymous referee for suggestions that greatly improved our exposition. We also thank Alan Frieze and Dongyeap Kang for pointing out a couple of minor errors in the previous version.

\printbibliography
\end{document}